\documentclass[12pt]{amsart}
\PassOptionsToPackage{final}{graphicx}
\PassOptionsToPackage{final,colorlinks,citecolor=black,linkcolor=black,
  urlcolor=black,pdftitle={Asymptotic cones of snowflake groups and
    the strong shortcut property},pdfauthor={Christopher H. Cashen and
  Nima Hoda and Daniel J. Woodhouse},pdfkeywords={snowflake groups,
  asymptotic cones, shortcut groups}}{hyperref}
\usepackage{ifdraft}
\usepackage[margin=1.15in]{geometry}
\usepackage{subcaption}
\usepackage{graphicx}
\usepackage{verbatim}
\usepackage{mathrsfs}
\usepackage{latexsym}
\usepackage{pinlabel}
\usepackage{microtype}
 \usepackage{hyperref}
\usepackage{tikz}
\usetikzlibrary{arrows}
\usetikzlibrary{decorations.markings}
\usetikzlibrary{decorations.pathmorphing}
\usepackage{tikz-cd} 
\tikzset{labl/.style={anchor=south, rotate=90, inner sep=.5mm},
  vertex/.style={circle,minimum size=0.15cm,inner sep=0,fill=black},
  thickeraser/.style={line width=2.4pt, white}
}

\title[Cones, snowflakes, shortcuts]{Asymptotic cones of snowflake groups and the strong shortcut property}
\author[Cashen]{Christopher H.\ Cashen}
\address{Faculty of Mathematics\\University of
  Vienna\\Oskar-Morgenstern-Platz 1\\1090 Vienna, Austria\\
\href{https://orcid.org/0000-0002-6340-469X}{\includegraphics[scale=.75]{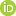}
  0000-0002-6340-469X}}
\email{christopher.cashen@univie.ac.at}
\author[Hoda]{Nima Hoda}
\address{DMA, École normale supérieure \\ Université
  PSL, CNRS \\ 75005 Paris, France}
\email{nima.hoda@mail.mcgill.ca}
\author[Woodhouse]{Daniel J.\ Woodhouse}
\address{Mathematical Institute\\ University of Oxford\\ Andrew Wiles Building\\ Radcliffe Observatory Quarter (550)\\ Woodstock Road, Oxford, OX2 6GG}
\email{daniel.woodhouse@mail.mcgill.ca}

\thanks{C.\ Cashen is supported by the Austrian Science Fund (FWF):
  P30487-N35 and P34214-N and visited the University of Arkansas
  during part of this work.  N.\ Hoda
  is supported by the ERC grant GroIsRan.}

\makeatletter
\@namedef{subjclassname@2020}{\textup{2020} Mathematics Subject Classification}
\makeatother


\allowdisplaybreaks 

\numberwithin{equation}{subsection}

\newtheorem{thm}[equation]{Theorem}
\newtheorem{corollary}[equation]{Corollary}
\newtheorem{lemma}[equation]{Lemma}
\newtheorem{prop}[equation]{Proposition}
\newtheorem{proposition}[equation]{Proposition}

\newtheorem*{theorem*}{Theorem}
\theoremstyle{definition}
\newtheorem{defn}[equation]{Definition}
\newtheorem{definition}[equation]{Definition}

\newtheorem{remark}[equation]{Remark}

\newtheorem{claim}[equation]{Claim}
\newtheorem{fact}[equation]{Fact}

\newcommand{\Z}{\mathbb{Z}}
\newcommand{\R}{\mathbb{R}}

\newcommand{\CAT}{\operatorname{CAT}}

\DeclareMathOperator{\Cone}{Cone}
\DeclareMathOperator{\Cay}{Cay}
\DeclareMathOperator{\sign}{sign}
\DeclareMathOperator{\cl}{cl}
\DeclareMathOperator{\SL}{SL}

\definecolor{amethyst}{rgb}{0.6, 0.4, 0.8}

\newcommand{\hide}[1]{}

\newcommand\restr[2]{{
  \left.\kern-\nulldelimiterspace 
  #1 
  \vphantom{\big|} 
  \right|_{#2} 
  }}
\newcommand{\from}{\colon\thinspace}

\newenvironment{subproof}[1][\proofname]{%
  \begin{proof}[#1]%
}{%
  \end{proof}%
}

\begin{document}

\begin{abstract}
We exhibit an infinite family of snowflake groups all of whose
asymptotic cones are simply connected. Our groups have neither
polynomial growth nor quadratic Dehn function, the two usual sources
of this phenomenon. We further show that each of our groups has an
asymptotic cone containing an isometrically embedded circle or,
equivalently, has a Cayley graph that is not strongly shortcut. These
are the first examples of groups whose asymptotic cones contain
`metrically nontrivial' loops but no topologically nontrivial ones.

\end{abstract}

\maketitle

\setcounter{tocdepth}{1}
\tableofcontents

\section{Introduction}
The strong shortcut property was introduced by the second named author
in order to explore commonalities between the many theories of
nonpositively curved groups which have been developed in recent
decades \cite{Hoda:shortcut_graphs:2022}.  A graph $\Gamma$ is
\emph{strongly shortcut} if, for some $K > 1$, there is a bound on the
lengths of the $K$-biLipschitz cycles of $\Gamma$.  This property has
a natural generalization to rough geodesic metric spaces
\cite{Hodstrongshortcut}.  A group $G$ is \emph{strongly shortcut} if
it satisfies one of the following three equivalent conditions
\cite{Hodstrongshortcut}:
\begin{enumerate}
\item $G$ has a strongly shortcut Cayley graph.
\item $G$ acts properly and cocompactly on a strongly shorcut graph.
\item $G$ acts metrically properly and coboundedly on a strongly
  shortcut rough geodesic metric space.
\end{enumerate}
Strongly shortcut groups are finitely presented and have polynomial
isoperimetric functions \cite{Hoda:shortcut_graphs:2022}.

The class of strongly shortcut spaces is vast.  It includes:
\begin{itemize}
\item Asymptotically $\CAT(0)$ spaces \cite{Kar:2011,
    Hodstrongshortcut}, in particular: $\CAT(0)$ spaces
  \cite{Gromov:1987, Bridson:1999}, Gromov-hyperbolic spaces
  \cite{Gromov:1987} and $\widetilde{\SL(2,\R)}$ with the Sasaki
  metric \cite{Kar:2011}.
\item $1$-skeletons of systolic \cite{Soltan:1983, Chepoi:2000,
    Haglund:2003, Januszkiewicz:2006}, quadric \cite{Bandelt:1988,
    Hod20} and finite dimensional $\CAT(0)$ cubical \cite{Avann:1961,
    Nebesky:1971, Gromov:1987} complexes
  \cite{Hoda:shortcut_graphs:2022}.
\item Standard Cayley graphs of Coxeter groups
  \cite{Hoda:shortcut_graphs:2022, Niblo:2003}.
\item Coarsely injective spaces of uniformly bounded geometry
  \cite{Haettel:hhshelly}, in particular: Helly graphs
  \cite{Chalopin:Helly_groups:2020} of bounded degree.
\item Heisenberg groups of all dimensions \cite{Hoda:heisenberg}.
\end{itemize}
In particular, all Thurston geometries except Sol are strongly
shortcut.  The class of strongly shortcut groups is correspondingly
vast, including: $\CAT(0)$ groups, hyperbolic groups, Helly groups,
systolic groups, quadric groups, the discrete Heisenberg groups and
hierarchically hyperbolic groups \cite{behrstock:2017:hhs1,
  berhstock:2019:hhs2, Haettel:hhshelly}.  In particular, finitely
presented small cancellation groups \cite{Wise:2003,
  Osajda:classyfing_systolic:2018, Hod20, Chalopin:Helly_groups:2020}
and mapping class groups \cite{masur:1999:cc1, masur:2000:cc2} of
surfaces are strongly shortcut.



There is a characterization of strongly shortcut spaces in terms of
asymptotic cones \cite{Hodstrongshortcut} (see
Section~\ref{sec:asymptoticcones} for definitions): a space is
strongly shortcut if and only if none of its asymptotic cones contains
an isometrically embedded copy of a unit length circle with its length
metric.  One might wonder whether, for Cayley graphs of groups, this
condition is equivalent to all asymptotic cones being simply
connected.  We provide the first counterexamples.

We consider an infinite family of \emph{snowflake groups}:
\[ G_L:=\langle a,x,y,s,t\mid sas^{-1}=x, tat^{-1}=y, a^L=xy,
  [a,x]=[x,y]=[y,a]=1\rangle\]
where $L\geq 6$ is an even integer.
These are special cases\footnote{Specifically, $G_{p,1}$ of Brady and
  Bridson is $G_{2p}$ in our notation, which is why the parameter $L$
  is even. The restriction to $L\geq 6$ is explained in Remark~\ref{whyLsix}.} of the groups introduced by Brady and Bridson
\cite{BraBri00}, who show that the Dehn function for such a group is
$n^{2\alpha}$ for $\alpha:=\log_2L$.

Our main results are:
\begin{theorem*}[{Theorem~\ref{maintheorem}}]
  For all even $L\geq 6$, every asymptotic cone of $G_L$ is simply connected.
\end{theorem*}
\begin{theorem*}[{Corollary~\ref{cor:isometriccircle}}]
  For all even $L\geq 6$, some asymptotic cone of $G_L$ contains an
  isometrically embedded circle.
\end{theorem*}
We find an isometrically embedded circle in an asymptotic cone of the
Cayley graph of $G_L$ corresponding to
the generating set $\{a,s,t\}$.
This means that this particular Cayley graph is not a
shortcut graph.
We do not know if $G_L$ is a strongly shortcut group---it is possible
that it admits an action on some other graph that is strongly
shortcut.
If $G_L$ is a strongly shortcut group then this is the first example
of a group for which the strong shortcut property depends on the
choice of generating set.
If $G_L$ is not a strongly shortcut group then this is the first
example of a group for which every asymptotic cone being simply connected does not imply the
strong shortcut property.

The structures of asymptotic cones of snowflake groups are interesting
independent of the strong shortcut property.
An application of having all cones simply connected is due to Riley
\cite[Theorem~C]{Ril03}: the isodiametric and filling length functions
of $G_L$ are both linear.
It is also interesting just to have more diverse examples of groups
with all cones simply connected.
As far as we are aware, the only existing examples are groups of polynomial growth, 
groups with at most quadratic Dehn function, and certain
combinations of these, see Section~\ref{sec:asymptoticcones}.
Our examples are not in any of these classes. 
We also do not know any other examples of groups having all of their
asymptotic cones simply connected, but with at least some of them
containing isometrically embedded circles.
Conjecturally there are no such groups of polynomial growth.
It is an open question whether there exist such groups with
quadratic Dehn function.

\subsubsection*{The idea of the proof that asymptotic cones are simply connected }
There is a criterion, the ``Loop Subdivision Property'', which we recall in Theorem~\ref{rileycriteria}, for a group to have all of its asymptotic cones
simply connected.
This says that every asymptotic cone is simply connected if and only
if every word in the generators representing the trivial element of
the group can be filled by a van Kampen diagram of uniformly bounded
area, subject to the constraint that the boundary lengths of 2--cells of the
diagram are bounded above by a fixed fraction of the length of the
word to be filled.

To construct such diagrams, we establish a fat/thin dichotomy.
The ``fat'' words are those such that the corresponding loop in the Cayley
graph is a $K$--biLipschitz loop for $K$ close to 1.
Thin words we cut into a bounded number of small pieces and at most
one fat one using the ``Circle Tightening Lemma'' of
\cite{Hodstrongshortcut}.
This reduces the problem to constructing diagrams for $K$--biLipschitz
loops for $K$ arbitrarily close to 1.
This reduction step occurs in Section~\ref{sec:circletightening}.

For $K=1$, i.e., geodesic loops, there are some natural example words to
consider: the ``snowflake words'' that give the groups their names.
We describe how to fill loops made from snowflake words in Section~\ref{sec:idea:snowflake}.

In the general case our thesis is that, for $K$ sufficiently close to
1, all $K$--biLipschitz loops ``ought to look like'' a snowflake loop,
and admit a filling diagram using a similar construction.  In
retrospect, this is more a declaration of na\"ivet\'e than adroitness.
It turns out to be true, but it was much harder than we had initially
expected to make rigorous.

\subsubsection*{Why is it hard?}
Consider these two facts:
\begin{fact}[Reverse H{\" o}lder inequality]
  \label{fact:reverse_holder}
  For any $\alpha > 1$ and any
    $r_1, r_2, \ldots, r_n \ge 0$, the following inequalities hold.
  \[
    \Bigl(\frac{1}{n}\Bigr)^{1-1/\alpha}(r_1^{1/\alpha}
    + \cdots + r_n^{1/\alpha}) \le (r_1 + \cdots +
    r_n)^{1/\alpha} \le r_1^{1/\alpha} + \cdots +
    r_n^{1/\alpha} \]
\end{fact}

\begin{fact}[Power-law distortion, cf \cite{BraBri00} or
  Theorem~\ref{powerdistortion}]\label{fact:powerdistortion}
  Given even $L\geq 6$ and setting $\alpha:=\log_2L$, there exist
$C_{lo}$ and $C_{hi}$ such that for 
$g\in\{a,x,y\}$ and all 
$|m|>0$ we have: 
\begin{equation*}
  C_{lo}\leq\frac{|g^m|}{|m|^{1/\alpha}} < C_{hi}
\end{equation*}  
\end{fact}
We will show in Section~\ref{ageodesics} that
Fact~\ref{fact:powerdistortion} is about as good as can be expected; the
relationship between $|m|$ and $|g^m|$ is not monotone.

Taken together, these two facts suggest that to travel 
between two points of a distorted subgroup like $\langle a\rangle$, it
is more efficient to make one long hop rather than several short ones;
the distortion is non-linear, so the long hop can achieve
proportionally more savings. For example, consider, for $m,n>0$ a path between $1$ and $a^{m+n}$ consisting of a
geodesic from $1$ to $a^m$, followed by a geodesic from $a^m$ to
$a^{m+n}$.
We expect: \[\frac{|a^{m}|+|a^n|}{|a^{m+n}|}\sim\frac{m^{1/\alpha}+n^{1/\alpha}}{(m+n)^{1/\alpha}}=\frac{1+(n/m)^{1/\alpha}}{(1+n/m)^{1/\alpha}}\]
This ratio goes to 1 as $n/m$ goes to 0 or infinity, but equals
$2^{1-1/\alpha}$ when $m=n$.
So the expectation should be that the path is very efficient, in the
sense that the ratio of its length to the distance between its endpoints
is very close to 1, if and only if one of the two subpaths
dominates the other, in the sense that $n/m$ is either very large or
very small.
If this were true it would give us good control over biLipschitz paths
when the biLipschitz constant is close to 1; they would, more or less,
have to be dominated by a long segment that makes good use of the distortion, possibly with some relatively very short additional segments
at the beginning and end.

However, Fact~\ref{fact:powerdistortion} is too rough an estimate to
confirm this intuition. It can only tell us, for instance, that
$\frac{|a^{m}|+|a^m|}{|a^{2m}|}>\frac{C_{lo}}{C_{hi}}\cdot
2^{1-1/\alpha}$, but $\frac{C_{lo}}{C_{hi}}\cdot 2^{1-1/\alpha}$ is
less than 1 when $L\geq 10$ (see
Proposition~\ref{this_whole_paper_is_not_a_waste_of_time}), so this
bound is useless.  The deviation of the distortion function from being
a power-law on the nose overwhelms the efficiency estimates coming
from reverse H\"older.  We must make finer estimates.  The bulk of the
paper is a detailed analysis of biLipschitz paths in snowflake groups.

\subsubsection*{The plan of the paper}
Section~\ref{sec:preliminaries} is background material.
In Section~\ref{sec:snowflakegroups} we analyze geodesics between elements
in the vertex group of the snowflake group.
In Section~\ref{sec:notstrongshortcut} we show that snowflake loops
are geodesic loops, and conclude that our Cayley graph is not strongly
shortcut.

In Section~\ref{sec:idea} we describe how snowflake loops satisfy the
loop subdivision property and give a more detailed explanation of how
we will adapt this construction to the general case. The first step,
in Section~\ref{sec:circletightening}, is to reduce to the case of
biLipschitz loops with biLipschitz constant sufficiently close to
1. We analyze such loops in Section~\ref{sec:bilipschitz}, and
construct fillings of them in Section~\ref{sec:constructfillings}. We
tie things up and make some closing remarks in
Section~\ref{sec:closing}.

\section{Preliminaries}\label{sec:preliminaries}
\subsection{Basics}

If $S$ is a finite set, let $F(S)$ denote the free group generated by
$S$.

If $G$ is a group generated by a finite set $S$, then the \emph{word
  metric} on $G$ with respect to $S$ is defined by taking $|g|_S$ to be the
minimal length of a word $w\in F(S)$ such that $w=_Gg$.
The \emph{Cayley graph} $\Cay(G,S)$ of $G$ with respect to $S$ is the labelled
directed graph whose vertices are in bijection with $G$, and such that
there is an edge from $g$ to $h$ labelled $s$ when $h=gs$.
We say that a word in $F(S)$ labels an edge path, or is read on an
edge path, if it is the concatenation of labels of the edges, with the
convention that an $s$ edge traversed against its orientation
contributes the letter $s^{-1}$.
Given a word $w\in F(S)$ and a basepoint $g\in G$ there is a unique
edge path starting from $g$ labelled by $w$.

Paths in $\Cay(G,S)$ will always start and end at vertices and will always be parameterized by arclength.
If $\gamma$ is a path, $\bar\gamma$ is the path with opposite
parameterization.
We use $+$ to denote concatenation.
We frequently conflate paths with their labels, but there will be only one reasonable interpretation.
For example, if $\gamma$ is a path in
$\Cay(G,S)$ from $g$ to $h$, with $g,h\in G$, $\delta$ is another path
in $\Cay(G,S)$ whose endpoints differ by $k\in G$, and $s,t\in S$ then
$s.\gamma+\bar\delta+st^{-1}$ is a path from $sg$ to $shk^{-1}st^{-1}$ while
$s+\gamma+\bar\delta+st^{-1}$ is a path from 1 to $sg^{-1}hk^{-1}st^{-1}$.

Let $\langle S\mid\mathcal{R}\rangle$ be a finite presentation of a
group $G$.
If $w=_G1$ then $w$ can be written in $F(S)$ as 
$w=\prod_{i=1}^{n}u_ir_iu_i^{-1}$ for some $u_i$ and $r_i$ such
that either $r_i$ or $r_i^{-1}$ belongs to $\mathcal{R}$.
Define $\mathrm{Area}(w)$ to be the minimal such $n$. 
The \emph{Dehn function} or \emph{isomperimetric function} of the presentation is
the function $n\mapsto \max_{|w|_S\leq n,\,w=_G1} \mathrm{Area}(w)$.
While the precise function depends on the choice of presentation, it
turns out that up to a certain notion of equivalence of functions all
of the finite presentations yield equivalent functions, so we define
the \emph{Dehn function of the group} to be this equivalence class.
The equivalence relation preserves, in particular, degrees of
polynomial growth, so it makes sense to say that a group has a Dehn
function of a particular polynomial degree.

\subsection{van Kampen diagrams}
\begin{definition} Let $\mathcal{A}$ be an alphabet.
  A \emph{van Kampen diagram} for a word $w\in\mathcal{A}^*$ over a subset
  $\mathcal{R}\subset\mathcal{A}^*$ of cyclically reduced words is a compact, connected, simply connected 2-dimensional cell
  complex $D$ embeded in the plane
  such that edges are directed and labelled by elements of
  $\mathcal{A}$ in such a way that starting from some base boundary vertex
  the word read around $\partial D$ is $w$ and the word read around the
  boundary of each 2--cell is freely and cyclically reduced and, up to cyclic
  permutation and inversion, belongs to $\mathcal{R}$.

  A 2-cell of $D$ is also called a \emph{region}.

  A van Kampen diagram is \emph{reduced} if there does not exist a
  pair of 2-cells sharing an edge such that the boundary labels
  starting from a shared edge are inverses.

  The \emph{area} of a van Kampen diagram is the number of 2--cells. 

  A van Kampen diagram is \emph{minimal} or has \emph{minimal area} if it has the fewest
  number of 2-cells among all van Kampen diagrams for the given
  boundary word. Minimal diagrams are necessarily reduced. 

  A \emph{disk diagram} is a van Kampen diagram that is topologically
  a disk.

  An \emph{arc} is a maximal segment in the 1--skeleton of $D$ such
  that every interior vertex has valence 2.
  For the purpose of drawing pictures it is often convenient to
  replace each arc by a single directed edge labelled by the word in
  $\mathcal{A}^*$ read along the arc.
  In this case the length of the new edge is defined to be the length
  of its label, that is, the length of the original arc.
  This is called \emph{forgetting the valence 2 vertices}.

  The \emph{mesh} of a diagram is the maximum length of the boundary
  of a 2--cell.
\end{definition}

Given a specified base vertex in $D$, there is a unique label
and orientation preserving map $D^{(1)}\to \Cay(G,\mathcal{A})$
sending the base vertex to 1, where
$G=\langle \mathcal{A}\mid\mathcal{R}\rangle$. 
Thus, we sometimes say that $D$ gives a \emph{filling} for the word
read around $\partial D$, which is a word $w\in F(\mathcal{A})$ such
that $w=_G 1$.
\medskip

The groups we are interested in are double HNN extensions over
$\mathbb{Z}$; that is, groups $G$ having a presentation of the form $\langle
a_1,\dots,a_k,s,t\mid \mathcal{R},
sus^{-1}v^{-1},twt^{-1}x^{-1}\rangle$, where
$H:=\langle
a_1,\dots,a_k\mid \mathcal{R}\rangle$ is the \emph{vertex group}, $s$
and $t$ are called \emph{stable letters}, and $u$, $v$, $w$, and $x$
are infinite order elements of $H$. 
It is a standard result
of Bass-Serre theory that $H$ is a subgroup of the resulting group
$G$.

Let $\mathcal{R}':=\mathcal{R}\cup\{ sus^{-1}v^{-1},twt^{-1}x^{-1}\}$.
Since each stable letter appears in only one relator, and appears as
an inverse pair, we can say more about the structure of reduced diagrams over $\mathcal{R}'$.
Suppose $D$ is a reduced diagram and some region has an $s$--edge in its
boundary.
Either this edge belongs to $\partial D$, or there is an
adjacent region that shares the $s$--edge. However, there are only two
possible occurrences of $s^\pm$ in $\mathcal{R}$, and one of the
possible choices would create an unreduced diagram, so there is a
unique way that the  boundary of the adjacent region can be labelled.
That region has one more edge labelled with $s$ or $s^{-1}$, which is
either on $\partial D$ or else is adjacent to another region whose
boundary label is determined, etc. 
Define an equivalence relation on 1--cells of $D$ labelled $s^\pm$
such that two such edges are equivalent if they belong to a common
2--cell. 
For each equivalence class of $s$--edge, define the
\emph{$s$--corridor} containing it to be the union of cells containing
one of those edges in its closure. 
We define $t$--corridors similarly, and just say \emph{corridor} if we
do not care which of the stable letters the corridor corresponds to. 
It is immediate from the definitions that distinct corridors have
disjoint interiors.

A priori there are three possible types of $s$--corridor: a single $s$--edge, a
closed loop of 2-cells, or a string of 2-cells ending at distinct
edges of $\partial D$. The first type cannot occur if $D$ is a disk
diagram.
This is true in particular when $\partial D$ embeds in $X$.
The second type cannot occur in a minimal diagram because if it did the subdiagram inside
the closed corridor would be a diagram filling a power of $u$, $v$,
$w$, or $x$, but these were assumed to be infinite order elements.
The same argument shows that not only does a corridor have distinct
edges of $\partial D$, these two edges are disjoint. 

Define an \emph{HNN diagram} (with respect to the above presentation
as a multiple HNN extension) as a minimal diagram over the set of
relations:
\[\{g \in F(\{a_i\}_{i=1}^k)\mid g=_G1\}\cup \{su^ns^{-1}v^{-n}\mid n\in
  \mathbb{N}\}\cup\{tw^nt^{-1}x^{-n}\mid n\in\mathbb{N}\}\]

Regions with boundary labelled by an element of the latter two sets correspond to
corridors in a diagram over $\mathcal{R}$, and we again call them
\emph{corridors}.
Regions labelled by element of the first set correspond to vertices in
the tree dual to the corridors in a diagram over
$\mathcal{R}$. We call such a region a \emph{vertex region}.

\subsection{Asymptotic cones}\label{sec:asymptoticcones}

We briefly recall the definition and some history of asymptotic
cones. See \cite{dructu2002quasi, Ril03, brady2007geometry} for more.

\begin{definition}
  Let $(X,d)$ be a metric space, $\omega$ a non-principal ultrafilter,
  $\mathbf{s}=(s_n)\to\infty$ a `scaling sequence' of positive numbers,
  and $\mathbf{o}=(o_n)$ a sequence of `observation points' in $X$.

  Define a pseudometric on sequences $\mathbf{x}=(x_n)$ of points in
  $X$ as follows.
  \[d_{\omega,\mathbf{s}}(\mathbf{x},\mathbf{y}):=\lim_\omega\frac{d(x_n,y_n)}{s_n}\]

  Define the asymptotic cone $\Cone_\omega(X, \mathbf{o}, \mathbf{s})$ of $X$ with respect to $\omega$, $\mathbf{s}$,
  and $\mathbf{o}$ to be the quotient of the space of sequences
  $\mathbf{x}$ at finite $d_{\omega,\mathbf{s}}$--distance from
  $\mathbf{o}$ by identifying sequences at $d_{\omega,\mathbf{s}}$--distance 0 from one another.
\end{definition}
It turns out that asymptotic cones are complete metric spaces.

We are interested in the case that $X$ is the Cayley graph of a group,
in which case $X$ is homogeneous. In this case the choice of
observation points does not matter, and we can take $\mathbf{o}=1$ to be
the constant sequence at the identity vertex.

\medskip

Gromov \cite{Gro81} used a process of passing to a convergent
subsequence of rescaled metric spaces in his proof of the Polynomial
Growth Theorem.
Outside of the polynomial growth case, convergence fails, but van den
Dries and Wilkie \cite{van1984gromov} recognized that this could be
fixed by the imposition of an ultrafilter. 

Pansu \cite{pansu1983croissance} proved nilpotent groups have simply
connected asymptotic cones.  Gromov \cite[$5F_1''$]{Gro93} showed that
if every asymptotic cone of $G$ is simply connected, then $G$ is
finitely presented and has polynomial Dehn function.  Papasoglu
\cite{Pap96} proved the converse for groups with quadratic Dehn
function, but it is not true in general
\cite{bridson1999asymptotic,sapir2002isoperimetric,OlsSap06}, even for
groups with cubic Dehn function.

Drutu and Sapir characterized relatively hyperbolic groups as having  asymptotic cones that are tree-graded with respect to the asymptotic cones of their peripheral subgroups~\cite{DrutuSapir05} and used this notion to give quasi-isometric rigidity results.
Behrstock~\cite{Behrstock06} showed that the asymptotic cone of the mapping class group is tree graded, although it is not relatively hyperbolic.
The tree graded structure of the asymptotic cones of the mapping class group were further studied by Behrstock-Minsky~\cite{BehrstockMinsky08} and Behrstock-Kleiner-Minsky-Mosher~\cite{BehrstockKleinerMinskyMosher12} as part of proving quasi-isometric rigidity.

Kent \cite{kent2014asymptotic} defines a \emph{prairie group} to be
one for which every asymptotic cone is simply connected.
As far as we are aware, the only groups currently known to be prairie
groups are either virtually nilpotent or have quadratic Dehn function
or are built by combining such pieces: direct products of prairie
groups are prairie groups, and groups that are hyperbolic relative to prairie
groups are prairie groups \cite{DrutuSapir05}.

Snowflake groups are not virtually nilpotent, have superquadratic Dehn
functions, do not split as direct products, and are not hyperbolic
relative to any collection of subgroups. 
We show that snowflake groups are prairie groups using a criterion suggested by Gromov 
and explained in detail by Papasoglu \cite[Section~1]{Pap96}.
As expressed by Riley \cite{brady2007geometry} (cf \cite{Ril03}):

\begin{thm}[Loop Subdivision Property {\cite[Theorem~II.4.3.1]{brady2007geometry}}]\label{rileycriteria}
 Let $G$ be a group with finite generating set $S$. 
 Fix a non-principal ultrafilter $\omega$. The following are equivalent.
 \begin{enumerate}
  \item The asymptotic cone $\Cone_\omega(\Cay(G,S), 1, \mathbf{s})$ is simply
    connected for all scaling sequences $\mathbf{s}\to\infty$.
  \item Given $\lambda\in (0,1)$ there exist $N, J \in \mathbb{N}$ such that for all null-homotopic words $w$ there is an equality
  \[
    w = \prod_{i=1}^N u_iw_iu_i^{-1}
  \]
  in the free group $F(S)$ for some words $u_i$ and $w_i$ such that
  the $w_i$ are null homotopic and have length $\ell(w_i) \leq
  \lambda\ell(w) +J$.
  \item  Given $\lambda\in (0,1)$ there exist $A, M \in
  \mathbb{N}$ such that for all words $w\in F(S)$ with $w=_G1$ there is
  a van Kampen diagram for $w$ over $\mathcal{R}=\{g\in F(S)\mid
  g=_G1\}$ with area at most $A$ and mesh at most $\lambda |w|+M$. 
 \end{enumerate}
\end{thm}

\begin{remark}
  Change of finite generating set gives a biLipschitz map between
  Cayley graphs, which in turn induces a biLipschitz homeomorphism
  between asymptotic cones for a common ultrafilter and scaling
  sequence, so if all asymptotic cones of $\Cay(G,S)$ are simply
  connected, then the same is true with respect to any finite
  generating set of $G$.
\end{remark}

\section{Snowflake groups}\label{sec:snowflakegroups}

We consider the following family of groups for $L\geq 6$ and even:
\begin{equation}
  \label{presentation}
  G_L:=\langle a,x,y,s,t\mid sas^{-1}=x, tat^{-1}=y, a^L=xy,
  [a,x]=[x,y]=[y,a]=1\rangle
\end{equation}
Such a group is isomorphic, via $x=a^pb$ and $y=a^pb^{-1}$, to the
Brady-Bridson snowflake group \cite{BraBri00} $G_{p,q}$ with $p=L/2$
and $q=1$ defined by:
\begin{equation}
  \label{bradybridsonsnowflake}
 G_{p,q} := \langle a,b,s,t \mid [a,b] =1, s a^q s^{-1} = a^pb, ta^qt^{-1} = a^pb^{-1} \rangle 
\end{equation}

The group $G=G_L$ is a double HNN extension over
$\mathbb{Z}$ of the subgroup $H:=\langle
a,x,y\rangle\cong\mathbb{Z}^2$.
Note that $\langle a\rangle$, $\langle x\rangle$, and $\langle
y\rangle$ pairwise have trivial intersections since
$H\cong\mathbb{Z}^2$, and that $H=\coprod_{i=0}^{L-1}a^i\langle x,y\rangle$.
Let $X=X_L$ denote the Cayley graph of $G_L$ with respect to the
generating set $\{a,x,y,s,t\}$, with the following modification: we
declare that the $a$, $s$ and $t$ edges have length 1 as usual, but
assign length $L$ to the $x$ and $y$ edges\footnote{For the purposes of this paper one may
  choose any length at least 6 for the length of the $x$ and $y$
  edges, since this is enough to guarantee that no 2--biLipschitz path
  contains such an edge. We choose $L$ because it is the length induced by the
  maximally symmetric metric for this group as in \cite{Cas10}.}.
This follows the convention of \cite{ArzCasTao15} and leads to a
simpler description of geodesics. 
This is not a significant deviation from the standard conventions,
because the induced metric on vertices is still a word metric: it is
the one corresponding to the generating set $\{a,s,t\}$.
Metrically the $x$ and $y$ edges are extraneous.
We include them to preserve the correspondence between the group
structure as an HNN extension and the topological structure of the
universal cover of the
 corresponding graph of spaces, which is a torus
with two annuli glued on.
In this correspondence we have, for instance, that cosets of $\langle
a,x,y\rangle$ correspond to the planes that are the connected components of the preimage of
the torus in the universal cover, so we would like to think of cosets
of $\langle a,x,y\rangle$ as connected subsets of $X$, even if they
are not convex.

Brady and Bridson introduced snowflake groups to furnish examples of groups with
interesting power-law Dehn functions:  the Dehn function of $G_{p,q}$ is $n^{2\alpha}$ where
$\alpha := \log_2(2p/q)$.

For each $n\geq 1$ we inductively construct  a pair of paths
$\sigma_{n,s}$ and $\sigma_{n,t}$ from 1 to $a^{L^n}$ by:
\begin{align*}
  \sigma_{1,s}&:=sas^{-1}tat^{-1}\\
  \sigma_{1,t}&:=tat^{-1}sas^{-1}\\
  \sigma_{n+1,s}&:=s+\sigma_{n,s}+s^{-1}+t+\sigma_{n,s}+t^{-1}\\
  \sigma_{n+1,t}&:=t+\sigma_{n,t}+t^{-1}+s+\sigma_{n,t}+s^{-1}
\end{align*}
We call $\sigma_{n,s}$ and $\sigma_{n,t}$ \emph{snowflake paths} and
call $\sigma_{n,s}+\bar\sigma_{n,t}$ a \emph{snowflake loop}.
The snowflake loops are precisely the loops that Brady and Bridson
showed were hard to fill, witnessing the Dehn function.
An HNN diagram for $\sigma_{7,s}+\bar\sigma_{7,t}$ is illustrated in
Figure~\ref{fig:snowflake}. (The central diamond and all triangular
regions are vertex regions, and the trapezoidal regions are corridors.)

We will show in Theorem~\ref{thm:geodesicloop} that the snowflake
loops are geodesic. In particular, we can then compute lengths:
\begin{equation}
  \label{eq:2}
  |a^{L^n}|=5\cdot2^n-4 \text{ if }n>0
\end{equation}

\begin{figure}[h]
  \centering
  \labellist
  \tiny
  \pinlabel $1$ [r] at 300 580
  \pinlabel $a^{L^7}$ [l] at 800 580
  \pinlabel $x^{L^6}$ [b] at 570 1030
  \pinlabel $y^{L^6}$ [t] at 550 140
  \endlabellist
  \includegraphics[height=5cm]{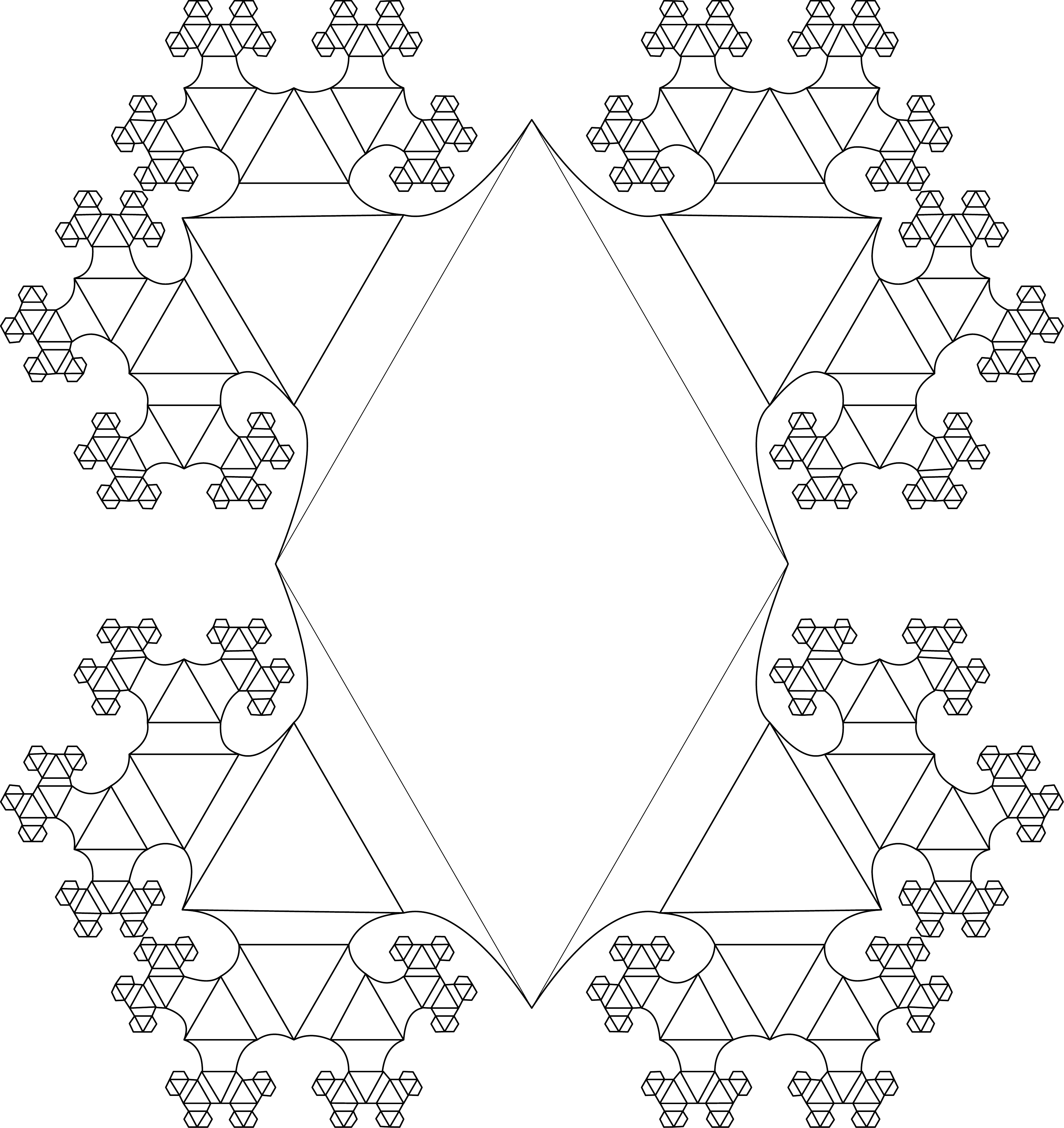}
  \caption{A snowflake for $a^{L^7}$.}
  \label{fig:snowflake}
\end{figure}
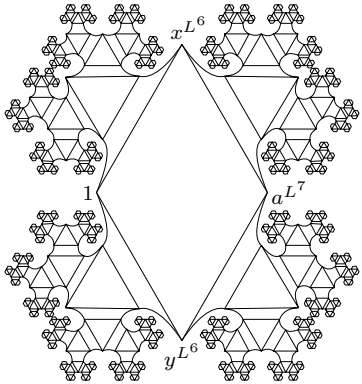

\subsection{A first result about geodesics}\label{first_geodesic_result}
Recall that $H=\langle a,x,y\rangle\cong \mathbb{Z}^2$ is the vertex
group of $G$.
With respect to a fixed coset $gH$ of $H$, an \emph{escape} is a
nontrivial path $\gamma \from P \to X$ that intersects $gH$ precisely
at its endpoints. 
The difference between the endpoints of such a $\gamma$ is a power of
either $a$, $x$, or $y$, and we call $\gamma$ an \emph{$a$--escape}, an
\emph{$x$--escape} or a \emph{$y$--escape}, respectively.
Note that there are two types of $a$--escape: paths of the form
$s^{-1}+\gamma+s$ or of the form $t^{-1}+\gamma+t$.
On the other hand, every $x$--escape is of the form $s+\gamma+s^{-1}$, and
every $y$--escape is of the form $t+\gamma+t^{-1}$.

A \emph{toral path}, with respect to $gH$, is an edge path contained
in $gH$.  By our choice of metric on $X$, a toral path whose biLipschitz
constant is less than $2$ does not contain any $x$ or $y$ edges, so it
is just an $a$--path.

If $\gamma$ is a single escape, then its endpoints differ by $g^n$ for
some $g\in\{a,x,y\}$ and some $n\in\mathbb{Z}$.
Define the \emph{trace} of $\gamma$ to be the toral path between the
endpoints of $\gamma$ consisting of $n$--many $g$--edges.
Define the \emph{exponent} of $\gamma$ to be $n$.

Given a path $\gamma$ with endpoints on $gH$
we can decompose it into a concatenation
$\gamma=\gamma_1+\cdots+\gamma_n$ of escapes and toral
subpaths.
We then define the trace of $\gamma$ (in $gH$) by replacing each escape by its
trace. 

Since $H$ is Abelian, for any permutation $\sigma$ of $\{1,\dots,n\}$ there are
translations $g_i\in H$ that give a new path $\gamma'=g_1.\gamma_{\sigma(1)}+\cdots+g_n.\gamma_{\sigma(n)}$ with
the same length and the same endpoints as $\gamma$. 
\emph{Consolidating toral subpaths, $x$--escapes, etc} means choosing
the permutation $\sigma$ so that in $\gamma'$ all of the toral
subpaths, $x$--escapes, etc are adjacent.
Thus, by consolidating toral subpaths we may replace $\gamma$ by a
path $\gamma'$ with the same endpoints that is a concatenation of escapes followed by a single toral subpath, for
instance.

Notice that our presentation \eqref{presentation} for $G$ has a
symmetry that exchanges $x$ and $y$ and fixes $\langle a\rangle$.
This induces an isometry on $X$, which we refer to as an
\emph{$x/y$--symmetry}. 

\begin{proposition}\label{3piecegeodesic}
  For all $h\in H=\langle a,x,y\rangle$, every geodesic from 1
  to $h$ consists of a concatenation, in some order,  of at most one $x$--escape, at
  most one $y$--escape, and some number of toral subpaths that are all
  of the form $a^{\ell_i}$, where all of the $\ell_i$ have the same
  sign.
  If $L>6$ or if $L=6$ and $h\notin\{a^{-10},a^{-9},\dots,a^{10}\}$ then $|\sum_i\ell_i|<L$.
  If $L=6$ and $1\neq h\in\{a^{-10},a^{-9},\dots,a^{10}\}$ 
  then the $a$--path from 1 to $h$ is a geodesic but there also exists a 
  geodesic consisting of at most one $x$--escape, at most one $y$--escape,
  and a toral path that is an $a$--path of exponent strictly less
  than $L$.

  Consequently, there is a geodesic from $1$ to $h$ of the form
  $\gamma_x+\gamma_y+\gamma_a$ (or any permutation of the order of
  these three paths) where $\gamma_x$ is trivial or a geodesic
  $x$--escape, $\gamma_y$ is trivial or a geodesic $y$--escape, and
  $\gamma_a$ is a geodesic $a$--path of length less than $L$. 

  In the special case of $h=x^my^n\neq 1$ every geodesic from 1 to $h$
  is of the
  form $\gamma_x+\gamma_y$ or $\gamma_y+\gamma_x$ where $\gamma_x$ has
  exponent $m$ and $\gamma_y$ has exponent $n$, except in the case
  that $L=6$ and $m=n=\pm 1$ so that $h=a^{\pm L}$, in which case there is such a geodesic,
  but the $a$--path from $1$ to $h$ is also geodesic.
\end{proposition}
\begin{remark}\label{whyLsix}
  One sees in Proposition~\ref{3piecegeodesic} that there is a `corner
  case' of low $L$ and small powers of $a$ that requires extra
  attention when describing geodesics. The situation is even worse
  when $L=4$, so we have chosen to omit $G_4$ completely rather
  than deal with additional special cases. We suspect that the main
  results of the paper are still true for $G_4$. The group $G_2$ has
  quadratic Dehn function, so we already know its asymptotic cones are
  simply connected by Papasoglu's theorem. 
\end{remark}
\begin{proof}
  A geodesic $\gamma$ from $1$ to $h$ can be subdivided into toral subpaths and
  escapes, each of which we may assume is geodesic, since otherwise we
  would find a shortening of $\gamma$.

  If there is more than one $x$--escape then by consolidating
  $x$--escapes we can replace $\gamma$ by a path of the same length
  with the same endpoints that has adjacent $x$ escapes.
  However, since $x$ escapes are of the form $s+\gamma'+s^{-1}$, if
  there are two adjacent then we get an $s$--cancellation, so there is a
  shorter path from $1$ to $h$, contradicting the fact that we started
  with a geodesic.
  By the same reasoning, there is at most one $y$--escape.

Suppose  that there exist nontrivial geodesics that are not a
single $x$--escape and whose endpoints differ by a power of $x$.
Consider a shortest such geodesic, $\gamma$.
By what we have already argued, $\gamma$ consists of at most one
$x$--escape, at most one $y$--escape, and paths, either toral path or $a$--escapes, whose endpoints differ
by powers of $a$.
By the `shortest' hypothesis, $\gamma$ does not contain an
$x$--escape, because if it did we could remove it and get a shorter path whose
endpoints still differ by a power of $x$.
Let $x^m$ be the difference between the endpoints of $\gamma$, and $y^n$
be the difference between endpoints of the $y$--escape.
Since the remainder of $\gamma$ is paths whose endpoints differ by
powers of $a$, we get that $x^m=a^\ell y^n$.
Now, $x^my^m=a^{mL}$, so $y^{m+n}=a^{mL-\ell}$.
Since $\langle a\rangle\cap\langle y\rangle=1$, this is only possible when $m+n=mL-\ell=0$.
However, this gives us:
\[|x^m|=|\gamma|=|y^m|+|a^{mL}|\]
By $x/y$--symmetry, $|x^m|=|y^m|$, which implies $mL=0$, a
contradiction.
Thus, every geodesic whose endpoints differ by a power of $x$ is a
single $x$--escape.
The corresponding statement for $y$ follows from $x/y$--symmetry. 

From this it follows that there are no geodesic $a$--escapes, for if there were we would have a
geodesic $\gamma$ such that, up to $x/y$--symmetry,
$\gamma=s^{-1}+\gamma'+s$, where $\gamma'$ is a geodesic whose
endpoints differ by a power of $x$.
But we have just shown that such a geodesic is a single $x$--escape,
so $\gamma'=s+\gamma''+s^{-1}$, which contradicts that $\gamma$ is
geodesic.

By consolidating toral subpaths we may replace $\gamma$ by a path
with the same length and the same endpoints in which all of the
toral subpaths are adjacent.
If any of them have different signs then there is an $a$--cancellation,
which contradicts $\gamma$ being a shortest path from $1$ to $h$.
Thus, we may suppose $\gamma$ has a single toral subpath of the form
$a^\ell$ for $|\sum_i\ell_i|=\ell\geq 0$.
The toral subpath must be geodesic, since otherwise we could shorten $\gamma$.
We have $|a^L|=|xy|=|sas^{-1}tat^{-1}|\leq 6$, so if $L>6$ then no
$a$--path of exponent at least $L$ is geodesic.
If $L=6$ then a similar argument shows $|a^{12}|\leq 8$, so no
$a$--path of exponent at least $11$ is geodesic.
Furthermore, if $\ell\geq 6$ then we can replace the toral subpath by
the path $sas^{-1}tat^{-1}a^{\ell-6}$ to get a geodesic $\gamma'$ with the
same endpoints as $\gamma$.
If $h\notin\{a^{-10},a^{-9},\dots,a^{10}\}$ then $\gamma$ could not
have been purely toral, because a geodesic toral path  has endpoints that
differ by a power of $a$ with exponent less
than 11 when $L=6$.
Since $a$--escapes are not geodesic, $\gamma$
must have  already had at least one $x$--escape or
$y$--escape.
Since $\gamma'$ has one more $x$--escape
and one more $y$--escape than $\gamma$, it has at least two of some
flavor of escape, which contradicts the fact that $\gamma'$ is
geodesic.
This leaves for our consideration the case that $h=a^\ell$ for  $0<\ell\leq
10$.
Suppose there is a geodesic $\gamma$ from $1$ to $h$ that is not toral.
Then it contains either an $x$ or $y$--escape, but since the endpoints
are in $\langle a\rangle$ it must contain both with equal exponents.
The minimum length of an escape is 3, so $|\gamma|\geq 6$, with
equality only if $\ell=6$.
Thus, if $\ell<6$ then the $a$--path from 1 to $a^\ell$, which has
exponent strictly less than $L$, is the unique
geodesic.
If $6\leq \ell\leq 10$ and the escapes of $\gamma$ have minimal length then, up to reordering,
$\gamma=sas^{-1}tat^{-1}a^{\ell-6}$, which has length $\ell$, so is
the same length as the $a$--path.
If the escapes are not minimal length then $\ell\geq |\gamma|\geq 8+|a^{\ell-6p}|$ for some $p\geq 2$.
The only solution, since $0<\ell\leq 10$, is when $\ell=10$ and $p=2$, in which
case the paths labelled $a^{10}$, $sas^{-1}tat^{-1}a^4$, and $sa^2s^{-1}ta^2t^{-1}a^{-2}$
all have the same length.
We have shown that when $L=6$ and $0<\ell\leq 10$ it is not possible
to find a shorter path from $1$ to $a^\ell$ than the $a$--path, so
$|a^\ell|=\ell$.
We have also exhibited paths, either $a^\ell$ when
$\ell<6$ or $sas^{-1}tat^{-1}a^{\ell-6}$ when $6\leq \ell\leq 10$, of
the desired length, hence geodesic, and of the desired form: at most one
$x$--escape, at most one $y$--escape, and a toral subpath of exponent
strictly less than $L$. 

Finally, consider $h=x^my^n$ with $m\neq0\neq n$.
Suppose a geodesic $\gamma$ from 1 to $h$ is not just an
$x$--escape and a $y$--escape. By commutativity we may assume that the
$x$ and $y$--escapes, if any,  come first, and end at a point $g\in\langle x,y\rangle$.
Then $1\neq g^{-1}h\in\langle a\rangle\cap\langle x,y\rangle=\langle
a^L\rangle$.
Thus, the geodesic contains a toral subpath whose exponent is a
nontrivial multiple of $L$.
As we have seen above, this is impossible except when $L=6$ and
$h=a^{\pm 6}$.
In this case there exists a geodesic that is a single $x$--escape and
a single $y$--escape, but there is also one that is a toral path.

Now suppose the $x$--escape has exponent $p$ and the $y$--escape has
exponent $q$. Then $h=x^my^n=x^py^q\implies x^{m-p}=y^{q-n}$, but
$\langle x\rangle\cap\langle y\rangle=1$, so $p=m$ and $q=n$.
\end{proof}

\subsection{Geodesics between elements of $\langle a\rangle$.}\label{ageodesics}
\begin{lemma}\label{guards}
  If $0\leq r<L$ and $q\in\mathbb{Z}$ then there exists a geodesic from 1 to
  $a^{qL+r}$ that consists of a geodesic from 1 to  either $a^{qL}$ or
  $a^{(q+1)L}$, followed by a geodesic $a$--path.
\end{lemma}
\begin{proof}
  By Proposition~\ref{3piecegeodesic}, there is a geodesic from 1 to $a^{qL+r}$
  of the form $\gamma_x+\gamma_y+a^p$ with $|p|<L$.
Suppose the endpoints of  $\gamma_x$ differ by $x^m$ and the endpoints
of $\gamma_y$ differ by $y^n$. 
Then $x^my^n=a^{qL+r-p}$, but since $\langle x,y\rangle\cap\langle
a\rangle=\langle a^L\rangle$, this implies $m=n$ and $qL+r-p=mL$.
Since $r-p=(m-q)L$ and $-L<r-p<2L$, we have $m=q$ or $m=q+1$.
Thus, $\gamma_x+\gamma_y$ ends at either $a^{qL}$ or $a^{(q+1)L}$.
\end{proof}

\begin{corollary}\label{lem:alengths}\hfill
  
  \begin{enumerate}\item $|a^\ell|=\ell$ for $0<\ell\leq 3+L/2$.\label{item:shorty}
    \item $|a^\ell|=6+L-\ell$ for $3+L/2\leq \ell\leq L$.\label{item:secondbasin}
    \item $ |a^{qL}|=4+2|a^q|$ for all $q\neq 0$.\label{item:powers}
    \item $||a^\ell|-|a^{\ell+1}||=1$ for all $\ell\in\mathbb{Z}$.\label{item:adjacent}
    \item $||a^{qL}|-|a^{(q+1)L}||=2$ for all $q>0$.\label{item:adjacentmultiples}
    \end{enumerate}
  \end{corollary}
  \begin{proof}
    If $q\neq 0$ then $a^{qL}=x^qy^q$ and 
    by Proposition~\ref{3piecegeodesic} there is a geodesic from 1 to
$a^{qL}$ consisting of a single $x$--escape and a single $y$--escape,
both of exponent $q$.
A geodesic $x$--escape of exponent $q$ is of the form
$s+\underline{a^q}+s^{-1}$, for some geodesic $\underline{a^q}$ whose
endpoints differ by $a^q$, and similarly for the $y$--escape, so
Item~\eqref{item:powers} follows.

By Lemma~\ref{guards}, if $0<\ell<L$ then either there is a geodesic from 1 to
$a^{\ell}$ that is an $a$--path of length $\ell$, or there is a
geodesic from 1 to $a^\ell$ that consists of a geodesic from 1 to $a^L$
followed by an $a$--path of length $L-\ell$. The length of the latter is
$6+L-\ell$.
When $\ell<3+L/2$ the former is shorter, when $\ell>3+L/2$ the latter is
shorter, and when $\ell=3+L/2$ they have the same length.
This proves Item~\eqref{item:shorty} and Item~\eqref{item:secondbasin}.

    Clearly $||a^\ell|-|a^{\ell+1}||\leq 1$, so to  prove
    Item~\eqref{item:adjacent} we only need to show it is
    not 0. It cannot be 0 because the relators all have even length, so two neighboring vertices cannot be equidistant to a third vertex.

    Item~\eqref{item:adjacentmultiples} follows from Item~\eqref{item:powers} and
    Item~\eqref{item:adjacent}.
  \end{proof}

We can use Lemma~\ref{guards} and its corollary to construct geodesics from $1$ to powers of $a$ with
nonzero exponent  inductively:
Suppose that we know a geodesic from $1$ to $a^{\ell'}$ for all
$0<\ell'<\ell$.
By Corollary~\ref{lem:alengths}, this is true for $\ell\leq L+1$, so we may
assume $\ell\geq L+1$.
Write $\ell=qL+r$ for $0\leq r<L$, so $q>0$.
By the induction hypothesis, we know a geodesic $\gamma_{qL}$ from 1
to $a^{qL}$ and a geodesic $\gamma_{q+1}$ from $1$ to $a^{q+1}$.
By Lemma~\ref{guards}, there is a geodesic from $1$ to $a^\ell$ that consists of a geodesic
through either $a^{qL}$ or $a^{(q+1)L}$, followed by an $a$--path.
By Proposition~\ref{3piecegeodesic} there exists a geodesic from 1 to $a^{(q+1)L}$ consisting of an
$x$--escape whose endpoints differ by $x^{q+1}$ and a $y$--escape
whose endpoints differ by $y^{q+1}$.
The former has the form
$s+\underline{a^{q+1}}+s^{-1}$ for some geodesic $\underline{a^{q+1}}$
whose endpoints differ by $a^{q+1}$, and similarly for the latter.
But since we know one such geodesic, $\gamma_{q+1}$, we know a
geodesic $\gamma_{(q+1)L}:=s+\gamma_{q+1}+s^{-1}+t+\gamma_{q+1}+t^{-1}$
from 1 to $a^{(q+1)L}$.
Therefore, the shorter of  $\gamma_{qL}+a^{r}$ and $\gamma_{(q+1)L}+a^{L-r}$ is
a geodesic from 1 to $a^{qL+r}$.

We wrote a computer program to compute $|a^\ell|$ inductively as
described. The results for $L=10$ are shown in
Figure~\ref{fig:alengths}.
The main takeaway from this experiment is that the function
$\ell\mapsto |a^\ell|$ is not monotonic, and in fact monotonicity
fails at arbitrarily large scales.
This essentially comes down to the fact that $|a^\ell|$ is relatively
shorter when $\ell$ is close to being a power of $L$, and is somewhat
larger if $\ell$ can only be expressed as a combination of several
different powers of $L$.
In the rest of this subsection we make the relationship between $|\ell|$
and $|a^\ell|$ as precise as possible.
\begin{figure}[h]
  \centering
  \includegraphics{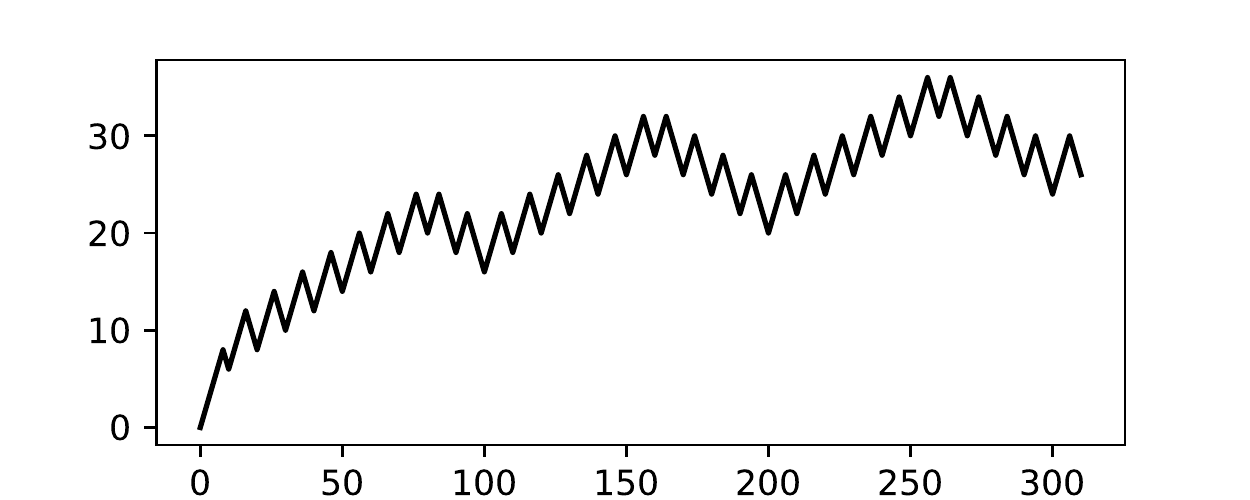}
  \caption{$|a^\ell|$ versus $\ell$ for $L=10$. Local minima occur
    when $L|\ell$, second order minima occur when $L|\ell^2$, etc.}
  \label{fig:alengths}
\end{figure}

Consider an expression $0<m=\sum_{i=0}^j m_iL^i$.
Inductively associate to such an expression a path
$\alpha+\beta+\gamma$ where $\gamma$ is $m_0$--many $a$--edges, either
forward or backward corresponding to the sign of $m_0$, and 
$\alpha=s+\delta+s^{-1}$  and $\beta=t+\delta+t^{-1}$ where $\delta$ is
the path corresponding to $\sum_{i=0}^{j-1} m_{i+1}L^i$.
This path has length:
\begin{equation}
  \label{eq:4}
  \sum_{i=0}^j|m_i|\cdot 2^i+4(2^j-1)
\end{equation}
Call $m=\sum_{i=0}^j m_iL^i$ a \emph{geodesic expression} of $m$ if
the corresponding path is a geodesic.
By our inductive argument, every integer admits a geodesic expression.
We derive constraints on the coefficients:
\begin{lemma} \label{lem:a_geodesic}\hfill
  \begin{itemize}
  \item For $0\leq \ell\leq 2+L/2$ there is a geodesic from $1$ to
    $a^\ell$ consisting of an $a$--path of length at most $2+L/2$.
    \item For $\ell>2+L/2$ there exists a geodesic from $1$ to
      $a^\ell$ consisting of one $x$--escape, one $y$--escape, and an
      $a$--path of length at most $L/2$.
  \end{itemize}
\end{lemma}
\begin{proof}
  The first item and the second item for $L|\ell$ follow immediately
  from Corollary~\ref{lem:alengths} Items~\eqref{item:shorty} and
  \eqref{item:powers}, respectively.
  The second item for $3+L/2\leq \ell\leq L$ follows from the proof of Corollary~\ref{lem:alengths}\eqref{item:secondbasin}.

  Now consider $a^\ell$ for  $\ell\in (qL,(q+1)L)$ with $q\geq 1$.
  By Lemma~\ref{guards}, there is a geodesic from 1 to $a^\ell$
  passing through either $a^{qL}$ or $a^{(q+1)L}$.
  By Corollary~\ref{lem:alengths}, $||a^{qL}|-|a^{(q+1)L}||=2$.
  So if $|a^{qL}|<|a^{(q+1)L}|$ then for $\ell=qL+r$ with $r\leq L/2$ there is a
  geodesic from 1 to $a^\ell$ passing through $a^{qL}$ and then
  proceeding to $a^\ell$ via an $a$--path of length $r\leq L/2$, while
  for $r\geq 1+L/2$ there is a geodesic from 1 to $a^\ell$ passing
  through $a^{(q+1)L}$ and then proceeding to $a^\ell$ via an
  $a$--path of length $L-r\leq L/2 -1$.
The argument for $|a^{qL}|>|a^{(q+1)L}|$ is similar. 
\end{proof}
 \begin{corollary}
   For all $m>0$ there is a geodesic expression
   $m=\sum_{i=0}^j m_iL^i$ satisfying: \begin{equation}
  \label{eq:1}
  0<m_j\leq L/2+2
  \quad\text{and} \quad |m_i|\leq L/2  \text{ for }i<j
\end{equation}
  \end{corollary}

Note the leading coefficient has to be positive when $m$ is because
$\sum_{i=0}^{j-1}\frac{L}{2}\cdot L^i<L^j$.
The conditions of \eqref{eq:1} are not sufficient to conclude that the corresponding
expression is geodesic, but they give useful general constraints.

\begin{lemma}\label{lem:log}
  If $0<m=\sum_{i=0}^j
    m_iL^i$ is a geodesic expression satisfying \eqref{eq:1} with
    $j\geq 1$ then $|j-\log_Lm|<1$.
  \end{lemma}
  \begin{proof}
    \[\big|\sum_{i=0}^{j-1}m_iL^i\big|\leq
      \sum_{i=0}^{j-1}\frac{L}{2}L^i=\frac{1}{2}\frac{L}{L-1}(L^j-1)<L^j\left(\frac{L}{2(L-1)}\right)\leq\frac{3}{5}L^j\]
    So:
    \[L^{j-1}<\frac{2}{5}L^j<m<L^j\left( L/2+2+3/5\right)<L^{j+1}\qedhere\]
  \end{proof}
  \begin{lemma}\label{lem:alengthboundgeodesicexpression}
      If $0<m=\sum_{i=0}^j
    m_iL^i$ is a geodesic expression satisfying \eqref{eq:1} with
    $j\geq 1$ then $5\cdot 2^j-4\leq |a^m|\leq (L+6)\cdot
2^j-4$.
  \end{lemma}
  \begin{proof}
    By \eqref{eq:4}, $|a^m|=\sum_{i=0}^j|m_i|\cdot 2^i+4(2^j-1)$. The
    lower bound is achieved by taking $m_j=1$ and $m_i=0$ for $0\leq
    i<j$.
    The upper bound is achieved by taking $m_j=L/2+2$ and $m_i=L/2$
    for all $0\leq i<j$.
  \end{proof}
We can now show that
the distortion of the subgroup $\langle a\rangle$ (and therefore also
of the
conjugate subgroups $\langle x\rangle$ and $\langle y\rangle$) is a
power function:

\begin{thm}[Power-law distortion]\label{powerdistortion}
There is a constant $C$ depending on $L$ such that for 
$g\in\{a,x,y\}$, 
$|m|>0$, and $\alpha:=\log_2L$ we have: 
\begin{equation*}
  1\leq\frac{|g^m|}{|m|^{1/\alpha}} < C 
\end{equation*}  
Moreover, the first inequality is strict except in the case $g=a$ and $|m|=1$.
\end{thm}
\begin{proof}
  Suppose $g=a$.
  By Corollary~\ref{lem:alengths}\eqref{item:shorty}, if $0<m<L/2+3$ then $|a^m|=m\geq m^{1/\alpha}\geq 1$, with equality
  only if $m=1$.
  If $L/2+3\leq m < L^{3/2}$ then by Lemma~\ref{lem:a_geodesic},
  there is a geodesic from 1 to $a^m$ that contains an $x$--escape and a
  $y$--escape, so has length at least 6.
  However, $m^{1/\alpha}< (L^{3/2})^{1/\alpha}=2^{3/2}<3$, so
  $|a^m|/m^{1/\alpha}>2$.
  In the other direction, $|a^m|/m^{1/\alpha}\leq m^{1-1/\alpha}$, so 
  the upper bound for $0< m<L^{3/2}$ is satisfied if $C\geq (L^{3/2})^{1-1/\alpha}=(L/2)^{3/2}$.

  Consider the case  $m\geq L^{3/2}$.
  There is a geodesic expression
$m=\sum_{i=0}^jm_iL^i$ for some
$j\geq 1$ satisfying \eqref{eq:1}.
By Lemma~\ref{lem:log}, $L^{j-1}<m<L^{j+1}$, so $2^{j-1}<m^{1/\alpha}<2^{j+1}$.

Lemma~\ref{lem:alengthboundgeodesicexpression} gives:
\[m^{1/\alpha}+\frac{3}{2}m^{1/\alpha}-4=\frac{5}{2}m^{1/\alpha}-4<5\cdot 2^j-4\leq|a^m|\leq (L+6)\cdot 2^j-4<2(L+6)m^{1/\alpha}\]
But $m\geq L^{3/2}$ implies $\frac{3}{2}m^{1/\alpha}-4>0$, so 
\[m^{1/\alpha}<|a^m|<2(L+6)m^{1/\alpha}\]

If $g\in\{x,y\}$ and $m\neq 0$ then $|g^m|=2+|a^m|$ and the result
follows from the result for $a$ by taking
$C:=2+\max\{2(L+6),\,(L/2)^{3/2}\}$.
\end{proof}

In light of the claim in the introduction that the constant of
Theorem~\ref{powerdistortion} overwhelms the constant of the reverse
H\"older inequality, and that this necessitates the labor
of this paper, one might wonder whether
this phenomenon is real or if it is an artifact of a suboptimal
statement, or proof, of Theorem~\ref{powerdistortion}.
Proposition~\ref{this_whole_paper_is_not_a_waste_of_time} shows that the
phenomenon is real; even if we pass to asymptotic bounds, the gap
between $C_{lo}$ and $C_{hi}$ grows with $L$. 

\begin{proposition}\label{this_whole_paper_is_not_a_waste_of_time}
  For even $L\geq 6$ consider $a\in G_L$; let $\alpha(L):=\log_2L$,
  $C_{lo}(L):=\liminf_{m\to\infty}\frac{|a^m|}{m^{1/\alpha(L)}}$, and
  $C_{hi}(L):=\limsup_{m\to\infty}\frac{|a^m|}{m^{1/\alpha(L)}}$.
  Then $\frac{C_{hi}(L)}{C_{lo}(L)}>(L+6)/10$.

  When $L\geq 10$, $\frac{C_{lo}(L)}{C_{hi}(L)}\cdot 2^{1-1/\alpha(L)}<1$.
\end{proposition}
\begin{proof}
For fixed $L\geq 6$ let $\alpha:=\alpha(L)$ and $M:=(L-2)/2$.
For $n\geq 0$ let $m_n:=ML^n-M\sum_{i=0}^{n-1}L^i$.
By Corollary~\ref{lem:alengths}\eqref{item:shorty}, $|a^{m_0}|=m_0=M$.
Notice that $m_{n+1}=Lm_n-M$.
It follows from Lemma~\ref{guards} and Corollary~\ref{lem:alengths} parts
\eqref{item:adjacentmultiples} and \eqref{item:powers} that
$|a^{m_{n+1}}|=M+4+2|a^{m_n}|$.
Induction then shows $|a^{m_n}|=(2^{n+1}-1)M+2^{n+2}-4$.
Thus:
\begin{align}
  \frac{|a^{m_n}|}{m_n^{1/\alpha}}&=\frac{(2^{n+1}-1)M+2^{n+2}-4}{\left(M\left(\frac{L^{n+1}-2L^n+1}{L-1}\right)\right)^{1/\alpha}}\notag\\
                                  &=\frac{(2^{n+1}-1)\left(\frac{L-2}{2}\right)+2^{n+2}-4}{\left(\left(\frac{L-2}{2}\right)\left(\frac{L^{n+1}-2L^n+1}{L-1}\right)\right)^{1/\alpha}}\notag\\
                                  &=\frac{2^n\left(\left(1-2^{-n-1}\right)(L-2)+4-2^{2-n}\right)}{L^{n/\alpha}\left(\frac{1}{2}\cdot\frac{L-2}{L-1}\cdot(L-2+L^{-n})\right)^{1/\alpha}}\notag\\
                                  &=\frac{\left(1-2^{-n-1}\right)(L-2)+4-2^{2-n}}{\left(\frac{1}{2}\cdot\frac{L-2}{L-1}\cdot(L-2+L^{-n})\right)^{1/\alpha}}\notag\\
                                  &\stackrel{n\to\infty}{\longrightarrow}\frac{L+2}{\left(\frac{1}{2}\cdot\frac{(L-2)^2}{L-1}\right)^{1/\alpha}}\notag\\
  &=\frac{L+2}{\left(\frac{1}{2}\cdot\frac{(L-2)^2}{L-1}\right)^{\log_L
    2}}\label{eq:42}
\end{align}
As a function of $L$, this result is greater than $L/2+3$ on $[6,\infty)$, so:
\[C_{hi}(L)\geq \lim_{n\to\infty}\frac{|a^{m_n}|}{m_n^{1/\alpha}}>L/2+3\]

In the other direction, induction with
Corollary~\ref{lem:alengths}\eqref{item:powers} gives:
\[C_{lo}(L)\leq
  \lim_{n\to\infty}\frac{|a^{L^n}|}{(L^n)^{1/\alpha}}=\lim_{n\to\infty}\frac{5\cdot
    2^n-4}{2^n}=5\] This proves the first claim. The second claim is
easy to see for $L\geq 14$ using the bounds
$C_{hi}(L)/C_{lo}(L)>(L+6)/10$ and $2^{-1/\alpha}<1$.  It turns out
also to be true for $L\in\{10,12\}$ using \eqref{eq:42} and explicit
computation.
\end{proof}

\begin{remark}\label{rk:fixedconstants}
For the rest of the paper $L$, and consequently, $\alpha$ and $C$ are
fixed, and these names are only used for the constants of Theorem~\ref{powerdistortion}.
\end{remark}

\subsection{Geodesics between elements of $H$}\label{hgeodesics}
Recall that $x$ and $y$ edges in $X$ have length $L$, so geodesics
consist only of $a$, $s$, and $t$ edges. 
\begin{lemma}\label{lem:ag_bound}
  Let $g\in\{x,y\}$. Define $\sign(q)$ to be $-1$, $0$, or $+1$ if $q$ is
  negative, zero, or positive, respectively. 
  \begin{enumerate}
  \item For $m\in\mathbb{N}$, $|g^m|=2+|a^m|$.\label{item:glength}
  \item For $m\in\mathbb{N}$,
    $||g^m|-|g^{m+1}||=1$. \label{item:adjacentg}
    \item If $\ell=qL+r$ with $0\leq |r|<L$ and $m,n\in\mathbb{Z}$ then:
      \[|a^\ell x^my^n|=\min\{|r|+|x^{m+q}|+|y^{n+q}|,\,
        L-|r|+|x^{m+q+\sign(r)}|+|y^{n+q+\sign(r)}|\}\]\label{item:planelength}
    \item   There exists a $K>1$ such 
  that for all $\ell$ and $m$, 
  $K|a^{\ell}g^m| \ge |a^{\ell}| + |g^m|$.\label{item:ag_bound}
  \end{enumerate}
\end{lemma}
\begin{proof}
Item~\eqref{item:glength}
follows immediately from Proposition~\ref{3piecegeodesic}, since a geodesic
from $1$ to $x^m$ is a single $x$--escape.
Item~\eqref{item:adjacentg} follows from this and
Lemma~\ref{lem:alengths}\eqref{item:adjacent}.
For Item~\eqref{item:planelength} the $\ell=0$ case follows
immediately from Proposition~\ref{3piecegeodesic}, while the general case
is the analogue of Lemma~\ref{guards}:
there exists a geodesic from $1$ to $a^\ell x^my^n$ passing though
either $x^{m+q}y^{n+q}$ or $x^{m+q+\sign(r)}y^{n+q+\sign(r)}$, and
then proceeding to $a^\ell x^my^n$ via a geodesic $a$--path.

To prove Item~\eqref{item:ag_bound}, it suffices to consider $\ell\geq
0$ and $g=x$.

If $\ell=0$ or $m=0$ then any $K>1$ will suffice, so assume not.

Applying  Item~\eqref{item:planelength}:
\begin{align*}
  |a^\ell x^m|&=\min\{|r|+|x^{m+q}|+|y^{q}|,\,
                L-|r|+|x^{m+q+\sign(r)}|+|y^{q+\sign(r)}|\}\\
   |a^\ell|+|x^m|&=|x^m|+\min\{|r|+|x^{q}|+|y^{q}|,\,
        L-|r|+|x^{q+\sign(r)}|+|y^{q+\sign(r)}|\} 
\end{align*}
Thus, there are four possibilities for
$K|a^\ell x^m|-(|a^\ell|+|x^m|)$, depending on which terms realize the
two minima.  Rewrite $K|a^\ell x^m|-(|a^\ell|+|x^m|)$ as $T'+T$ where
$T'$ contains the $|r|$ and $L$ terms, and $T$ contains everything
else.
If $r=0$ the minima are clear and $T'=(K-1)|r|=0$. In general for $r\neq 0$: 
\[T'\in\{(K-1)|r|,\,(K+1)|r|-L,\, KL-(K+1)|r|,\,(K-1)(L-|r|)\}\]
All of these are
  non-negative once $K\geq L-1$, so we assume this lower bound on $K$
and then it suffices to find $K$ such that the remaining part $T$ is
non-negative.  The inequality $T\geq 0$ is equivalent to:
\[K\geq \frac{2|x^{q''}|+|x^m|}{|x^{q'}|+|x^{m+q'}|}\]
for some $q',q''\in\{q,q+\sign(r)\}$.
Notice $|q'-q''|\leq 1$.

If $q'=0$ then $|q''|\leq 1$, so $|x^{q''}|\leq 3$, and it suffices to
take $K\geq 7$.

If $0\neq q'$ then by Item~\eqref{item:adjacentg} we have
$|x^{q''}|/|x^{q'}|\leq (|x^{q'}|+1)/|x^{q'}|\leq 4/3$, so:
\[K':=\frac{3|x^{q'}|+|x^m|}{|x^{q'}|+|x^{m+q'}|}>\frac{2|x^{q''}|+|x^m|}{|x^{q'}|+|x^{m+q'}|}\]

Thus, it suffices to take $K\geq K'$, in the
case that $\ell>0$, $m\neq 0$, and
$q'\neq 0$.
We deduce upper bounds on $K'$, hence, sufficient lower bounds on $K$,
according to several subcases.

When $-q'=m$ we have $|x^m|=|x^{q'}|$, so $K'=4$.

If $q'=-1$ and $m\neq 0,1$ then:
\[K'=\frac{9+|x^m|}{3+|x^{m-1}|}\leq \frac{9+|x^m|}{3+|x^m|-1}\leq
  12/5\]

Since $\ell>0$, we have $q'\geq -1$, so it remains to consider cases
where $q'>0$ and $m+q'\neq 0$.
Suppose that $m<-q'<0$.
Then $\frac{|x^{q'}|}{|x^m|}<\frac{C(q')^{1/\alpha}}{|m|^{1/\alpha}}=C\bigl|\frac{q'}{m}\bigr|^{1/\alpha}<C$ and, using reverse H\"older and Theorem~\ref{powerdistortion}: 
\[|x^{m+q'}|>|m+q'|^{1/\alpha}=(-m-q')^{1/\alpha}\geq
  (-m)^{1/\alpha}-(q')^{1/\alpha}>\frac{1}{C}|x^m|-|x^{q'}|\]
Thus:
\[K'<\frac{3|x^{q'}|+|x^m|}{|x^{q'}|+\frac{1}{C}|x^m|-|x^{q'}|}=3C\frac{|x^{q'}|}{|x^m|}+C<C(3C+1)\]

Now suppose $-q'<m<0$, so $|x^{m}|/|x^{q'}|<C$ and:
\[K'=\frac{3+|x^m|/|x^{q'}|}{1+|x^{m+q'}|/|x^{q'}|}<\frac{3+C}{1}\]

Finally, we consider $m>0$. In this case:
\[|x^{m+q'}|>(m+q')^{1/\alpha}>2^{-1+1/\alpha}(m^{1/\alpha}+(q')^{1/\alpha})>C^{-1}\cdot
  2^{-1+1/\alpha}(|x^m|+|x^{q'}|)\]
The corresponding estimate for $K'$ is then:
\begin{align*}
  K'&<\frac{3|x^{q'}|+|x^m|}{|x^{q'}|+C^{-1}\cdot
      2^{-1+1/\alpha}(|x^m|+|x^{q'}|)}\\
  &=\frac{3|x^{q'}|+|x^m|}{(1+C^{-1}\cdot
  2^{-1+1/\alpha})|x^{q'}|+C^{-1}\cdot
    2^{-1+1/\alpha}|x^m|}\\
  &<\frac{3}{(1+C^{-1}\cdot
    2^{-1+1/\alpha})}+\frac{1}{C^{-1}\cdot
    2^{-1+1/\alpha}}\\
    &=C\cdot
    2^{1-1/\alpha}\left(\frac{3}{1+C\cdot
    2^{1-1/\alpha}}+1\right)
\end{align*}
\end{proof}

\begin{corollary}\label{cor:planebound}
  If $\ell=qL+r$ with $0\leq |r|<L$ and $m,n\in\mathbb{Z}$ then the
  following inequality holds.
  \[|a^\ell x^my^n| \ge |m+q|^{1/\alpha} + |n+q|^{1/\alpha} - 5 \]
\end{corollary}
\begin{proof}
  By Lemma~\ref{lem:ag_bound}(\ref{item:planelength}):
  \[|a^\ell x^my^n|=\min\{|r|+|x^{m+q}|+|y^{n+q}|,\,
    L-|r|+|x^{m+q+\sign(r)}|+|y^{n+q+\sign(r)}|\}\]
  By
  Theorem~\ref{powerdistortion},  $|m+q|^{1/\alpha} + |n+q|^{1/\alpha} -
  5\leq |x^{m+q}|+|y^{n+q}|-5$, which is clearly less than the first
  term of the minimum, so it suffices to prove it is also a lower
  bound for  the second.
  But $L - |r| \ge 1$ and, by
  Lemma~\ref{lem:ag_bound}(\ref{item:adjacentg}) and the fact that
  $|x| = |y| = 3$, we have $|g^{k\pm 1}| \ge |g^k| - 3$ for any
  $k \in \mathbb{Z}$ and $g \in \{x,y\}$.
\end{proof}

\subsection{Lines, intersection, projections}
The following lemmas are needed in Section~\ref{sec:constructfillings}.
\begin{lemma}\label{axintersection}
  For every $h\in H-\langle a\rangle$ there
  are exactly two closest points of $\langle a\rangle$. They are the
  unique element of $\langle a\rangle\cap h\langle x\rangle$ and the
  unique element of $\langle a\rangle\cap h\langle y\rangle$. 
  \end{lemma}
  
  \begin{proof}
    The isometric action by $a$ stabilizes $\langle a\rangle$, and
    $H=\coprod_{i=0}^{L-1}a^i\langle x,y\rangle$, so it suffices to
    assume $h=x^my^n\in\langle x,y\rangle-\langle a\rangle$.  Consider
    a closest point $a^p$ of $\langle a\rangle$ to $h$.  By
    Proposition~\ref{3piecegeodesic}, there is a geodesic from $h$ to
    $a^p$ of the form $\gamma_x+\gamma_y+\gamma_a$, but $\gamma_a$ is
    clearly trivial because $\gamma_x+\gamma_y$ ends at a point of
    $\langle a\rangle$.  Suppose $\gamma_x$ and $\gamma_y$ are both
    nontrivial.  By $x/y$--symmetry there is a geodesic $x$--escape
    $\gamma'_x$ with the same exponent as $\gamma_y$ such that
    $\gamma' = \gamma_x+\gamma'_x$ ends on $\langle a\rangle$.  Then
    $|\gamma'|=|\gamma_x|+|\gamma_y|$, but $\gamma'$ consists of
    consecutive $x$--escapes, so can be shortened, contradicting the
    fact that $\gamma_x+\gamma_y$ was a shortest path to
    $\langle a\rangle$.  Thus, either
    $a^p=hx^{n-m}\in H\langle x\rangle$ or
    $a^p=hy^{m-n}\in h\langle y\rangle$, and by $x/y$--symmetry, both
    of these choices are the same distance from $h$.
  \end{proof}
  
\begin{lemma}\label{notallxylinesintersect}
  For all $h\in H$ there exists $\ell\in\mathbb{Z}$
    with $|\ell|\leq L/2$ such that $\langle x\rangle\cap ha^\ell\langle
    y\rangle\neq\emptyset$, in which case the intersection is a unique
    element. 
  \end{lemma}
  \begin{proof}
    It suffices to consider the case $h=a^p$.
    Since $\langle x,y\rangle\cap\langle a\rangle=\langle a^L\rangle$, 
    we have $\langle x\rangle\cap a^{\ell+p}\langle y\rangle\neq \emptyset$
    if and only if $L$ divides $\ell+p$, in which case the intersection
    is the element $x^{(\ell+p)/L}=a^{\ell+p}y^{-(\ell+p)/L}$.
Choose $\ell$ to be an 
    integer of smallest absolute value for which $L|(\ell+p)$.
  \end{proof}

  \begin{lemma}\label{intersectingxylinesprojectnicely}
    For all $n\in\mathbb{Z}$, the unique
    closest point of $y^n\langle x\rangle$ to $1$ is $y^n$.
 \end{lemma}
 \begin{proof}
   The claim is trivial if $n=0$, so suppose not.  Suppose $y^nx^m$ is
   a closest point of $y^n\langle x\rangle$ to $1$.  If $m\neq 0$ then
   by Proposition~\ref{3piecegeodesic} there is a geodesic from 1 to
   $x^my^n$ consisting of a nontrivial $y$--escape of exponent $n$
   followed by a nontrivial $x$--escape.  This is contradictory, since
   such a geodesic contains an interior point in
   $y^n\langle x\rangle$, so $m=0$.
  \end{proof}

  \begin{lemma}\label{projectionatox}
    If $L|p$ then $x^{p/L}$ is the unique closest point of $\langle
    x\rangle$ to $a^p$.
  \end{lemma}
  \begin{proof}
    If $p=0$ the claim is trivial, so suppose not.
    Let $x^m$ be a closest point of $\langle x\rangle$ to $a^p$.
    By Proposition~\ref{3piecegeodesic} there is a geodesic from $a^p$
    to $x^m$ of the form $\gamma_a+\gamma_y+\gamma_x$.

    First, notice $\gamma_x$ is trivial, because otherwise
    $\gamma_a+\gamma_y$ is a shorter path from $a^p$ to $\langle
    x\rangle$.

    Next, observe that if $\gamma_a$ has exponent $r$ then $\gamma_y$ is a $y$--escape with endpoints
    $a^{p+r}$ and $x^m$, so $a^{p+r}y^n=x^m$ for some $n$.
    Since $\langle x,y\rangle\cap\langle a\rangle=\langle a^L\rangle$,
    this implies $L|(p+r)$.
    By hypothesis, $L|p$, so $L|r$, but $|r|<L$, so $r=0$.
This leaves us with $a^p=x^my^{-n}$, which is true if and only if $m=p/L=-n$.
  \end{proof}

\section{The snowflake Cayley graph is not strongly shortcut}\label{sec:notstrongshortcut}
In this section we prove the following theorem.

\begin{thm}\label{thm:geodesicloop}
  The snowflake graph $X_L$ contains arbitrarily large geodesic loops.
\end{thm}
\begin{corollary}\label{cor:notstrongshort}
 The snowflake graph $X_L$ is not shortcut or strongly shortcut.
\end{corollary}
\begin{corollary}\label{cor:isometriccircle}
  There is an asymptotic cone of $X_L$ that contains an
  isometrically embedded unit length circle.
\end{corollary}
\begin{proof}[Proof of Corollary~\ref{cor:isometriccircle}]
  Choose the observation point to be fixed at $1$ and choose the
  scaling sequence to match the lengths of the sequence of geodesic
  loops of Theorem~\ref{thm:geodesicloop}.
\end{proof}

The existence of an asymptotic cone containing an isometrically
embedded unit length circle is in fact equivalent to not being
strongly shortcut \cite[Theorem~A]{Hodstrongshortcut}.

\begin{proof}[Proof of Theorem~\ref{thm:geodesicloop}]
Recall for each $p\geq 1$ we inductively construct snowflake paths:
\begin{align*}
  \sigma_{1,s}&:=sas^{-1}tat^{-1}\\
  \sigma_{1,t}&:=tat^{-1}sas^{-1}\\
  \sigma_{p+1,s}&:=s+\sigma_{p,s}+s^{-1}+t+\sigma_{p,s}+t^{-1}\\
  \sigma_{p+1,t}&:=t+\sigma_{p,t}+t^{-1}+s+\sigma_{p,t}+s^{-1}
\end{align*}
We aim to show snowflake loops are geodesic; that is, for all $p\geq 1$, for every pair of antipodal
 points on the loop $\sigma_{p,s}+\bar\sigma_{p,t}$ the
 two complementary segments of the loop are geodesics.
 It suffices to show that a geodesic between two antipodal points has
 length at least $|\sigma_{p,s}|=|\sigma_{p,t}|=|\sigma_{p,s}+\bar\sigma_{p,t}|/2$.
 Furthermore, we show that every geodesic between a pair of antipodal points on the
 loop passes through either $1$ or
 $a^{L^{p}}$.

 By Corollary~\ref{lem:alengths}\eqref{item:powers} and induction on
 $p$, we have $|a^{L^p}|=5\cdot 2^p-4=|\sigma_{p,s}|=|\sigma_{p,t}|$,
 so the snowflake paths themselves are geodesic.
 
 Next, consider the points $u=x^{L^{p-1}}$ and
 $v=y^{L^{p-1}}$, which are the only pair of antipodal points other
 than $\{1,a^{L^p}\}$ that belong to $H$.
 By Proposition~\ref{3piecegeodesic}, every geodesic $\gamma$ between $u$ and
 $v$ consists of a single geodesic $x$--escape of exponent $L^{p-1}$ and a
 single geodesic $y$--escape of exponent $L^{p-1}$, so $|u^{-1}v|=|x^{L^{p-1}}|+|y^{L^{p-1}}|=|\sigma_{p,s}|$.

 For the general case, by symmetry it suffices to consider antipodal
 points $u,v\notin H$ such that
  $u$ lies
  on the first half of $\sigma_{p,s}$, which is an $x$--escape with
  trace in  $\langle x\rangle$, and $v$ lies on the second half of
  $\sigma_{p,t}$, which is an $x$--escape with trace in 
 $a^{L^p}\langle x\rangle=y^{L^{p-1}}\langle x\rangle$.
 Let $\gamma$ be a geodesic from $u$ to $v$.
 Decompose $\gamma$ as $\gamma=\gamma_1+\gamma_2+\gamma_3$,
 where $\gamma_1$ is the shortest subpath of $\gamma$ connecting $u$
 to $\langle x \rangle$, and $\gamma_3$ is the shortest subpath of
 $\gamma$ connecting $v$ to $y^{L^{p-1}} \langle x \rangle$.
 Note since $u,v\notin H$, $\gamma_1=\gamma_1 ' + s^{-1}$
 and $\gamma_3=s + \gamma_3'$.

 As $\gamma_2$ is a geodesic between elements of $H$, by
 Proposition~\ref{3piecegeodesic} we can rearrange escapes and toral
 subpaths of $\gamma_2$ to obtain a geodesic $\gamma_2'$
 with the same endpoints that 
 consists of at most one $x$--escape, at most one $y$--escape, and at
 most one toral subpath of exponent less than $L$, in that order.
 First, we observe that $\gamma_2'$ cannot contain an  $x$--escape as
 it would start with an $s$--edge, which would given an $s$--cancellation with the final
 edge of $\gamma_1$, contradicting the fact that $\gamma$ is
 geodesic. 

 Next, observe that there is no toral subpath in $\gamma_2'$, because $\gamma_2'$ is a path between $\langle x\rangle$ and
$a^{L^p}\langle x\rangle$, so if the toral path has exponent $\ell$ and
the $y$--escape has exponent $n$ then we would have $a^\ell
y^n=a^{L^p}x^m$ for some $m$. 
This implies $a^{L^p-\ell}=x^{-m}y^n$, but $\langle a\rangle\cap\langle
x,y\rangle=\langle a^L\rangle$, so $\ell$ is divisible by $L$.
But $0\leq \ell<L$, so $\ell=0$.
Thus, $\gamma_2'$ is a single geodesic $y$--escape of exponent
$L^{p-1}$, which
implies that $\gamma_2=\gamma_2'$ and $|\gamma_2|=|y^{L^{p-1}}|=2+|a^{L^{p-1}}|$,  by Lemma~\ref{lem:ag_bound}\eqref{item:glength}.

We proceed by induction on $p$.

In the case $p=1$ the snowflake paths are $\sigma_{1,s}=sas^{-1}tat^{-1}$ and
$\sigma_{1,t}=tat^{-1}sas^{-1}$, so the pairs of antipodal points we are
considering are
$\{u=s,\,v=tat^{-1}sa=a^Ls\}$ and $\{u=sa=xs,\,v=tat^{-1}s=ys\}$.
Observe that in both cases the points $us^{-1}$ and $vs^{-1}$ do not differ by a power of $y$, so
it is not the case that both $\gamma_1'$ and $\gamma_3'$ are trivial. 
Thus $|\gamma_2|=|\gamma|-|\gamma_1'|-1-|\gamma_3'|-1\leq 6-3=3$.
The only length 3 $y$--escapes  from $\langle x\rangle$ to $a^L\langle
x\rangle$ are the ones labelled $tat^{-1}$, so we have
$\gamma=\gamma_1'+s^{-1}tat^{-1}s+\gamma_3'$.
Since $|\gamma| \le 6$, one of
$\gamma_1'$ or $\gamma_3'$ is trivial and the other consists of a
single edge.
If $\{u,\,v\}=\{s,a^Ls\}$ then we can choose either $\gamma_1'=1$ and
$\gamma_3'=a$ to get $\gamma$ to be the path from $u$ labelled $s^{-1}tat^{-1}sa$, which goes
through 1, or $\gamma_1'=a$ and $\gamma_3'=1$ to get
$\gamma$ to be the path from $u$ labelled $as^{-1}tat^{-1}s$, which goes through $a^L$.
These are precisely the two subsegments of the snowflake loop between
$u$ and $v$.
The calculation is similar for $\{u,\,v\}=\{xs,\,ys\}$.

Now we suppose that the claims are true up to some power $p$, and
extend to the case of $\sigma_{p+1,s}+\bar\sigma_{p+1,t}$.  With the
setup as above, $\gamma_1'$ begins at $u$ and ends at $sa^m$, for some
$m$, so $s^{-1}.\gamma_1'$ is a path from $s^{-1}u$ to $a^m$.  Since
$\gamma_2$ is a single $y$--escape of exponent $L^p$, $\bar\gamma_3'$
begins at $sa^ms^{-1}y^{L^p}s=x^my^{L^p}s=y^{L^p}x^ms=y^{L^p}sa^m$ and
ends at $v$, hence $s^{-1}y^{-L^p}.\bar\gamma_3'$ begins at $a^m$ and
ends at $s^{-1}y^{-L^p}v$.  Thus
$s^{-1}.\gamma_1'+s^{-1}y^{-L^p}.\bar\gamma_3'$ is a path from
$s^{-1}u\in\sigma_{p,s}$ to $s^{-1}y^{-L^p}v\in\sigma_{p,t}$.
Furthermore, since $u$ and $v$ were antipodal, we claim that $s^{-1}u$
and $s^{-1}y^{-L^p}v$ are antipodal in
$\sigma_{p,s}+\bar\sigma_{p,t}$.  To see this, let $\alpha$ be the
segment of $\sigma_{p+1,s}$ from $u$ to $a^{L^{p+1}}$, and let $\beta$
be the segment of $\bar\sigma_{p+1,t}$ from $a^{L^{p+1}}$ to $v$, so
that $|\alpha|+|\beta|=|\sigma_{p+1,s}|$.  The path $\alpha$
decomposes as $\alpha'+s^{-1}+t+\sigma_{p,s}+t^{-1}$, where
$\alpha'+s^{-1}$ is the subsegment of $\sigma_{p+1,s}$ from $u$ to
$x^{L^p}$.  We have that $s^{-1}.\alpha'$ is the subsegment of
$\sigma_{p,s}$ from $s^{-1}u$ to $s^{-1}x^{L^p}s=a^{L^{p}}$.  On the
other hand, $\beta=s+\beta'$, where $\beta'$ is the subsegment of
$\bar\sigma_{p+1,t}$ from $a^{L^{p+1}}s$ to $v$, and
$s^{-1}y^{-L^p}.\beta'$ is the subsegment of $\bar\sigma_{p,t}$ from
$s^{-1}y^{-L^p}a^{L^{p+1}}s=a^{L^p}$ to $s^{-1}y^{-L^p}v$.  Thus,
$s^{-1}.\alpha'+s^{-1}y^{-L^p}.\beta'$ is the subsegment of
$\sigma_{p,s}+\bar\sigma_{p,t}$ from $s^{-1}u$ to $s^{-1}y^{-L^p}v$
going through $a^{L^p}$, and it has length:
\begin{align*}
  |\alpha'|+|\beta'|&=|\alpha|-3-|\sigma_{p,s}|+|\beta|-1\\
                    &=|\sigma_{p+1,s}|-|\sigma_{p,s}|-4\\
                    &=5\cdot 2^{p+1}-4-(5\cdot 2^p-4)-4\\
                    &=5\cdot 2^p-4=|\sigma_{p,s}|
\end{align*}

Since $s^{-1}u$ and $s^{-1}y^{-L^p}v$  are
   antipodal in $\sigma_{p,s}+\bar\sigma_{p,t}$,
   the induction hypothesis says that the subsegment $\delta$ of
   $\sigma_{p,s}+\bar\sigma_{p,t}$ from $s^{-1}u$ to $s^{-1}y^{-L^p}v$
   passing through $a^{L^p}$ is geodesic.
Split $\delta$ at  $a^{L^p}$ as a concatenation $\delta=\delta_1+\delta_3$, and define
$\delta':=s.\delta_1+s^{-1}+t+\sigma_{p,s}+t^{-1}+s+\bar\delta_3$.
Then $\delta'$ is the subsegment of
$\sigma_{p+1,s}+\bar\sigma_{p+1,t}$ from $u$ to $v$ that goes through $a^{L^{p+1}}$.
We must show $\delta'$ is geodesic.
We have:
\begin{align*}
  |\delta'|-|\gamma|&=|\delta_1|+|\delta_3|+4+|\sigma_{p,s}|-\left(|\gamma_1'|+|\gamma_3'|+2+|\gamma_2|\right)\\
  &=|\delta_1|+|\delta_3|+4+|a^{L^p}|-\left(|\gamma_1'|+|\gamma_3'|+4+|a^{L^p}|\right)\\
  &=|\delta_1|+|\delta_3|-|\gamma_1'|-|\gamma_3'|
\end{align*}
On the one hand this quantity is nonnegative, because $\delta'$ is a
path from $u$ to $v$ and $\gamma$ is a geodesic with the same
endpoints. On the other hand it is nonpositive, since
$s^{-1}.\gamma_1'+s^{-1}y^{-L^p}.\gamma_3'$ is a path from $s^{-1}u$ to
$s^{-1}y^{-L^p}v$ and $\delta=\delta_1+\delta_3$ is a geodesic with the same endpoints.
Thus, $|\delta'|=|\gamma|$, so $\delta'$  is geodesic. Since the
complement to $\delta$ in $\sigma_{p+1,s}+\bar\sigma_{p+1,t}$ has the
same length and the same endpoints, it is geodesic too.
We also get that  $s^{-1}.\gamma_1'+s^{-1}y^{-L^p}.\gamma_3'$ is a geodesic
from $s^{-1}u$ to $s^{-1}y^{-L^p}v$, which are antipodal points on $\sigma_{p,s}+\bar\sigma_{p,t}$, so by  the induction hypothesis
$s^{-1}.\gamma_1'+s^{-1}y^{-L^p}.\gamma_3'$  passes through either $1$ or $a^{L^p}$, necessarily at the
concatenation point.
This means that either $\gamma_1$ ends at $1$ or $\gamma_3$ begins at
$a^{L^{p+1}}$, so $\gamma$ passes through either $1$ or $a^{L^{p+1}}$.
That completes the induction.
\end{proof}

\section{The idea of the proof that asymptotic cones are simply
  connected}\label{sec:idea}

\subsection{The snowflake case}\label{sec:idea:snowflake}
Let us sketch the argument that snowflake loops can be subdivided in a way that satisfies
Theorem~\ref{rileycriteria}.
The idea in the general case will be that, for $K$ sufficiently close
to 1, $K$--biLipschitz loops look sufficiently similar to snowflake
loops that the same strategy works.

Consider a snowflake loop $\gamma$ for $a^{L^p}$ with HNN
diagram as in Figure~\ref{fig:snowflake}.
We will subdivide each region and replace it by regions with
relatively short boundaries.
The diagram has a central diamond, the sides of which are labelled
$x^{L^{p-1}}$ and $y^{L^{p-1}}$.
From \eqref{eq:2}, the length of $\gamma$ is $2(5\cdot2^p-4)$.
We choose a subdivision constant $\Lambda$, which in this case we can
take to be $\Lambda=L$.
Subdivide the central diamond into $\Lambda^2$ small diamonds whose
sides are labelled $x^{L^{p-2}}$ and $y^{L^{p-2}}$.
If we take one of these small diamonds and replace its sides by
geodesics we get a small snowflake loop whose length is $2(5\cdot
2^{p-1}-4)<\frac{1}{2}|\gamma|$.
Propagate the subdivision into the branches of the snowflakes to get
subdivisions of the neighboring triangular regions, as in Figure~\ref{fig:subdividedsnowflake}.
This gives us $\sim \Lambda^2$ more short loops.

\begin{figure}[h]
  \centering
  \includegraphics[height=4cm]{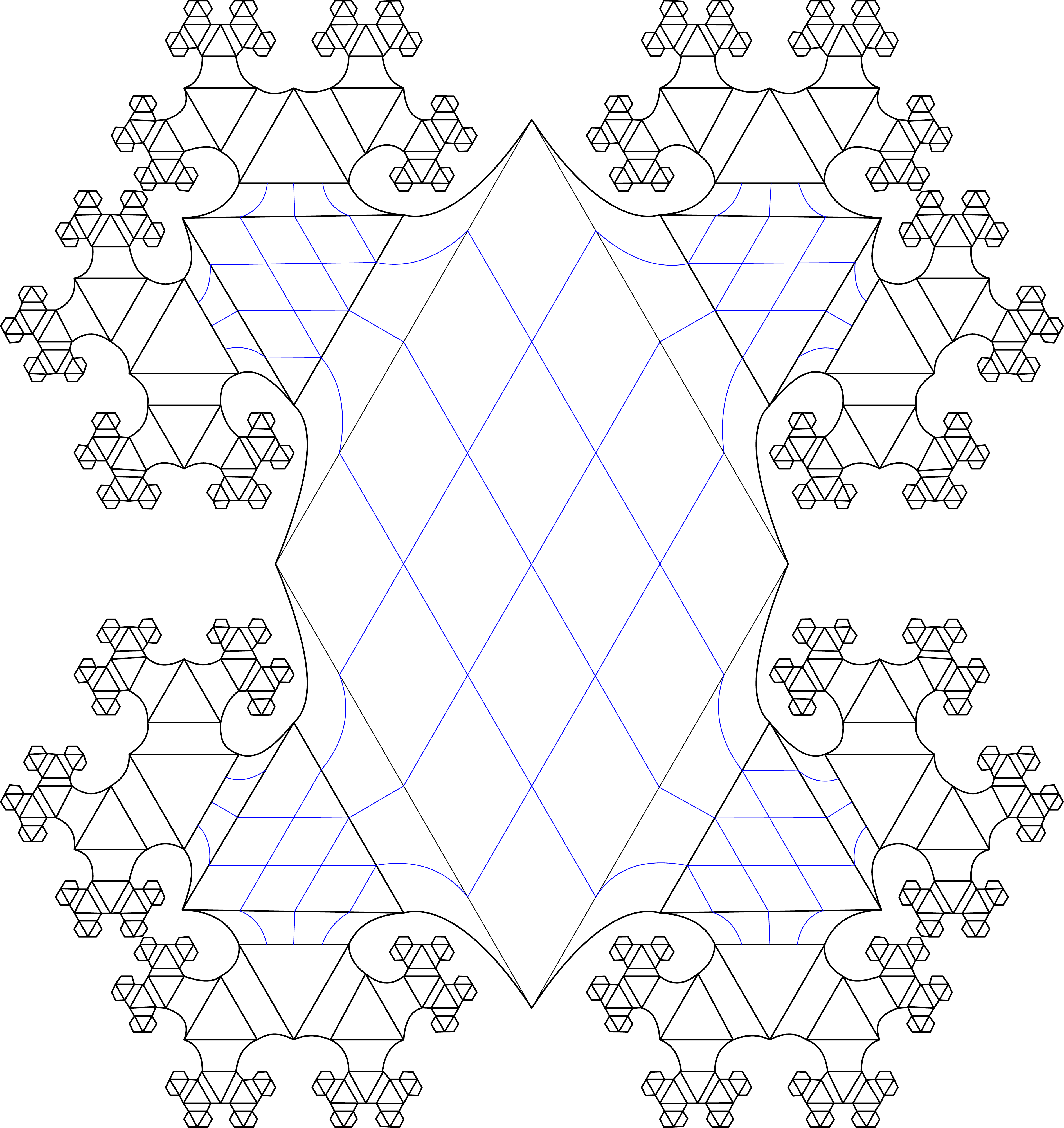}
  \caption{Subdivided snowflake to depth 1.}
  \label{fig:subdividedsnowflake}
\end{figure}

Obviously, we cannot continue this propagation indefinitely, since the depth of the snowflake depends on the
boundary word, but to apply Theorem~\ref{rileycriteria} we need the
number of short loops to be uniform.
Consider the situation depicted in Figure~\ref{fig:cap}---we have
subdivided some number of levels deeply into the snowflake, and are
faced with the remaining branch.
\begin{figure}[h]
  \centering
  \includegraphics[height=2.5cm]{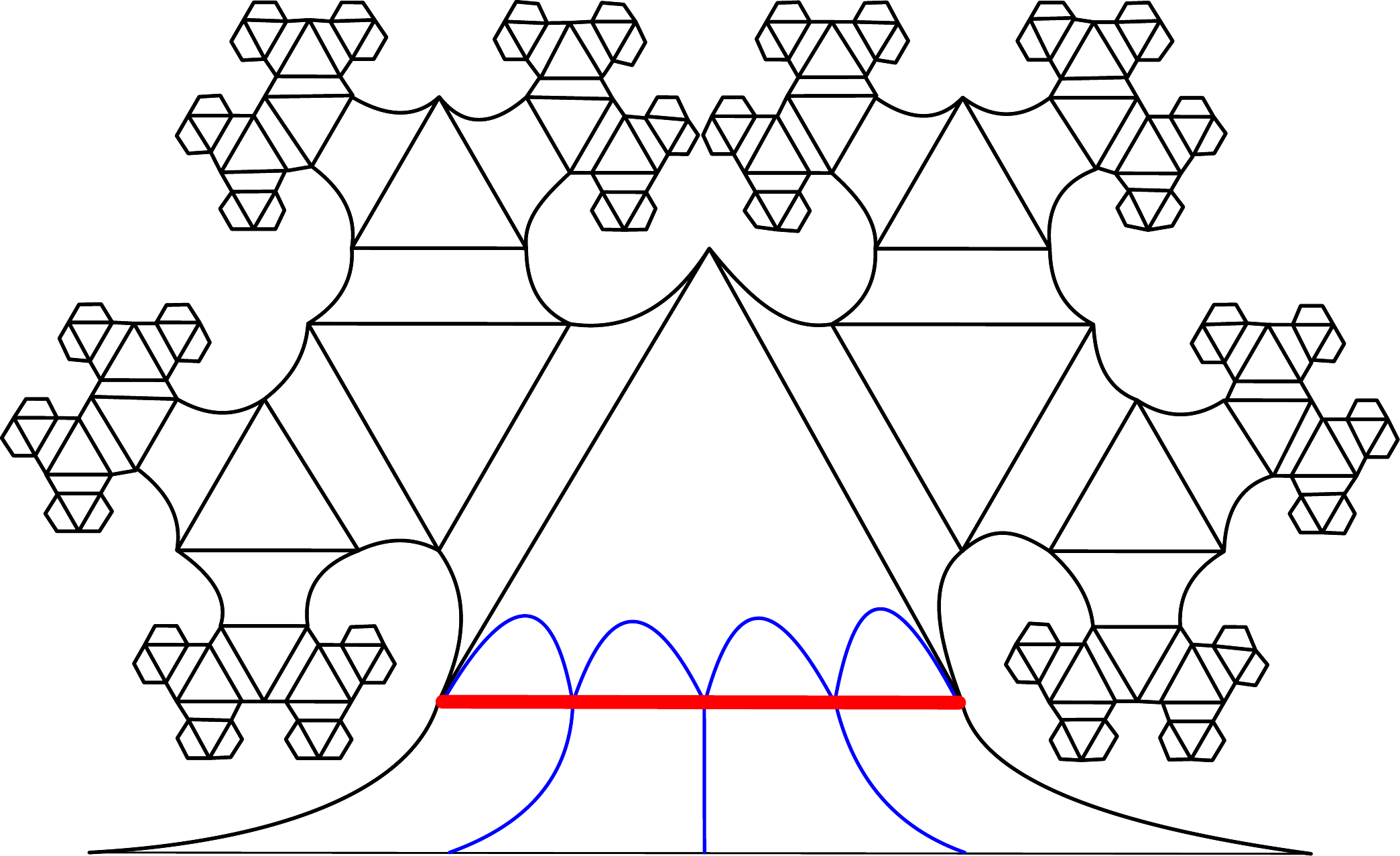}
  \caption{Capping off a branch}
  \label{fig:cap}
\end{figure}
The incoming corridor has been subdivided into $\Lambda$--many pieces.
These pieces have been replaced by short geodesic segments.
The idea is to make one short loop consisting of these $\Lambda$--many
geodesics, plus the outer boundary of the branch.
The incoming edge of the triangle, the red edge in
Figure~\ref{fig:cap}, is labelled $a^{L^{p-m}}$ if we
are at depth $m$.
Thus, the length of the outer boundary of the branch is
$|a^{L^{p-m}}|$, while the length of each of the
$\Lambda$--many short geodesics is $|a^{L^{p-m-1}}|$, since
$\Lambda=L$.
Given our formula \eqref{eq:2}, we can solve to see that
$|a^{L^{p-m}}|+\Lambda |a^{L^{p-m-1}}|\leq \frac{1}{2}|\gamma|$ once
$m\sim\log_2L$.
So at some depth $m$, \emph{independent of $p$}, the paths have become
short enough that we can cap off a branch
with a single short loop. 

This argument decomposes the snowflake loop $\gamma$ into  $\sim L^3$ loops of length at most $\frac{1}{2}|\gamma|$.
If a similar bound were to hold for biLipschitz loops, then, by
\cite[Theorem~4.3.3]{brady2007geometry}, this would imply that the Dehn function of $G_L$ is bounded above by $n^{3\log_2L}$.
This is consistent with the true Dehn function, which is known to be $n^{2\log_2L}$.
We have made no effort to optimize the decomposition, and do not know
if it is possible to subdivide into only $\sim L^2$--many short loops.

\subsection{The general case}\label{sec:generalcase}
We must show that for all $\lambda\in (0,1)$ there exist $M$ and $A$ such that every
loop $\gamma$ in $X$ can be filled by a diagram with area at most $A$
and mesh at most $\lambda|\gamma|+M$.

{\bf Step 1, Reduce to biLipschitz loops:}
 In Section~\ref{sec:circletightening} we show that if $\gamma$ fails to be biLipschitz at a large scale, in a sense to be made 
  precise, then it is easy to
  fill: introduce a geodesic segment between two far apart points
  where the biLipschitz condition fails and take an area 2 filling of
  the $\Theta$ graph with the 1--skeleton mapping to the union of
  $\gamma$ and this geodesic segment.

  Otherwise, if $\gamma$ fails to be biLipschitz only at smaller
  scales then Theorem~\ref{thm:circle_tightening} says it can be
  tightened to a $K$--biLipschitz loop without changing it very much,
  where we are free to choose $K>1$ as small as we like.
  Corollary~\ref{cor:reduction} says if we can find the desired
  filling of the biLipschitz loop then we can extend it to the
  original loop without changing the area or the mesh too much. Thus, we have reduced the problem to finding some
  $K>1$ for which we can fill all $K$--biLipschitz loops. 

  {\bf Interlude; how different are biLipschitz loops from snowflakes?:}
Figure~\ref{fig:jankysnowflake} shows a schematic of an HNN diagram
for a biLipschitz loop. 
\begin{figure}[h]
  \centering
  \includegraphics[angle=90,width=.5\textwidth]{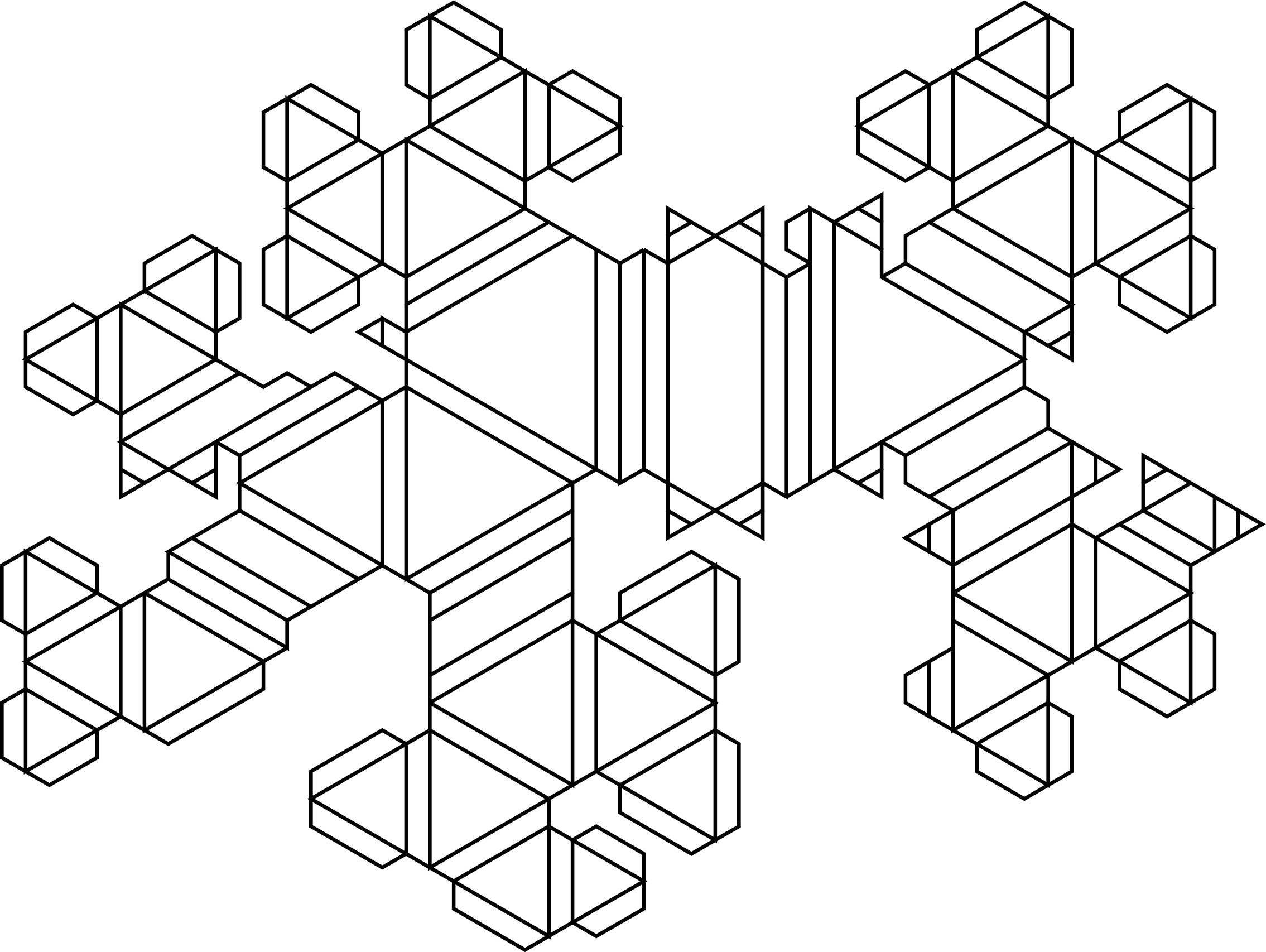}
  \caption{biLipschitz snowflake}
  \label{fig:jankysnowflake}
\end{figure}
We point out some differences between
Figure~\ref{fig:jankysnowflake} and Figure~\ref{fig:snowflake}:
\begin{enumerate}
\item Figure~\ref{fig:jankysnowflake} has no central diamond.
      \item There can be adjacent corridors with no vertex region in between.
  \item Not all of the noncentral vertex regions are triangles.
  \end{enumerate}

  We claim these differences are not severe.
  It turns out there is a way to choose a central region.
  Moreover, although regions may have any number of sides, we will
  show that most of them are negligibly small and are easy to account
  for in our filling diagram without affecting the mesh very much. 
  More specifically, the central region has at most four long sides,
  and the other regions have either two or three long sides.
  It will take some work to make these statements precise.
  
  Branching of the diagram played an important role in the snowflake
  case, so the presence of bigon regions may be concerning.
  We assemble maximal chains of corridor and bigon regions into
  subdiagrams that we call \emph{enfilades} in Section~\ref{sec:noncentralregions}.

  {\bf Step 2, all diagrams have a central `region':} Consider an HNN
  diagram $D\to X$ for an embedded loop $\gamma$ in $X$.  Since
  $\gamma$ is embedded, $D$ is a disk.  If $D$ is a single region then
  call that one the central region.  If not, $D$ contains corridors.
  Recall that corridors have disjoint interiors and bisect $D$, and
  contain two disjoint edges in $\partial D$.  Furthermore, corridors
  are 2--cells with boundary label of the form $sa^ns^{-1}x^{-n}$ or
  $ta^nt^{-1}y^{-n}$. We parameterize such a 2--cell as a rectangle
  $[0,n]\times [0,1]$ and define a map $\phi\from D\to |T|$ from $D$ to
  the geometric realization $|T|$ of the dual tree $T$ to the corridor
  decomposition of $D$, such that $\phi$ restricted to a corridor
  parameterized $[0,n]\times [0,1]$ is projection to the second
  coordinate followed by a linear isometry to the geometric
  realization of an edge.

  Since corridors contain disjoint edges on $\partial D$, for every $x\in
  |T|$, $\phi^{-1}(x)$ contains at least two points of $\partial D$, so
  $\gamma-\phi^{-1}(x)$ consists of at least two components.
  Define a function $f$ on $|T|$ by:
    \[x\mapsto (\text{maximum length of a component of $\gamma-\phi^{-1}(x)$})-|\gamma|/2\]

    The function $f$ is typically not continuous.  We claim that there
    is a unique point at which $f$ achieves a non-positive value.

    First, suppose that $f$ achieves a non-positive value at $x$. Then
    $\gamma-\phi^{-1}(x)$ consists of finitely many segments
    $\gamma_i$, each of length at most $|\gamma|/2$.
    For any other $y\in |T|$, there is a component $\gamma'$ of
    $\gamma-\phi^{-1}(y)$ and an $i_0$ such that $\gamma-\gamma'$ is
    contained in $\gamma_{i_0}$.
    Furthermore, $\gamma'$ crosses corridors corresponding to the
    edges of $T$ on the geodesic between $x$ and $y$.
    Thus, $|\gamma'|\geq
    |\gamma|-|\gamma_{i_0}|+2d_T(x,y)>|\gamma|-|\gamma_{i_0}|\geq
    |\gamma|/2$, so  $f(y)\geq |\gamma'|-|\gamma|/2>0$.
    So if $f$ achieves a non-positive value, it does so at a unique
    point of $|T|$. 

    Next, suppose that $v$ is a vertex of $T$ with $f(v)>0$. Then
    there is a component $\gamma'$  of $\gamma-\phi^{-1}(v)$ of length strictly
    greater than $|\gamma|/2$.
    The closure $\cl(\gamma')$ of $\gamma'$ maps via $\phi$ to  a loop in $T$ at $v$ that begins and ends
    by crossing some edge $e$ incident to $v$.
    Suppose this edge is parameterized $[0,1]$ with $0$ at $v$.
    Then for $x\in (0,1)$ there is a component of
    $\gamma-\phi^{-1}(x)$ whose closure is a subsegment of $\gamma'$
    of length $|\gamma'|-2x$.
    So either $f(v)>1$ and $f$ is linearly decreasing and positive on
    $[0,1)$, or $f(v)=1$ and $f$ achieves value 0 at the midpoint of
    $e$.
    In both cases, $f(v)>0$ implies that $f(v)$ is not a local
    minimum. 
    
    Finally, suppose that $f$ is positive and decreasing on some edge
    parameterized $[0,1]$.  Then for each $x\in (0,1)$ there is a
    component $\gamma_x$ of $\gamma-\phi^{-1}(x)$ such that
    $|\gamma_x| > |\gamma|/2$ and as $x\to 1$ the closed segments
    $\cl(\gamma_x)$ form a decreasing nested family all containing
    $\phi^{-1}(1)\cap \partial D$. The limit $\gamma'$ of this family
    as $x\to 1$ is a closed subsegment of $\gamma$ of length at least
    $|\gamma|/2$ containing $\phi^{-1}(1) \cap \partial D$.  There are
    two possibilities: either $\phi^{-1}(1) \cap \partial D$ is
    exactly two points and $f(1)=|\gamma'|-|\gamma|/2$, so
    $f|_{(0,1]}$ is continuous at 1, or $\phi^{-1}(1) \cap \partial D$
    has more than two points and the components of
    $\gamma-\phi^{-1}(1)$ are either $\gamma-\gamma'$ or proper
    subsegments of $\gamma'$.  In this case
    $f(1) \le |\gamma'|-|\gamma|/2$, so $f(1)$ is a strict lower bound for
    $f$ on $(0,1)$.

   Since there are finitely many vertices, we conclude that $f$
   achieves a minimum value.
   If $f(x)>0$
   then $x$ is not a local minimum, so the minimum value achieved by
   $f$ is non-positive.

    If $f$ achieves value 0 on the midpoint of an edge then we call
    the corresponding corridor the central region $R$ of $D$.
    If $f$ achieves its minimum value at a vertex $v$ then we define
    $R=\phi^{-1}(v)$ to be the central region of $D$.
    This includes a degenerate possibility that $R$ is not really a
    region; it can be that $\phi^{-1}(v)$ is a tree in
    the 1-skeleton of $D$ with leaves on $\partial D$.
    The important point is that the map $D\to X$
    sends $\phi^{-1}(v)$ into a single coset of $H$. 

\begin{remark}
  There is a natural map from the dual tree in the above argument into
  the Bass-Serre tree of $G$, so to every vertex $v\in T$ there is
  associated a coset $g_vH$.
  However, we cannot assume this map is injective.
  In particular, if $v$ is a vertex of $T$, we cannot conclude that
the closures of components of $\gamma-\phi^{-1}(v)$ are \emph{escapes} from $g_vH$.
  One of these components could return to $g_vH$ in its interior. 
We need a more general definition to describe these
 paths. 
\end{remark}
\begin{definition}\label{def:compoundescape}
  Let $\gamma$ be a path in $X$ that intersects a coset $gH$.
  Define a \emph{compound escape} of $\gamma$ from $gH$ to be a
  subpath $\gamma'$ of
  $\gamma$ with both endpoints in $gH$ such that one of the following
  holds:
  \begin{itemize}
  \item $\gamma'$ begins with $s$, ends with $s^{-1}$, and has
    endpoints that differ by a power of $x$.
    \item $\gamma'$ begins with $t$, ends with $t^{-1}$, and has
      endpoints that differ by a power of $y$.
      \item $\gamma'$ begins with $s^{-1}$, ends with $s$, and has
        endpoints that differ by  a power of $a$.
        \item $\gamma'$ begins with $t^{-1}$, ends with $t$, and has
          endpoints that differ by a power of $a$.
        \end{itemize}
 The first two cases we call compound $x$ and $y$--escapes,
 respectively, and the last two compound $a$--escapes.
\end{definition}
So our usual, simple, escapes are also compound escapes, but a compound escape
may return to $gH$ in its interior. 
\medskip

In the case where $\gamma$ is a $K$-biLipschitz (and hence embedded)
loop, we gain two things by identifying the central region:
\begin{itemize}
\item If $R$ is the central region, each component of $\gamma-R$ is a
  $K$-biLipschitz path with endpoints in a coset of $H$. 
\item If $R$ is a non-central region then at most one component $\gamma'$ of
  $\gamma-R$ can fail to be $K$--biLipschitz, and $\gamma-\gamma'$ is
  a single $K$--biLipschitz path with endpoints in a coset of $H$.
\end{itemize}
We analyzed \emph{geodesics} between elements of $H$ in
Section~\ref{hgeodesics}, and now we must:

{\bf Step 3, Analyze biLipschitz paths between elements of $H$:}
This is Section~\ref{sec:bilippaths}.
The underlying idea is that for such a path to be efficient it should
make as much use of the distortion of $H$ as possible, and the way
that that happens is when the path is almost entirely either one long
escape or two long escapes that are an $x$--escape and a $y$--escape. 

{\bf Step 4, Apply the previous step to analyze central and
  non-central regions of the diagram:}
This is done in 
Section~\ref{sec:escapecentralregion}
and
Section~\ref{sec:noncentralregions}.
The conclusions, as alluded to above, are that the central region has
at most four long sides, and they are configured in one of 6 possible
ways, and every other region has at most 3 long sides, and in the case
of 3 they approximately form a triangle of $a$, $x$, and $y$--escapes
just like in the snowflake case.
Furthermore, we show that when we lump together adjacent  bigon regions and corridors
into enfilades, the enfilades are metrically thin.

{\bf Step 5, Construct a new filling diagram:} We have established
that the structure of the diagram is a central region with at most 4
adjacent enfilades, each of which leads to a triangular region. Each
of these has two outgoing sides with adjacent enfilades leading to
more triangular regions, etc.  The strategy now will be the same as in
the snowflake case: choose a way to subdivide the central region and
propagate this subdivision out across the original diagram in a way
that the mesh is controlled and such that after a uniform number of
steps we can cap off the remaining branches with single 2--cells.
This is done in Section~\ref{sec:constructfillings}.

\section{Reduction to biLipschitz cycles}\label{sec:circletightening}
The following theorem is a specialized form of a result of Hoda \cite{Hodstrongshortcut}.
\begin{thm}[{Circle Tightening Lemma
    \cite[Theorem~G]{Hodstrongshortcut}}]
  \label{thm:circle_tightening}
  Let $N > 1$ and $K > 1$, and let $\theta > 1$ be small enough (depending on
  $N$ and $K$) and let $M > 0$ be large enough (depending on $N$, $K$,
  and $\theta$).  Let $\alpha \from S \to X$ be a $1$-Lipschitz map from a
  Riemannian circle $S$ to a geodesic metric space $X$.  If
  \[ d_X\bigl(\alpha(p), \alpha(\bar p)\bigr) \ge \frac{1}{\theta} d_S(p,
    \bar p) \] for every antipodal pair $p, \bar p \in S$ and
  $|S| > M$ then there exists a countable collection $\{Q_i\}_i$ of
  pairwise disjoint closed segments in $S$ such that the following
  conditions hold.
  \begin{enumerate}
  \item $\sum_i|Q_i| < \frac{|S|}{N}$
  \item For each $i$, we have
    $d_X\bigl(\alpha(p_i), \alpha(q_i)\bigr) < |Q_i|$ where $p_i$ and
    $q_i$ denote the endpoints of $Q_i$.
  \item There is a collection of geodesics
    $\{\gamma_i \from \bar Q_i \to X\}_i$ such that $\gamma_i$ has
    endpoints $\alpha(p_i)$ and $\alpha(q_i)$ and replacing each
    $Q_i \to X$ in $\alpha$ with $\bar Q_i \to X$ results in a
    $K$-biLipschitz embedding $\alpha'$ of a circle that is also
    $1$-Lipschitz.
  \end{enumerate}
\end{thm}
\begin{remark}
  \label{rmk:graph_circle_tightening} If $X$ is a graph and $\alpha$
  is a combinatorial map then no $Q_i$ can be contained in a single
  edge, as this would contradict
  $d_X\bigl(\alpha(p_i), \alpha(q_i)\bigr) < |Q_i|$.  Then each $Q_i$
  contains a vertex in its interior and it follows that the family
  $Q_i$ is necessarily finite.  Moreover, since $\alpha'$ is
  bilipschitz, it is also a combinatorial map.
\end{remark}

\begin{corollary}\label{cor:reduction}
  Let $X$ be a Cayley graph. 
  Suppose there exist $K>1$, $0<\lambda<1$, $M$, and $A$ such that every
  $K$--biLipschitz loop $\alpha$ admits a filling of area at most $A$ and mesh
  at most $\lambda|\alpha|+M$. Then there exists $0<\lambda'<1$ and
  $M'$ such
  that every 1--Lipschitz loop $\alpha$ admits a filling of area at
  most $A+3$ and mesh at most $\lambda'|\alpha|+M'$.
\end{corollary}
\begin{proof}
  Take $N>12$. Let $\theta>1$ be as in
  Theorem~\ref{thm:circle_tightening} with respect to $N$ and $K$.
  Let $M'$ be the larger of $M$ and the $M$ of
  Theorem~\ref{thm:circle_tightening}.
  Define:
  \[\lambda':=\max\left\{\lambda,\,\frac{1+\theta}{2\theta},\, 2\left(\frac{1}{3}+\frac{2}{N}\right)\right\}<1\]

  Consider a 1--Lipschitz loop $\alpha\from S\to X$. 
  If there exist a pair of antipodal points $p$ and $\bar p\in S$ such
  that $d_X(\alpha(p),\alpha(\bar p))<\frac{1}{\theta}\frac{|S|}{2}$
  then there is a diagram of area 2 filling $\alpha$ obtained by
  taking one 2--cell to have boundary the segment of $\alpha$ restricted to the interval of $S$ from $p$ to $\bar
  p$ followed by a geodesic $\delta$ from $\alpha(\bar p)$ to
  $\alpha(p)$, and taking the other to have boundary $\bar\delta$
  followed by the remaining subsegment of $\alpha$.
  These loops have length
  $|S|/2+|\delta|<\frac{1+\theta}{2\theta}|S|\leq \lambda'|S|$.

  If there is no such pair of antipodal points then
  Theorem~\ref{thm:circle_tightening} applies.
  We use notation as in the theorem.
  By hypothesis, there is a filling of $\alpha'$ of area at most $A$
  and mesh at most $\lambda|S'|+M\leq \lambda'|S|+M'$.
  We need to extend this to a filling of $\alpha$.
  Pick points $s_0$, $s_1$, and $s_2$ in
  $S-\cup_i\stackrel{\circ}{Q}_i$ such that for all $i$, subscripts
  mod 3, we have
  $d_S(s_i,s_{i+1})\leq\left(\frac{1}{3}+\frac{2}{N}\right)|S|$.
  Then we get 3 loops of the form: segment of $\alpha$ restricted to
  the interval of $S$ from $s_i$ to $s_{i+1}$ followed by reverse of
  the segment of $\alpha'$ from $s_i$ to $s_{i+1}$.
  The length of such a loop is at most twice the $\alpha$ part, since
  $\alpha'$ is a tightening of $\alpha$, so it is bounded above by
  $2\left(\frac{1}{3}+\frac{2}{N}\right)|S|\leq \lambda'|S|$.
  Adding three 2--cells with these boundary loops to the diagram for
  $\alpha$ gives us the desired diagram for $\alpha$.
\end{proof}

\section{BiLipschitz cycles in snowflake
  groups}\label{sec:bilipschitz}
Recall, Remark~\ref{rk:fixedconstants}, that the group $G_L$ is deemed to be
fixed, so $L$, $C$, and $\alpha$ are constant.
The other `constants' in this section implicitly depend on the choice
of $L$.

In Section~\ref{sec:bilippaths} we analyze biLipschitz paths with
endpoints in $H$.
Sections \ref{sec:escapecentralregion} and \ref{sec:noncentralregions}
are preparatory for Section~\ref{sec:constructfillings}. In those
sections we establish properties for central and non-central regions of
a putative filling of a biLipschitz loop $\gamma$.
Since `region' is a term applied to an already existing filling, we will
phrase these properties in a way that is intrinsic to $\gamma$.

\subsection{BiLipschitz paths with endpoints in $H$}\label{sec:bilippaths}

The main result of this subsection is:
\begin{theorem*}[Theorem~\ref{longescapesinlength}]
    For all $R>1$ there exist $J>1$ and $1<K<2$ such that if
    $\gamma$ is a $K$--biLipschitz path with endpoints in
    $H=\langle a,x,y\rangle$ 
    then either $|\gamma|\leq J$ or $\gamma$ has either one or two long escapes
    that collectively account for at least $\frac{R-1}{R}$ fraction
    of its length.  
    Furthermore, in the case of two long escapes they
    are one $x$--escape and one $y$--escape.
    
    If the
    endpoints of $\gamma$ differ by a power of $x$ or $y$, then there
    is only one long escape, and it is of the corresponding flavor,
    and if the endpoints of $\gamma$ differ by a power of $a$ then
    either there is one long $a$--escape or there is one long
    $x$--escape and one long $y$--escape. 
  \end{theorem*}
  We build up to this theorem through a string of auxiliary results.
  Most follow the same template: given a target $R$ that quantifies the
  desired ratio of domination, it is possible to choose $K>1$
  sufficiently close to $1$ so that any $K$--optimal path $\gamma$
  with
  endpoints in $H$ must have one or two escapes that dominate the others. 
The complexity of the assumed structure of
$\gamma$ increases with each lemma.

There is also an issue that in the theorem the domination is phrased
in terms of length, but for most of the auxiliary results it is more
convenient to work with exponents, so we will have to make a
conversion. Because the relationship between exponent and length is
not monotonic, this conversion is nontrivial. 

For the first step, we suppose that $\gamma$ consists of two escapes
of the same flavor:
\begin{lemma}
  \label{lem:triangle_ineq_enhanced}
There exists $J>1$ such that for all $R>1$ there exists
   $1 < K < 2$ such that for all
    $|m|,|n|\geq J$, if $g\in\{a,x,y\}$ and 
\[ |g^m|+|g^n|\leq K|g^{m+n}|\]
then $\max\{\frac{|n|}{|m|},\frac{|m|}{|n|}\} > R$.
\end{lemma}

\begin{proof}
  First assume $g=a$.
  Let $J:=1+(L+3)/2$.
Let $m=\sum_{i=0}^j m_iL^i$ and  $n=\sum_{i=0}^k n_iL^i$ be geodesic  expressions
for $m$ and $n$ satisfying Condition~\eqref{eq:1}.
Since $|m|,|n|\geq J$, we have $j,k>0$.
Without loss of generality, assume $k\geq j$ and $n>0$ so that
$n_k\geq 1$. 

We approximate as follows:
\begin{align*}
  \frac{|n|}{|m|} & =\frac{|\sum_{i=0}^kn_iL^i|}{|\sum_{i=0}^jm_iL^i|} & \\
             &=\frac{|n_kL^k+\sum_{i=0}^{k-1}n_iL^i|}{|m_jL^j+\sum_{i=0}^{j-1}m_iL^i|} &  \\
               &\geq\frac{L^k+\sum_{i=0}^{k-1}-(L/2)L^i}{(L/2+2)L^j+\sum_{i=0}^{j-1}(L/2)L^i} & \textrm{by~\eqref{eq:1}}\\
&=
                                                                                                  \frac{L^k-\frac{1}{2}\frac{L}{L-1}(L^k-1)}{\frac{1}{2}(L+4)L^j+\frac{1}{2}\frac{L}{L-1}(L^j-1)} & \\
                 &=\frac{2(L-1)L^k-L(L^k-1)}{(L-1)(L+4)L^j+L(L^j-1)} & \\
                 &=\frac{L^{k+1}-2L^k+L}{L^{j+2}+4L^{j+1}-4L^j-L} & \\
  &=\frac{(L-2)L^{k-1}+1}{(L+4)L^j-4L^{j-1}-1} &
\end{align*}

Thus, $|n|/|m|>R$ if:
\[(L-2)L^{k-1}+1-R\left((L+4)L^j-4L^{j-1}-1\right)>0\]
The left hand side is bigger than $(L-2)L^{k-1}-R(L+4)L^j$, so the claim holds if:
\begin{equation}
  \label{eq:5}
  k-j>1+\log_L\left(\frac{R(L+4)}{L-2}\right)
\end{equation}

Otherwise, if $0\leq k-j\leq 1+\log_L(R(L+4)/(L-2))$, we choose $K$ sufficiently close to $1$ to make
$|a^m|+|a^n|\leq K|a^{m+n}|$ impossible.
To show that this can be done we use the fact that
$m+n=\sum_{i=0}^k(m_i+n_i)L^i$, where $m_i=0$ for $i>j$, to provide an
upper bound for $|a^{m+n}|$. 

\begin{align*}
  |a^m|+|a^n|&-K|a^{m+n}|\geq
                          |a^m|+|a^n|-K\left(\sum_{i=0}^{k}|m_i+n_i|2^i
                          +4(2^k-1)\right) & \\
                        &=4(2^j-1)+4(2^k-1)-4K(2^k-1) & \\
                        &\quad\quad\quad+\sum_{i=0}^{j}(|m_i|+|n_i|-K|m_i+n_i|)2^i
                                           & \text{by \eqref{eq:4}}\\
                        & \geq 4(2^j-1)-(K-1)4(2^k-1) & \\
                        &\quad\quad\quad+\sum_{i=0}^k-(K-1)2^i(|m_i|+|n_i|)
                                           & \text{by tri.\ ineq.}\\
                          &\geq 4(2^j-1)-(K-1)4(2^k-1) &\\
                        &\quad\quad\quad-(K-1)\sum_{i=0}^k2^iL &
                                                                 \text{by \eqref{eq:1}}\\
                        &\quad\quad\quad-(K-1)(2^{k+1}+2^{j+1}) &\\
                        &=4(2^j-1)-(K-1)\left(4(2^k-1)+L(2^{k+1}-1)+2^{k+1}+2^{j+1}\right) & \\
  &=2^{j+2}-(K-1)\left(2^{k+2}+L(2^{k+1}-1)+2^{k+1}+2^{j+1}\right)-4(2-K) &
\end{align*}
To make this expression positive, it suffices to have:
\begin{equation}
  \label{eq:6}
  1-\frac{2-K}{2^j}>2^{k-j}(K-1)\left(1+\frac{L}{2}+2^{-1}+2^{j-1-k}\right)
\end{equation}
The left hand side is at least $1/2$, and the right
hand side is bounded above by:
\[(1/2)\cdot 2^{k-j}(K-1)(L+4)\]
Thus, applying the bound on $k-j$ from the negation of \eqref{eq:5},
we see that \eqref{eq:6} is satisfied for 
\[
K<1+\frac{1}{2^{k-j}(L+4)}\leq 1+\frac{1}{2(L+4)}\cdot\left(\frac{R(L+4)}{L-2}\right)^{-1/\alpha} 
\]

This completes the proof when $g=a$.

If $g\in\{x,y\}$ then $|g^p|=|a^p|+2$ (by Lemma~\ref{lem:ag_bound}), so $|g^m|+|g^n|\leq K|g^{m+n}|$
implies:
\[|a^m|+|a^n|=|g^m|+|g^n|-4\leq K|g^{m+n}|-4=K|a^{m+n}|+2K-4\leq K|a^{m+n}|\]
Thus, the result follows from the result for $a$.
\end{proof}

\begin{remark} \label{rem:duh}
 In Lemma~\ref{lem:triangle_ineq_enhanced} we provide a condition to
 show that two quantities, let us call them $A$ and $B$, are related by $A/B > R$ for some $R> 0$.
 We will frequently pass without comment between equivalent
 formulations of this relation that we note here to avoid any
 confusion: \[A/B>R\iff A > RB \iff A > \frac{R}{R+1}(A+B) \iff B < \frac{1}{R+1}(A+B)\]
\end{remark}

Next, we consider multiple escapes, all of the same flavor and with
all exponents positive:
\begin{corollary}\label{cor:thepositiveone}
  For all $R>1$ there exist $J>1$ and $1<K<2$ such that if $g \in \{ a, x,y\}$ and
  $m_1 \ge m_2 \ge \cdots \ge m_k > 0$ and $\sum_{i=1}^k m_i \ge J$
  and
  \[ \sum_{i=1}^k|g^{m_i}|\leq K|g^{\sum_{i=1}^k m_i}|\] then
  $m_1 > \frac{R}{R+1}\sum_{i=1}^k m_i$.
\end{corollary}

\begin{proof}
  Replace $R$ by $\max\{R,5\}$.  Observe that the result for the new
  $R$ implies the result for the original $R$.
  Let $J':=1+(L+3)/2$ as in Lemma~\ref{lem:triangle_ineq_enhanced},
  let $J:=J'(R+1)$,
  and let $K$ be
  a constant satisfying Lemma~\ref{lem:triangle_ineq_enhanced} with
  respect to $R$.
  By Lemma~\ref{lem:triangle_ineq_enhanced}, we have one of the
  following possibilities:
  \begin{itemize}
  \item $\frac{m_1}{m_2 + m_3 + \ldots + m_k} > R$ 
\item $m_2 + m_3 + \ldots + m_k < J'$ 
\item $\frac{m_2 + m_3 + \ldots + m_k}{m_1} > R$ 
  \item$m_1 < J'$
  \end{itemize}
  
    Since
  $\sum_{i=1}^k m_i \ge J$, either of the first two conditions imply
  $m_1 > \frac{R}{R+1}\sum_{i=1}^k m_i$.
  The disjunction of the last two conditions implies
  $m_2 + m_3 + \cdots + m_k > \frac{R}{R+1}\sum_{i=1}^k m_i$.  Thus it
  suffices to rule out this latter possibility.

  Suppose $m_2 + m_3 + \cdots + m_k > \frac{R}{R+1}\sum_{i=1}^k m_i$.
  Let $\ell$ be least such that
  $\sum_{i=1}^{\ell} m_i \ge \frac{1}{2} \sum_{i=1}^k m_i$.  Then, for
  all $i$, we have $m_i \le m_1 < \frac{1}{R} \sum_{i=1}^k m_i$ and so
  \[ \frac{1}{2} \sum_{i=1}^k m_i \le \sum_{i=1}^{\ell} m_i <
    \frac{1}{2} \sum_{i=1}^k m_i + m_{\ell} <
    \frac{R+2}{2R}\sum_{i=1}^k m_i \] holds. 
    Then  $\sum_{i=1}^{\ell} m_i \ge \frac{1}{2} \sum_{i=1}^k m_i \ge
    J/2>J'$ and
  $\sum_{i=\ell+1}^k m_i > \frac{R-2}{2R} \sum_{i=1}^k m_i \ge
  \frac{R-2}{2R}J>J'$, since $R\geq 5$, so by
  Lemma~\ref{lem:triangle_ineq_enhanced}, either (I)
  $\sum_{i=1}^{\ell} m_i > R \sum_{i=\ell+1}^k m_i$ or (II)
  $\sum_{i=\ell+1}^k m_i > R \sum_{i=1}^{\ell} m_i$.  We will show
  that either case results in a contradiction.  In case (I) we have
  \[ \frac{R+2}{2R}\sum_{i=1}^k m_i > \sum_{i=1}^{\ell} m_i > R
    \sum_{\ell+1}^k m_i > \frac{R(R-2)}{2R}\sum_{i=1}^k m_i \] which
  contradicts $R \ge 5$.  In case (II) we have
  \[ \frac{1}{2}\sum_{i=1}^k m_i \ge \sum_{i=\ell+1}^k m_i > R
    \sum_1^{\ell} m_i \ge \frac{R}{2}\sum_{i=1}^k m_i \] which again
  contradicts $R \ge 5$.
\end{proof}

Next we will generalize to the case that $\gamma$ consists of escapes of the same
flavor, this time allowing both positive and negative exponents. First, there is an
auxiliary result, Corollary~\ref{cor:thecomplicatedone}, that says
either the positive exponents or the negative exponents dominate.
Then, in Corollary~\ref{cor:one_long_escape}, we show that a single
escape dominates.
\begin{corollary}\label{cor:thecomplicatedone}
  For all $R >1$ there exist $J>1$
  such that for any sufficiently small $1<K<2$, independent
  of $J$, if
  $g \in \{ a, x, y \}$  and
  $\sum_{i=1}^k |m_i| \geq J$ and 
  \[  \sum_{i=1}^k|g^{m_i}| \leq K|g^{\sum_{i=1}^k m_i}| \] then
  for  $I = \{i : m_i > 0 \}$ and for one of $I^*\in\{I,I^c\}$ we have
  \begin{center}
    \[ \biggl|\sum_{i \in I^*} m_i\biggr| > \frac{R}{R+1}\sum_{i=1}^k
      |m_i| \] and
    \[ \sum_{i\in
        I^*}|g^{m_i}| < \bigl(K + (K-1) CR^{(-1/\alpha)}\bigr) \bigl|g^{\sum_{i\in I^*} m_i}\bigr| \]
  \end{center}
\end{corollary}
\begin{proof}
Let $J':=1+(L+3)/2$ and $J:=J'(R+1)$ as in the previous corollary. 
  By Lemma~\ref{lem:triangle_ineq_enhanced}, if $K > 1$ is sufficiently small, then whenever
  $K|g^{m+n}| \geq |g^m| +|g^n|$  either we have
  $$|n| > \frac{R}{R+1}\bigl(|m|+|n|\bigr) \; \textrm{ or } \; |m| < J' \; \textrm{ or } \; |m| > \frac{R}{R+1}\bigl(|m|+|n|\bigr) \; \textrm{ or } \; |n| < J'.$$
  When $|m| + |n| \geq J$, the second condition implies the first
  condition and the fourth condition implies the third.
  So, if
  \[ K|g^{\sum_{i=1}^k m_i}| \geq \sum_{i=1}^k|g^{m_i}| \] and
  $\sum_{i=1}^k |m_i| \geq J$ then, since
  \[ \sum_{i=1}^k|g^{m_i}| = \sum_{i\in I}|g^{m_i}| + \sum_{i\notin
      I}|g^{m_i}| \ge |g^{\sum_{i\in I} m_i}| + |g^{\sum_{i\notin I}
      m_i}| \] we can set $n = \sum_{i\in I} m_i$ and
  $m = \sum_{i\notin I} m_i$ to obtain one of the two following
  conditions.
    \[ \biggl|\sum_{i \in I} m_i\biggr| > \frac{R}{R+1}\sum_{i=1}^k
      |m_i| \]
    \[ \biggl|\sum_{i \notin I} m_i\biggr| > \frac{R}{R+1}\sum_{i=1}^k
      |m_i| \] By symmetry, the first condition holds without loss of
    generality so it remains to prove
    \[ \bigl(K + (K-1) CR^{(-1/\alpha)}\bigr) |g^{\sum_{i\in I} m_i}| > \sum_{i\in
        I}|g^{m_i}| \] 
    We
    have:
    \begin{align*}
      |g^{\sum_{i \in I}m_i}|
      &\ge |g^{\sum_{i=1}^k m_i}| - |g^{\sum_{i \notin I}m_i}| \\
      &\ge \frac{1}{K}\sum_{i\in I}|g^{m_i}| + \frac{1}{K}\sum_{i\notin I}|g^{m_i}|
        - |g^{\sum_{i \notin I}m_i}| \\
      &\ge \frac{1}{K}\sum_{i\in I}|g^{m_i}| + \frac{1}{K}|g^{\sum_{i\notin I} m_i}|
        - |g^{\sum_{i \notin I}m_i}| \\
      &= \frac{1}{K}\sum_{i\in I}|g^{m_i}|
        - \frac{K - 1}{K} \cdot |g^{\sum_{i \notin I}m_i}|
    \end{align*}
    So:
    \begin{align*}
      \sum_{i\in I}|g^{m_i}|
      &\le K|g^{\sum_{i \in I}m_i}| +  (K-1)|g^{\sum_{i \notin I}m_i}| \\
      &\le \biggl(K +
        (K-1)\cdot\frac{|g^{\sum_{i \notin I}m_i}|}{|g^{\sum_{i \in I}m_i}|}\biggr)
        |g^{\sum_{i \in I}m_i}|
    \end{align*}
    Finally, since we know from Theorem~\ref{powerdistortion}  that the
    distortion of $\langle g\rangle$ obeys a power law, $\bigl|\sum_{i \notin I}m_i\bigr| < \bigl|\sum_{i \in
      I} m_i\bigr|/R$ implies $\frac{|g^{\sum_{i \notin I}m_i}|}{|g^{\sum_{i \in
          I}m_i}|} < CR^{(-1/\alpha)}$.
\end{proof}

\begin{corollary}\label{cor:one_long_escape}
  For all $R>1$ there exist $J,\,K>1$ such that if
  $g \in \{ a, x, y \}$ and $\sum_{i=1}^k |m_i| \ge J$ and
  \[ \sum_{i=1}^k|g^{m_i}|  \leq K|g^{\sum_{i=1}^k m_i}| \] then
  $\max_i|m_i| > \frac{R}{R+1}\sum_{i=1}^k |m_i|$.
\end{corollary}
\begin{proof}
 We may assume $0<m_1=\max_i m_i$ and $0>m_k=\min_im_i$.
  Let $I:=\{i\mid m_i>0\}.$
We will let $R'\geq 5$ be large enough that
$\big(\frac{R'}{R'+1}\big)^2>\frac{R}{R+1}$.
Let $J_{\ref{cor:thecomplicatedone}}$,
$K_{\ref{cor:thecomplicatedone}}$, and
 $J_{\ref{cor:thepositiveone}}$
and $K_{\ref{cor:thepositiveone}}$ be
the numbers provided by Corollaries~\ref{cor:thecomplicatedone} and \ref{cor:thepositiveone} 
respect to $R'$. (Actually, $J_{\ref{cor:thecomplicatedone}}=J_{\ref{cor:thepositiveone}}$.)
We claim it suffices to take any $J\geq
 \frac{R'+1}{R'}J_{\ref{cor:thepositiveone}}> J_{\ref{cor:thecomplicatedone}}$
and any $1<K\leq \min\{K_{\ref{cor:thecomplicatedone}},\frac{K_{\ref{cor:thepositiveone}}+C(R')^{(-1/\alpha)}}{1+C(R')^{(-1/\alpha)}}\}$.

To see this, we first apply
Corollary~\ref{cor:thecomplicatedone}.
Since $K\leq
K_{\ref{cor:thecomplicatedone}}$, $K$ is sufficiently small for Corollary~\ref{cor:thecomplicatedone}, so, since
\[
 \sum_{i=1}^k |m_i| \geq J> J_{\ref{cor:thecomplicatedone}} \;  \;
 \textrm{ and } \; \;  \sum_{i=1}^k |g^{m_i}| \leq K \big| g^{\sum_{i=1}^k m_i} \big|
\]
there is an $I^* \in \{ I, I^c\}$ such that:
\begin{center}
    \[ \biggl|\sum_{i \in I^*} m_i\biggr| > \frac{R'}{R'+1}\sum_{i=1}^k
      |m_i| \] and
    \[ (K + (K -1) C(R')^{(-1/\alpha)})|g^{\sum_{i\in I^*} m_i}\bigr| > \sum_{i\in
        I^*}|g^{m_i}| \]
  \end{center}

  Now,
  $K\leq\frac{K_{\ref{cor:thepositiveone}}+C(R')^{(-1/\alpha)}}{1+C(R')^{(-1/\alpha)}}$
  if and only if
  $(K + (K -1) C(R')^{(-1/\alpha)})\leq K_{\ref{cor:thepositiveone}}
  $, so
  $ K_{\ref{cor:thepositiveone}} \big| g^{\sum_{i \in I^*} m_i} \big|
  \geq \sum_{i \in I^*} |g^{m_i}|$.  Furthermore:
  \[\sum_{i\in I^*} |m_i| =\biggl|\sum_{i \in I^*} m_i\biggr| >
    \frac{R'}{R'+1}\sum_{i=1}^k|m_i| \geq  \frac{R'}{R'+1}J\geq
    J_{\ref{cor:thepositiveone}}\]
  Thus, the hypotheses of Corollary
  ~\ref{cor:thepositiveone} for $I^*$ with respect to $R'$ are satisfied, and we get:
 \[\max_i|m_i|\geq \max_{i\in I^*}|m_i|>\frac{R'}{R'+1}\biggl|\sum_{i
     \in I^*}
   m_i\biggr|>\bigg(\frac{R'}{R'+1}\bigg)^2\sum_{i=1}^k|m_i|>\frac{R}{R+1}\sum_{i=1}^k|m_i|\qedhere\]
\end{proof}

In the previous results we picked out a `dominant' escape in terms of
exponents. In the next result we show it is still dominant in terms of length.
\begin{corollary}\label{cor:one_long_path}
  For all $R>1$ there exist $J,\,K>1$ such that if
  $g \in \{ a, x, y \}$ and
  $\gamma = \gamma_1\gamma_2 \cdots \gamma_k$ is a concatenation of
  paths such that
  \begin{enumerate}
  \item\label{itm:one_long_path:group} for all $i$, the endpoints of
    $\gamma_i$ differ by $g^{m_i}$,
  \item\label{itm:one_long_path:geom} $J \le |\gamma| \le K|g^{\sum_{i=1}^k m_i}|$
  \end{enumerate}
  then $\max_i|\gamma_i| > \frac{R}{R+1} \cdot |\gamma|$.
\end{corollary}
\begin{proof}
  Let $R_0 > 1$ and let $J_0, K_0 > 1$ be such that the conclusion of
  Corollary~\ref{cor:one_long_escape} holds for
  $(R,J,K) = (R_0, J_0, K_0)$.  Suppose
  $\gamma = \gamma_1\gamma_2 \cdots \gamma_k$ satisfies conditions
  (\ref{itm:one_long_path:group}) and (\ref{itm:one_long_path:geom})
  for some $K \in (1,K_0)$, some $J > J_0^{1/\alpha}KC$ and some
  $g \in \{a,x,y\}$.  Then we have
  \[ K|g^{\sum_{i=1}^k m_i}| \ge |\gamma| = \sum_{i=1}^k|\gamma_i| \ge
    \sum_{i=1}^k|g^{m_i}| \] and
  \[ J_0^{1/\alpha} < \frac{J}{KC} \le
    \frac{|\gamma|}{KC} \le
    \frac{|g^{\sum_{i=1}^k m_i}|}{C} \le
    \biggl|\sum_{i=1}^k m_i\biggr|^{1/\alpha} \le
    \biggl(\sum_{i=1}^k |m_i|\biggr)^{1/\alpha} \] so, by
  Corollary~\ref{cor:one_long_escape}, we have
  $\max_i|m_i| > \frac{R_0}{R_0+1} \sum_{i=1}^k|m_i|$.  If $i=\iota$
  maximizes $|m_i|$ then this inequality is equivalent to
  $\sum_{i \neq \iota} |m_i| < \frac{1}{R_0} \cdot |m_{\iota}|$ and we have
  \begin{align*}
    \frac{1}{C} \cdot |g^{\sum_{i \neq \iota} m_i}|
    &\le \Bigl|\sum_{i \neq \iota} m_i\Bigr|^{\frac{1}{\alpha}} \\
    &\le \Bigl(\sum_{i \neq \iota} |m_i|\Bigr)^{\frac{1}{\alpha}} \\
    &< \frac{1}{R_0^{1/\alpha}} \cdot
      |m_{\iota}|^{\frac{1}{\alpha}} \\
    &\le \frac{1}{R_0^{1/\alpha}} \cdot |g^{m_{\iota}}| \\
    &\le \frac{1}{R_0^{1/\alpha}} \cdot
    \bigl(|g^{\sum_{i \neq \iota} m_i}| + |g^{\sum_{i = 1}^k m_i}| \bigr)
  \end{align*}
  from which we can obtain
  \[ |g^{\sum_{i \neq \iota} m_i}| < \frac{C}{R_0^{1/\alpha} - C}
    \cdot |g^{\sum_{i = 1}^k m_i}| \] as long as
  $R_0 > C^{\alpha}$.  Thus
  \[ |g^{\sum_{i = 1}^k m_i}| \le |g^{m_1}| + |g^{\sum_{i \neq \iota}
      m_i}| < |\gamma_{\iota}| + \frac{C}{R_0^{1/\alpha} - C} \cdot
    |g^{\sum_{i = 1}^k m_i}| \] and so
  \[ \max_i|\gamma_i| \ge |\gamma_{\iota}| > \frac{R_0^{1/\alpha} - 2C}{R_0^{1/\alpha} - C}
    \cdot |g^{\sum_{i = 1}^k m_i}| \ge \frac{R_0^{1/\alpha} -
      2C}{K(R_0^{1/\alpha} - C)} \cdot |\gamma| \] and the factor
  $\frac{R_0^{1/\alpha} - 2C}{K(R_0^{1/\alpha} - C)}$ tends to $1$ as
  $K$ tends to $1$ and $R_0$ tends to infinity.

  To complete the proof of the corollary, we need to explain how we
  obtain the $J, K > 1$ of the corollary statement.  To obtain
  $J, K > 1$ we pick $R_0 > C^{\alpha}$ large enough that
  $\frac{R_0^{1/\alpha} - 2C}{(R_0^{1/\alpha} - C)} > \frac{R}{R+1}$.
  Next, we obtain $K_0, J_0 > 1$ such that the conclusion of
  Corollary~\ref{cor:one_long_escape} holds for
  $(R,J,K) = (R_0, J_0, K_0)$.  Finally, we choose $K \in (1,K_0)$ so
  that
  $\frac{R_0^{1/\alpha} - 2C}{K(R_0^{1/\alpha} - C)} > \frac{R}{R+1}$
  and then choose $J > J_0^{1/\alpha}KC$.
\end{proof}

We will now start allowing escapes of different flavors.
The first result says if there is exactly one escape of each flavor
and the $a$--escape is not tiny then it dominates.
\begin{lemma}
  \label{lem:compatible_escapes} For all $R > 1$ there exist $J, K > 1$
  such that if $|\ell| + |m| + |n| \ge J$ and
  \[ |a^{\ell}| + |x^m| + |y^n| \leq K|a^{\ell} x^m y^n| \] and
  \[|\ell| > \frac{|\ell|+|m|+|n|}{R} \] then
  $|m|,|n| \le \frac{|\ell|+|m|+|n|}{R}$.
\end{lemma}
\begin{proof}
The result for large $R$ yields $J$ and $K$ for which the same result
is true for all smaller $R$, so we may
assume in the proof that $R>2$.
 We can multiply all three exponents by $-1$ without changing the
 hypothesized inequalities, so we may assume $\ell>0$.
 Further, by the $x$--$y$ symmetry we may assume $|n|\leq |m|$.

There exist integers $p>0$ and $q$ with $|q|\leq L/2$ such that
$\ell=pL+q$ and $|a^\ell|=|q|+|a^{pL}|$, by (\ref{item:planelength}) of Lemma~\ref{lem:ag_bound}.  
The strategy for the proof is to compare the approximation of $|a^{\ell} x^m y^n|$ by $|a^{\ell}| + |x^m| + |y^n|$ (illustrated by the solid line in Figure~\ref{fig:compatible_escapes}), to the approximation given by $|a^q| + |x^{m +p}| +|y^{n+p}|$ (illustrated by the dotted line in Figure~\ref{fig:compatible_escapes}).

\begin{figure}[h!]
  \labellist
  \pinlabel $a^q$  at 26 42
  \pinlabel $x^{p+m}$  at 99 109
  \pinlabel $y^{p+n}$ [l] at 223 112
  \pinlabel $a^\ell$ at 86 14
  \pinlabel $x^m$ at 193 59
  \pinlabel $y^n$   at 235 58 
  \endlabellist
  \centering
  \includegraphics[height=5cm]{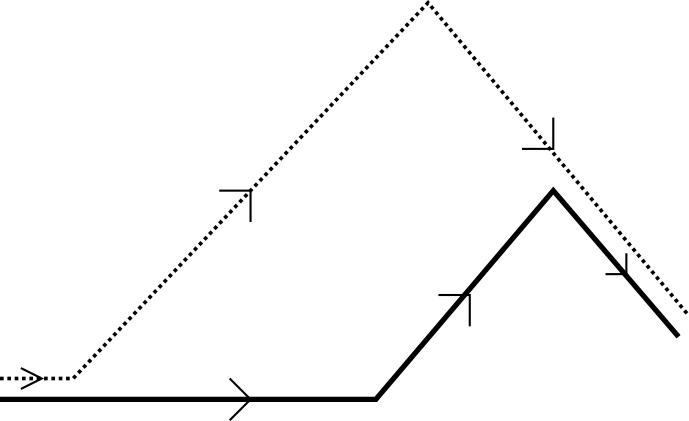}
  \caption{Illustration of the two decompositions of $a^{\ell} x^m y^n$.}
  \label{fig:compatible_escapes}
\end{figure}
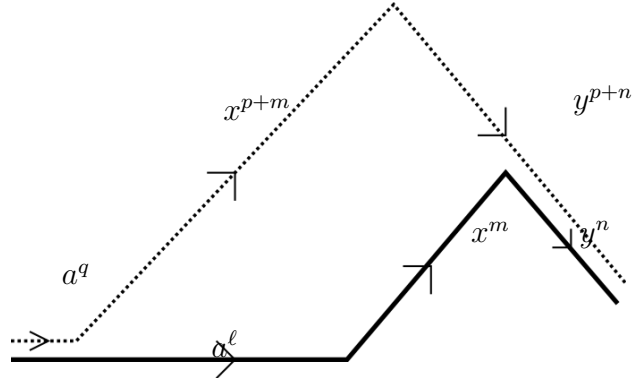


Notice that if
$\ell\geq L+1$
then we can verify the following bounds:
\begin{equation} \label{eq:p_l_ratio}
\frac{2}{L+2}\cdot\frac{1}{L}\leq\frac{p}{\ell}\leq\frac{L+2}{2}\cdot\frac{1}{L}
\end{equation}

Let $J'=(L+5)/2$ as in Lemma~\ref{lem:triangle_ineq_enhanced}.  Let
$J:=R(1+LJ')$.  Let $R':=(1/2)L(L+2)(R-1)$, and let $K'$ be as in
Lemma~\ref{lem:triangle_ineq_enhanced} with respect to $R'$.  We will
choose $K > 1$ to be sufficiently small that the below inequalities
labelled~\eqref{eq:future_1} and~\eqref{eq:future_2} are contradicted,
and the below inequality labelled~\eqref{eq:future_3} is satisfied.

Suppose that $|m|>\frac{\ell+|m|+|n|}{R}$. We will derive a contradiction.
First observe that we have the bounds
\begin{equation} \label{eq:l_m_ratio}
1/(R-1)<\ell/|m|<R-1. 
\end{equation}
Thus, by combining equations~\eqref{eq:p_l_ratio} and~\eqref{eq:l_m_ratio} we deduce that 

\begin{equation} \label{eq:p_m_ratio}
\frac{1}{R'}<\frac{p}{|m|}<\frac{(L+2)(R-1)}{2L}<\frac{L(L+2)(R-1)}{2}=R'
\end{equation}

As both $|m|$ and $p$ are larger than $J'$, we can combine
Lemma~\ref{lem:triangle_ineq_enhanced} with Equation~\eqref{eq:p_m_ratio} to see that 
\begin{equation} \label{eq:Kprime_bound}
|x^m|+|x^p|>K'|x^{m+p}|. 
\end{equation}

\noindent We now deduce that 
\begin{align*}
|q|+|x^p|+|x^m|+|y^p|+|y^n| &\leq |a^l|+|x^m|+|y^n| & \textrm{By~\ref{item:planelength} of Lemma~\ref{lem:ag_bound}.}\\ 
                            & \leq K|a^\ell x^my^n| & \textrm{By assumption.} \\ 
                            & = K(|a^qx^{m+p}y^{n+p}|) & \\
                            & \leq K|q|+K(|x^{p+m}|+|y^{p+n}|). & \textrm{Triangle inequality}.
\end{align*}
  Combining this with~\eqref{eq:Kprime_bound} yields
\begin{equation}
  \label{eq:7}
  (1-\frac{K}{K'})(|x^p|+|x^m|)+|y^p|+|y^n|-K|y^{p+n}|< (K-1)L/2
\end{equation}

First, suppose that either $|n|<J'$ or
$|n|<p/R'$.
Both of these conditions imply $|n|<p$, so by
Theorem~\ref{powerdistortion}  we have
$|y^n|<C|n|^{1/\alpha}<Cp^{1/\alpha}<C|y^p|$.
Continuing from \eqref{eq:7}, we have:
\begin{align*}
  (K-1)L/2&>
            (1-\frac{K}{K'})(|x^p|+|x^m|)+|y^p|+|y^n|-K|y^{p+n}|\\
          &\geq (1-\frac{K}{K'})(|x^p|+|x^m|)-(K-1)(|y^p|+|y^n|)\\
          &\geq (1-\frac{K}{K'})(|x^p|+|x^m|)-(K-1)|y^p|(1+C)\\
  &= (1-\frac{K}{K'}-(K-1)(1+C))|x^p|+(1-\frac{K}{K'})|x^m|
\end{align*}

To ensure that the coefficients in this final expression are non-negative we will have assumed that $K$ is chosen to satisfy the inequalities:
\begin{equation} \tag{$\heartsuit$} \label{eq:future_3}
1< K < \min\{K',\,  1+\frac{K'-1}{(C+1)K'+1}\}
\end{equation}
%
%
By Theorem~\ref{powerdistortion}, the lower bounds on $p$ and $|m|$ imply
$|x^p|,|x^m|>(J')^{1/\alpha}$, so:
\[(K-1)L/2>\left(1-\frac{K}{K'}-(K-1)(1+C)+1-\frac{K}{K'}\right)(J')^{1/\alpha}\]
Solving for $K$ gives us:
\[ K   > \frac{LK'+2(J')^{1/\alpha}\cdot (K'(1+C)+2K')}{LK' +
      2(J')^{1/\alpha}\cdot (K'(1+C) + 2)} >1\tag{$\clubsuit$} \label{eq:future_1}\]
Having assumed that $K > 1$ is chosen to be small enough that~\eqref{eq:future_1} does not hold, we must in fact have $|n|\geq J'$ and $|n|\geq p/R'$.
Then Lemma~\ref{lem:triangle_ineq_enhanced} implies $|y^p|+|y^n|> K'|y^{p+n}|$.
In this case, continuing from \eqref{eq:7} we have:
\begin{align*}
  (K-1)L/2&>
            (1-\frac{K}{K'})\left(|x^p|+|x^m|+|y^p|+|y^n|\right)\\
          &> (1-\frac{K}{K'})(4(J')^{1/\alpha})
\end{align*}
Solving for $K$ gives 
\begin{equation} \tag{$\diamondsuit$} \label{eq:future_2} 
 K > \frac{8K'(J')^{1/ \alpha} + LK'}{8(J')^{1/\alpha} + LK'} > 1. 
\end{equation}
By having assumed that $K >1$ is chosen small enough that~\eqref{eq:future_2} does not hold, this cannot happen.
%
Thus, for $J:=R(1+LJ')$ and $K$ as specified, the hypotheses imply $|n|\leq |m|\leq \frac{|\ell|+|m|+|n|}{R}$.
\end{proof}

The next result, Lemma~\ref{biglemma} is a stepping stone. It says
that if we have a mixture of different flavors of escapes, then
whenever one flavor makes up a nontrivial amount of the total, then there is an escape of that flavor that is dominant \emph{among
the other escapes of that flavor}. Moreover, if the $a$--escapes make
up a nontrivial amount of the total then the $x$--escapes and
$y$--escapes are negligible. 

\begin{lemma}\label{biglemma}
For all $R > 1$ there exist $J, K > 1$
  such that if
  $T = \sum_i |\ell_i| + \sum_j |m_j| + \sum_k |n_k| \ge J$ and
  \begin{equation}
    \label{eq:optimal}
    \sum_i
    |a^{{\ell}_i}| + \sum_j |x^{m_j}| + \sum_k |y^{n_k}|\leq
    K|a^{\sum_i {\ell}_i}x^{\sum_j m_j}y^{\sum_k n_k}|  
  \end{equation}
  then the following implications hold.

  \[ \sum_i |\ell_i| > \frac{1}{R} \cdot T \;\; \Longrightarrow \;\;
    \max_i|\ell_i| > \frac{R-1}{R}\sum_i |\ell_i|\]
  \[ \sum_j |m_j| > \frac{1}{R} \cdot T \;\; \Longrightarrow \;\;
    \max_j|m_j| > \frac{R-1}{R}\sum_j |m_j|\]
  \[ \sum_k |n_k| > \frac{1}{R} \cdot T \;\; \Longrightarrow \;\;
    \max_k|n_k| > \frac{R-1}{R}\sum_k |n_k|\]

  Moreover, if $\sum_i |\ell_i| > \frac{1}{R} \cdot T$ then
  $\sum_j |m_j| \le \frac{1}{R} \cdot T$ and
  $\sum_k |n_k| \le \frac{1}{R} \cdot T$.
\end{lemma}
\begin{proof}
  We will prove the first implication.  The remaining implications are
  proved similarly.  
  Assume that \eqref{eq:optimal} and the left-hand side of the first implication holds with:
  \begin{equation}
    \label{eq:k_bound}
    1< K < 1 + \Bigl(\frac{1}{2}\Bigr)^{1+\frac{\alpha-1}{\alpha}}
    \cdot \frac{1}{CR^{1/\alpha}}
  \end{equation}

  \begin{claim}
    \label{claim:inequality}
    $|x^{\sum_j m_j} y^{\sum_k n_k}| <
    2^{1-1/\alpha}C\Bigl(\sum_j|m_j|
    + \sum_k|n_k|\Bigr)^{1/\alpha}$
  \end{claim}
  \begin{subproof}[Proof of claim]
    \begin{align*}
      \Bigl(\sum_j|m_j| +
        \sum_k|n_k|\Bigr)^{1/\alpha}
      &\ge \biggl(\Bigl|\sum_j m_j\Bigr| + \Bigl|\sum_k
        n_k\Bigr|\biggr)^{1/\alpha} & \\
      &\ge \Bigl(\frac{1}{2}\Bigr)^{\frac{\alpha-1}{\alpha}}
        \biggl(\Bigl|\sum_j m_j\Bigr|^{1/\alpha}
        + \Bigl|\sum_k n_k\Bigr|^{1/\alpha}\biggr) &
         \textrm{Fact~\ref{fact:reverse_holder}}\\
      &> \frac{1}{C}\Bigl(\frac{1}{2}\Bigr)^{\frac{\alpha-1}{\alpha}}
        \bigl(|x^{\sum_j m_j}| + |y^{\sum_k n_k}|\bigr) &
        \textrm{Theorem~\ref{powerdistortion}} \\
      &\ge \frac{1}{C}\Bigl(\frac{1}{2}\Bigr)^{\frac{\alpha-1}{\alpha}}
        |x^{\sum_j m_j} y^{\sum_k n_k}| &       \qedhere 
    \end{align*}
  \end{subproof}
  
  \begin{claim}
    \label{claim:ell_bound}
    $|\sum_i \ell_i| > \frac{1}{(4C)^{\alpha}R} \cdot T$
  \end{claim}
  \begin{subproof}[Proof of claim]
    For the sake of finding a contradiction, assume
    $|\sum_i \ell_i| \le \frac{1}{(4C)^{\alpha}R} \cdot T$.  Then,
    since $T < R \sum_i|\ell_i|$, we have
    $|\sum_i \ell_i| < \frac{1}{(4C)^{\alpha}} \sum_i |\ell_i|$.
    Then we have
    \begin{align*}
    |a^{\sum_i \ell_i}| & < C \Big|\sum_i \ell_i \Big|^{1/\alpha} & \textrm{ by Theorem~\ref{powerdistortion}} \\
                        & < C\bigl(\frac{1}{(4C)^{\alpha}} \sum_i |\ell_i|\bigr)^{1/\alpha}
                        & \textrm{ by the above }\\
                        & \le \frac{1}{4} \sum_i |\ell_i|^{1/\alpha} 
                        & \textrm{by Fact~\ref{fact:reverse_holder}}\\
                        & < \frac{1}{4} \sum_i |a^{\ell_i}| 
                        & \textrm{ by Theorem~\ref{powerdistortion}}\\
    \end{align*}
    thus
    $\sum_i|a^{\ell_i}| + |x^{\sum_j m_j}y^{\sum_k n_k}| \le
    \sum_i|a^{\ell_i}| + \sum_j|x^{m_j}| + \sum_k|y^{n_k}| \le
    K|a^{\sum_i \ell_i}x^{\sum_j m_j}y^{\sum_k n_k}| \le K|a^{\sum_i
      \ell_i}| + K|x^{\sum_j m_j}y^{\sum_k n_k}| < \frac{K}{4} \sum_i
    |a^{\ell_i}| + K|x^{\sum_j m_j}y^{\sum_k n_k}|$ 
    which implies
    \[ \frac{1}{2}\sum_i|a^{\ell_i}| < (K-1)|x^{\sum_j m_j}y^{\sum_k
        n_k}| \] since \eqref{eq:k_bound} implies $K < 2$. 
        
    On the other hand, by Theorem~\ref{powerdistortion},
    Claim~\ref{claim:inequality} and Fact~\ref{fact:reverse_holder},
    we have
    \begin{align*}
    \sum_i|a^{\ell_i}|  & > \sum_i|\ell_i|^{1/\alpha} & \textrm{by Theorem~\ref{powerdistortion}} \\
                        & \ge \bigl(\sum_i|\ell_i|\bigr)^{1/\alpha} & \textrm{ by Fact~\ref{fact:reverse_holder}} \\
                        & > \frac{1}{R^{1/\alpha}} \cdot T^{1/\alpha} & \textrm{ by assumption}  \\
                        & \ge \frac{1}{R^{1/\alpha}}\bigl(\sum_j|m_j| + \sum_k|n_k|\bigr)^{1/\alpha} &  \\
                        & \ge \frac{1}{CR^{1/\alpha}}\bigl(\frac{1}{2}\bigr)^{\frac{\alpha-1}{\alpha}} |x^{\sum_j m_j} y^{\sum_k n_k}| &
                        \textrm{by Claim~\ref{claim:inequality}}
    \end{align*}
    thus
    \[ \frac{1}{2}\sum_i|a^{\ell_i}| < (K-1)|x^{\sum_j m_j}y^{\sum_k
        n_k}| < 2^{1-1/\alpha}CR^{1/\alpha}(K-1)
      \sum_i|a^{\ell_i}| < \frac{1}{2} \sum_i|a^{\ell_i}| \] where the
    last inequality follows from \eqref{eq:k_bound}.  This is a
    contradiction.
  \end{subproof}
  Then  we have:
  \begin{align*}
  |a^{\sum_i \ell_i}|  & > |\sum_i \ell_i|^{1/\alpha} & \textrm{by Theorem~\ref{powerdistortion}}\\
                       & > \frac{1}{4CR^{1/\alpha}} \cdot T^{1/\alpha} & \textrm{Claim~\ref{claim:ell_bound}} \\
                       & \ge \frac{1}{4CR^{1/\alpha}} \cdot \bigl(\sum_j |m_j| + \sum_k |n_k|\bigr)^{1/\alpha} & \textrm{by assumption} \\
                       & \ge \frac{1}{4C^2R^{1/\alpha}} \cdot
  \bigl(\frac{1}{2}\bigr)^{\frac{\alpha-1}{\alpha}}\cdot |x^{\sum_j m_j} y^{\sum_k n_k}| & \textrm{by Claim~\ref{claim:inequality}} \\
    \end{align*}
    This gives us:
  \[ |x^{\sum_j m_j} y^{\sum_k n_k}| <
    2^{3-1/\alpha}C^2 R^{1/\alpha}|a^{\sum_i
      \ell_i}|. \]
  On the other hand:
  \begin{align*}
    \sum_i|a^{\ell_i}| + |x^{\sum_j m_j}y^{\sum_k n_k}|
    &\le \sum_i|a^{\ell_i}| + \sum_j|x^{m_j}| + \sum_k|y^{n_k}| \\
    &\le K|a^{\sum_i \ell_j}x^{\sum_j m_j}y^{\sum_k n_k}| \\
    &\le K|a^{\sum_i \ell_j}| + K|x^{\sum_j m_j}y^{\sum_k n_k}|
  \end{align*}
  So:
  \begin{align*}
    \sum_i|a^{\ell_i}|
    &\le K|a^{\sum_i \ell_j}| + (K-1)|x^{\sum_j m_j}y^{\sum_k n_k}| \\
    &< K|a^{\sum_i \ell_j}|
      + 2^{3-1/\alpha}C^2 R^{1/\alpha}(K-1)|a^{\sum_i \ell_i}| \\
   \tag{$\dag$} \label{eq:K_restraint} &= \bigl(K
      + 2^{3-1/\alpha}C^2 R^{1/\alpha}(K-1)\bigr)|a^{\sum_i \ell_j}| \\
    &\stackrel{K\to 1}{\longrightarrow} |a^{\sum_i \ell_j}|.
  \end{align*}
  
  Noting that if $T \ge J$ then
  $\sum_i|\ell_i| > \frac{1}{R} \cdot T \ge \frac{J}{R}$.  Thus, by
  Corollary~\ref{cor:one_long_escape}, if $J$ is large enough and
  $K > 1$ is small enough that~\eqref{eq:K_restraint} satisfies the
  requisite assumption --- all depending only on $R$ and $L$ --- then
  $\max_i|\ell_i| > \frac{R-1}{R} \sum_i |\ell_i|$.  

  We now assume $\sum_i|\ell_i| > \frac{1}{R} \cdot T$ to prove the
  ``moreover part.''  We will show that
  $\sum_j|m_j| \le \frac{1}{R} \cdot T$.  A similar argument can be
  used to show that $\sum_k|n_k| \le \frac{1}{R} \cdot T$.  For the
  sake of finding a contradiction, assume that
  $\sum_j|m_j| > \frac{1}{R} \cdot T$.  Without loss of generality, we
  have $\max_i|\ell_i| = |\ell_1|$ and $\max_j|m_j| = |m_1|$.  Then
  \begin{equation} \label{eq:m1_and_l1_bounds}
  |\ell_1| > \frac{R-1}{R}\sum_i|\ell_i| > \frac{R-1}{R^2} \cdot T
   \; \textrm{ and } \; |m_1| > \frac{R-1}{R}\sum_j|m_j| > \frac{R-1}{R^2} \cdot T.
   \end{equation}

  \noindent Then, by Fact~\ref{fact:reverse_holder} and
  Theorem~\ref{powerdistortion}, we have
  \begin{align*}
    &|a^{\ell_1}| \\
    &> |\ell_1|^{1/\alpha} & \textrm{Theorem~\ref{powerdistortion}} \\
    &> \Bigl(\frac{R-1}{R^2}\Bigr)^{1/\alpha} \cdot T^{1/\alpha} & \textrm{Equation~\eqref{eq:m1_and_l1_bounds}}\\
    &\ge \Bigl(\frac{R-1}{R^2}\Bigr)^{1/\alpha} \cdot \biggl(\Bigl|\sum_{i\ge
      2}\ell_i\Bigr| + \Bigl|\sum_{j\ge 2}m_j\Bigr|
      + \Bigl|\sum_kn_k\Bigr|\biggr)^{1/\alpha} & \textrm{assumption} \\
    &\ge \Bigl(\frac{R-1}{R^2}\Bigr)^{1/\alpha} \cdot
      \Bigl(\frac{1}{3}\Bigr)^{\frac{\alpha-1}{\alpha}} \cdot
      \biggl(\Bigl|\sum_{i\ge 2}\ell_i\Bigr|^{1/\alpha} + \Bigl|\sum_{j\ge
      2}m_j\Bigr|^{1/\alpha} + \Bigl|\sum_kn_k\Bigr|^{1/\alpha}\biggr) & \textrm{Fact~\ref{fact:reverse_holder}} \\
    &> \Bigl(\frac{R-1}{R^2}\Bigr)^{1/\alpha} \cdot
      \Bigl(\frac{1}{3}\Bigr)^{\frac{\alpha-1}{\alpha}} \cdot
      \frac{1}{C} \cdot \bigl(|a^{\sum_{i\ge 2}\ell_i}| + |x^{\sum_{j\ge
      2}m_j}| + |y^{\sum_kn_k}|\bigr) & \textrm{Theorem~\ref{powerdistortion}}\\
    &\ge \Bigl(\frac{R-1}{R^2}\Bigr)^{1/\alpha} \cdot
      \Bigl(\frac{1}{3}\Bigr)^{\frac{\alpha-1}{\alpha}} \cdot
      \frac{1}{C} \cdot |a^{\sum_{i\ge 2}\ell_i} x^{\sum_{j\ge 2}m_j}
      y^{\sum_kn_k}|
  \end{align*}
  and, similarly, we have
  $|x^{m_1}| > \bigl(\frac{R-1}{R^2}\bigr)^{1/\alpha} \cdot
  \bigl(\frac{1}{3}\bigr)^{\frac{\alpha-1}{\alpha}} \cdot
  \frac{1}{C} \cdot |a^{\sum_{i\ge 2}\ell_i} x^{\sum_{j\ge 2}m_j}
  y^{\sum_kn_k}|$.
  
  By Lemma~\ref{lem:ag_bound}\eqref{item:ag_bound}, for sufficiently
  small $K'>1$ we have:
  \begin{equation} \label{eq:a_l1_x_m1_bound}
  |a^{\ell_1}x^{m_1}| >\frac{1}{K'}(|a^{\ell_1}|+|x^{m_1}|)= \bigl(\frac{R-1}{R^2}\bigr)^{1/\alpha}
    \cdot \bigl(\frac{1}{3}\bigr)^{\frac{\alpha-1}{\alpha}} \cdot
    \frac{2}{CK'} \cdot |a^{\sum_{i\ge 2}\ell_i} x^{\sum_{j\ge 2}m_j}
    y^{\sum_kn_k}| 
    \end{equation}
    
    But:
    \begin{align*}
  |a^{\ell_1}| + |x^{m_1}| + |a^{\sum_{i\ge 2}\ell_i} x^{\sum_{j\ge 
      2}m_j} y^{\sum_kn_k}| &\le \sum_i |a^{{\ell}_i}| + \sum_j 
                              |x^{m_j}| + \sum_k |y^{n_k}|\\
      &\le K|a^{\sum_i \ell_i} x^{\sum_j m_j}
        y^{\sum_kn_k}|\\
      &\le K|a^{\ell_1} x^{m_1}| + K|a^{\sum_{i \ge 2}
    \ell_i} x^{\sum_{j \ge 2} m_j} y^{\sum_k n_k}|    
    \end{align*}

    so that
  \begin{align*}
    |a^{\ell_1}| + |x^{m_1}|
    &\le K|a^{\ell_1} x^{m_1}|
      + (K-1)|a^{\sum_{i \ge 2} \ell_i} x^{\sum_{j \ge 2} m_j} y^{\sum_k n_k}| \\
    &< K|a^{\ell_1} x^{m_1}| +  3^{\frac{\alpha-1}{\alpha}}
      \Bigl(\frac{R^2}{R-1}\Bigr)^{1/\alpha}\cdot\frac{CK'}{2}(K-1)|a^{\ell_1} x^{m_1}| & \textrm{Equation~\eqref{eq:a_l1_x_m1_bound}}\\
    &= \biggl(K + 3^{\frac{\alpha-1}{\alpha}}
      \Bigl(\frac{R^2}{R-1}\Bigr)^{1/\alpha}
      \cdot\frac{CK'}{2}(K-1)\biggr)|a^{\ell_1} x^{m_1}| \\
    &\stackrel{K\to 1}{\longrightarrow} |a^{\ell_1}x^{m_1}|
  \end{align*}
  Also, if $T \ge J$ then
  $|\ell_1| + |m_1| > \frac{2(R-1)}{R^2} \cdot T \ge
  \frac{2(R-1)}{R^2} \cdot TJ$.  Thus, by
  Lemma~\ref{lem:compatible_escapes}, if $J$ is large enough and
  $K > 1$ is small enough---depending only on $R$, $K'$, $C$ and
  $L$---then $|m_1| < \frac{1}{R}|\ell_1| < \frac{1}{R} \cdot T$,
  which is a contradiction.  Since $K'$ and $C$ depend only on $L$, we
  are done.
\end{proof}

\begin{lemma}\label{lem:largeescape}
  If $R\geq 2$ and $|n_1|>\frac{R-1}{R}\sum_{i=1}^j|n_i|$ then:
  \[\left|\frac{\sum_{i=1}^jn_i}{n_1}-1\right|<\frac{1}{R-1}\]
  and
  \[\frac{R-2}{R}<\frac{|\sum_{i=1}^jn_i|}{\sum_{i=1}^j|n_i|}\leq1\]
\end{lemma}
\begin{proof}
  The first and last inequalities are easy. For the remaining
  inequality,  we have
  $|n_1|>(R-1)\sum_{i=2}^j|n_i|$, so:
    \[\frac{|\sum_{i=1}^jn_i|}{\sum_{i=1}^j|n_i|}\geq \frac{|n_1|-\sum_{i=2}^j|n_i|}{|n_1|+\sum_{i=2}^j|n_i|}>\frac{|n_1|-\frac{1}{R-1}|n_1|}{|n_1|+\frac{1}{R-1}|n_1|}=\frac{R-2}{R}\qedhere\]
\end{proof}

Now we put it all together. Proposition~\ref{longescapesinexponent}
and Theorem~\ref{longescapesinlength} say basically the same thing:
either there is a single dominant escape or there are two and they are
an $x$--escape and a $y$--escape. The difference in the two results is
that in Proposition~\ref{longescapesinexponent} `dominance' is in terms
of magnitude of exponent and in Theorem~\ref{longescapesinlength} it
is in terms of length.
\begin{proposition}\label{longescapesinexponent}
  Suppose $R\geq 3L+4$ and $J$ and $K$ are as in Lemma~\ref{biglemma} and:
  \[\sum_i
    |a^{{\ell}_i}| + \sum_j |x^{m_j}| + \sum_k |y^{n_k}| \leq K|a^{\sum_i {\ell}_i}x^{\sum_j m_j}y^{\sum_k n_k}|\]
  Suppose  $T:=\sum_i|\ell_i|+\sum_j|m_j|+\sum_k|n_k|\geq J$.

  \begin{enumerate}
    \item If  $\sum_i|\ell_i|>T/R$ or exactly one of $\sum_j|m_j|$ and
      $\sum_k|n_k|$ is strictly greater than $T/R$ then there is one
      maximal term whose absolute value is strictly greater than
      $\frac{(R-1)(R-2)}{R^2}T$.
      \item If $\sum_j|m_j|>T/R$ and $\sum_k|n_k|>T/R$ then $\max_j|m_j|+\max_k|n_k|>\left(\frac{R-1}{R}\right)^2T$.
  \item If  $x^z=a^{\sum_i\ell_i}x^{\sum_jm_j}y^{\sum_kn_k}$, then
$\sum_j|m_j|>T/R$
  and $\sum_i|\ell_i|,\,\sum_k|n_k|\leq T/R$.\label{casex}
  \item\label{casea}
     If  $a^z=a^{\sum_i\ell_i}x^{\sum_jm_j}y^{\sum_kn_k}$, then
     either:
     \begin{itemize}
       \item $\sum_i|\ell_i|>T/R$ and $\sum_j|m_j|,\,\sum_k|n_k|\leq T/R$, 
         or
       \item
       \begin{itemize}
         \item $\sum_j|m_j|,\,\sum_k|n_k|> T/R$,
         \item$\sum_i|\ell_i|\leq T/R$,
             \item
           $\max_j|m_j|,\,\max_k|n_k|>\frac{1}{2}\cdot \frac{(R-1)^2(R-2)}{R^3}\cdot
             T$, and
         \item
           $\left(\frac{R-2}{R+2}\right)^2<\frac{\max_j|m_j|}{\max_k|n_k|}<\left(\frac{R+2}{R-2}\right)^2$.
       
       \end{itemize}
     \end{itemize}
  \end{enumerate}
\end{proposition}
\begin{proof}
  The first case follows from Lemma~\ref{biglemma}, since if two of
  the sums are bounded above by $T/R$ then the third is bounded below
  by $\frac{R-2}{R}T$.

  The second case also follows immediately from Lemma~\ref{biglemma},
  since:
  \[\max_j|m_j|+\max_k|n_k|>\frac{R-1}{R}(\sum_j|m_j|+\sum_k|n_k|)\geq
  \frac{R-1}{R}(T-\sum_i|\ell_i|)\geq \frac{R-1}{R}(1-\frac{1}{R})T\]
  
  In Case~\eqref{casex}, we have $\sum_i\ell_i=pL$ and $z=p+\sum_jm_j$
  and $0=p+\sum_kn_k$.
  Suppose that $\sum_i|\ell_i|>T/R$.
  By Lemma~\ref{biglemma}, $$\max_i|\ell_i|>\frac{R-1}{R}\sum_k|\ell_k| \;\;\;\; \textrm{
and } \;\;\;\; \sum_j|m_j|,\,\sum_k|n_k|\leq T/R,$$ the latter implying
$\sum_i|\ell_i|\geq \frac{R-2}{R}T$.
Applying Lemma~\ref{lem:largeescape}:
\[\frac{T}{R} \geq \sum_k|n_k|
              \geq |\sum_kn_k|
              = |p|
              = \frac{|\sum_i \ell_i|}{|L|}
              > \frac{1}{L}\cdot\frac{R-2}{R}\cdot\sum_i|\ell_i|
              \geq \frac{1}{L}\cdot\left(\frac{R-2}{R}\right)^2T
              \]
This gives a contradiction if $R\geq 3L+4$.

  Suppose that $\sum_k|n_k|>T/R$.
  By Lemma~\ref{biglemma}, $\sum_i|\ell_i|\leq T/R$ and $\max_k|n_k|>\frac{R-1}{R}\sum_k|n_k|$ so, applying Lemma~\ref{lem:largeescape}:
  \[\frac{T}{R} \geq \sum_i|\ell_i|
                \geq|\sum_i\ell_i|=L|p|=L|\sum_kn_k|>L\cdot\frac{R-2}{R}\cdot\sum_k|n_k|>L\cdot\frac{R-2}{R}\cdot\frac{T}{R}\]
  This implies $R<\frac{2L}{L-1}\leq 12/5$, which is a contradiction.

  Since $R>3$, we are left with the possibility that only
  $\sum_j|m_j|$ is strictly greater than $T/R$.

  In Case~\eqref{casea}, if $\sum_i|\ell_i|>T/R$ then
  Lemma~\ref{biglemma} lets us revert to Case~(1), so assume
  $\sum_i|\ell_i|\leq T/R$.
  Without loss of generality, assume $\sum_j|m_j|\geq \sum_k|n_k|$, so
  that $\sum_j|m_j|\geq\frac{R-1}{2R}\cdot T>\frac{T}{R}$.
  By
  Lemma~\ref{biglemma}
  $\max_j|m_j|>\frac{R-1}{R}\sum_j|m_j|\geq\frac{1}{2}\cdot\left(\frac{R-1}{R}\right)^2\cdot
  T$.
  
  The constraint $a^z=a^{\sum_i\ell_i}x^{\sum_jm_j}y^{\sum_kn_k}$
  implies $\sum_j m_j=\sum_kn_k=(z-\sum_i\ell_i)/L$, so, applying Lemma~\ref{lem:largeescape}:
  \[\sum_k|n_k|\geq |\sum_kn_k|=|\sum_j m_j|\geq
    \frac{R-2}{R}\sum_j|m_j|\geq \frac{(R-1)(R-2)}{2R^2}\cdot T\]
  The given lower bound on $R$ is enough to imply that this quantity
  is greater than $T/R$, so  Lemma~\ref{biglemma} applies to give:
  \[\max_k|n_k|>\frac{R-1}{R}\sum_k|n_k|\geq
    \frac{(R-1)^2(R-2)}{2R^3}\cdot T\]
  It remains to estimate $\max_j|m_j|/\max_k|n_k|$:
  \begin{align*}
    \frac{\max_j|m_j|}{\max_k|n_k|}&\leq
                                     \frac{\sum_j|m_j|}{\frac{R-1}{R}\sum_k|n_k|}\\
    &\leq
                                     \frac{T-\sum_k|n_k|}{\frac{R-1}{R}\sum_k|n_k|}\\
                                   &\leq \frac{R}{R-1}\left(\frac{2R^2}{(R-1)(R-2)}-1\right)\\
    &<\left(\frac{R+2}{R-2}\right)^2
  \end{align*}
  The last inequality is true, in particular, when $R>2$.
  A similar computation gives the lower bound.
  \end{proof}

  \begin{thm}\label{longescapesinlength}
    For any $R'>1$ there exist $J'>1$ and $2>K>1$ such that if
    $\gamma$ is a $K$--biLipschitz path with endpoints in
    $\langle a,x,y\rangle$ 
    then either $|\gamma|\leq J'$ or $\gamma$ has either one or two long escapes
    that collectively account for at least $\frac{R'-1}{R'}$ fraction
    of its length.  
    Furthermore, in the case of two long escapes $\epsilon_0$ and
    $\epsilon_1$ they
    are one $x$--escape and one $y$--escape.
    
    If the
    endpoints of $\gamma$ differ by a power of $x$ or $y$, then there
    is only one long escape, and it is of the corresponding flavor,
    and if the endpoints of $\gamma$ differ by a power of $a$ then
    either there is one long $a$--escape or there is one long
    $x$--escape and one long $y$--escape. 
  \end{thm}

  \begin{proof}
    Since $\gamma$ is $K$--biLipschitz for $K<2$, $\gamma$ does not
    use any $x$ or $y$ edges, so toral subpaths are just powers of $a$.
    Thus, the $a$--escapes and/or maximal toral subpaths have
    endpoints that differ by $a^{\ell_1},a^{\ell_2},\dots$ and the
    $x$--escapes and $y$--escapes have endpoints that differ by
    $x^{m_1},x^{m_2},\dots$ and $y^{n_1},y^{n_2},\dots$, respectively.
    
    Given $R'$ we will determine in the course of the proof a bound
    such that the rest of the argument works provided $R$ is larger
    than that bound and $K$ and $J$ are chosen as in Lemma~\ref{biglemma} with
    respect to $R$.
    Since we are allowed to enlarge $R$ we may assume $R\geq 3L+4$ so
    that Proposition~\ref{longescapesinexponent} applies. 
    Let $T$ be the usual sum of absolute
    values.
    Let $J':=JK$.
    If $T<J$ then $|\gamma|\leq KT<KJ$, so it suffices to take
    $|\gamma|\geq J'$ to guarantee $T\geq J$ so that we may apply Lemma~\ref{biglemma}.

    By Lemma~\ref{biglemma} there are either one or two long escapes.
    Proposition~\ref{longescapesinexponent} bounds these in terms of the
    exponent of the difference between their endpoints.
    It remains only to recast these bounds in terms of path length.
    
    Suppose there are two long escapes $\epsilon_1$ and $\epsilon_2$,
    the strategy in the other cases being
    similar.
    By Proposition~\ref{longescapesinexponent},
    $\max_j|m_j|+\max_k|n_k|>\left(\frac{R-1}{R}\right)^2T$.
    The complement of these two escapes in $\gamma$ consists of at
    most three subpaths, $\delta_1$, $\delta_2$, and $\delta_3$.
    Let $I_{a,1}$ be the set of indices $i$ such that the escape of $\gamma$
    corresponding to the term $a^{\ell_i}$ belongs to $\delta_1$.
    Similarly, define  $I_{a,2}$, $I_{a,3}$, $I_{x,1}$, etc.
    \begin{align*}
      \frac{|\delta_1|+|\delta_2|+|\delta_3|}{|\gamma|}&\leq\frac{K\sum_{h=1}^3|a^{\sum_{i\in
                                                         I_{a,h}}\ell_i}x^{\sum_{j\in
                                                         I_{x,h}}m_j}y^{\sum_{k\in
                                                         I_{y,h}}n_k}|}{|\delta_1|+|\epsilon_1|+|\delta_2|+|\epsilon_2|+|\delta_3|}\\
      &\leq\frac{K\sum_{h=1}^3|a^{\sum_{i\in
                                                         I_{a,h}}\ell_i}|+|x^{\sum_{j\in
                                                         I_{x,h}}m_j}|+|y^{\sum_{k\in
        I_{y,h}}n_k}|}{|\epsilon_1|+|\epsilon_2|}\\
      &\leq\frac{KC\sum_{h=1}^3|{\sum_{i\in
                                                         I_{a,h}}\ell_i}|^{1/\alpha}+|{\sum_{j\in
                                                         I_{x,h}}m_j}|^{1/\alpha}+|{\sum_{k\in
        I_{y,h}}n_k}|^{1/\alpha}}{(\max_j|m_j|)^{1/\alpha}+(\max_k|n_k|)^{1/\alpha}}\\
      &\leq\frac{KC\cdot 9^{1-1/\alpha}(\sum_{h=1}^3|{\sum_{i\in
                                                         I_{a,h}}\ell_i}|+|{\sum_{j\in
                                                         I_{x,h}}m_j}|+|{\sum_{k\in
        I_{y,h}}n_k}|)^{1/\alpha}}{(\max_j|m_j|+\max_k|n_k|)^{1/\alpha}}\\
       &\leq\frac{KC\cdot 9^{1-1/\alpha}(\sum_{h=1}^3{\sum_{i\in
                                                         I_{a,h}}|\ell_i|}+{\sum_{j\in
                                                         I_{x,h}}|m_j|}+{\sum_{k\in
         I_{y,h}}|n_k|})^{1/\alpha}}{(\max_j|m_j|+\max_k|n_k|)^{1/\alpha}}\\
      &<\frac{KC\cdot
        9^{1-1/\alpha}\left(\left(1-\left(\frac{R-1}{R}\right)^2\right)T\right)^{1/\alpha}}{\left(\left(\frac{R-1}{R}\right)^2T\right)^{1/\alpha}}\\
      &=KC\cdot
        9^{1-1/\alpha}\left(\frac{2R-1}{(R-1)^2}\right)^{1/\alpha}\\
      &<2C\cdot
        9^{1-1/\alpha}\left(\frac{2R-1}{(R-1)^2}\right)^{1/\alpha}\stackrel{R\to\infty}{\longrightarrow}0
    \end{align*}
    Thus, given $R'>1$, for sufficiently large $R$ we have
    $\frac{|\delta_1|+|\delta_2|+|\delta_3|}{|\gamma|}<\frac{1}{R'}$,
    or, conversely, $\frac{|\epsilon_1|+|\epsilon_2|}{|\gamma|}>
    \frac{R'-1}{R'}$.
%
  \end{proof}

\subsection{The central `region' of a biLipschitz cycle}\label{sec:escapecentralregion}

The goal of this subsection is to show, as promised in Section~\ref{sec:idea:snowflake},
that the central region has at most four long sides.

\newcommand{\gvert}{\node[vertex]}
\newcommand{\dgedge}{%
  \draw[thick,%
  postaction={decorate},%
  decoration={markings,mark=at position 1/2 with {\arrow{>}}}]}
\newcommand{\dgedgedec}[1]{%
  \draw[thick,%
  postaction={decorate},%
  decoration={markings,mark=at position 1/2 with {\arrow{>},#1}}]}

\newcommand{\squig}{%
  \draw [decorate, decoration={snake, segment length=3pt, amplitude=1pt}]}

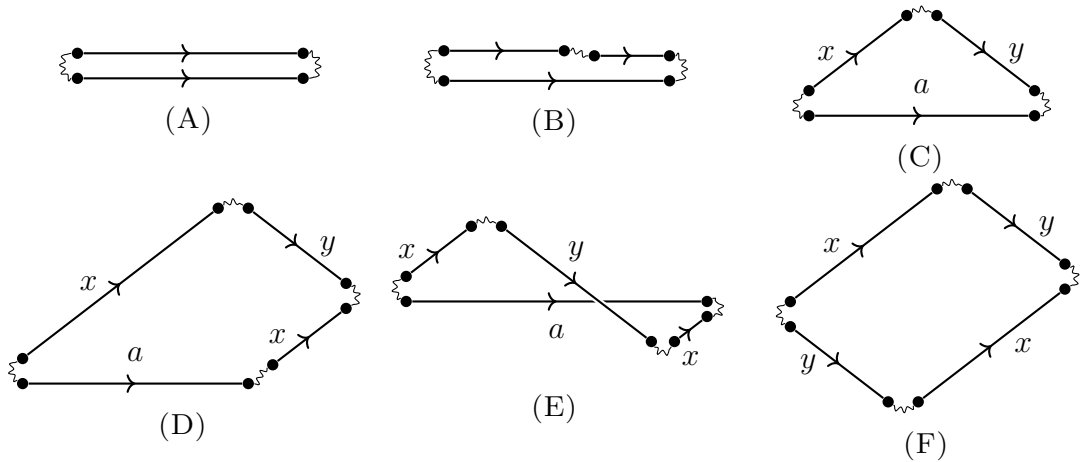
\begin{figure}
  \begin{subfigure}[c]{0.3\linewidth}
    \centering
    \begin{tikzpicture}[scale=1]
      \gvert (a1) at (0,0) {};
      \gvert (a2) at (3,0) {};
      \gvert (b1) at (0,1/3) {};
      \gvert (b2) at (3,1/3) {};
      \dgedge (a1) -- (a2);
      \dgedge (b1) -- (b2);
      
      \squig (a1) to[out=180,in=180] (b1);
      \squig (b2) to[out=0,in=0] (a2);
    \end{tikzpicture}
    \caption{}
    \label{fig:central_region:two_escapes}
  \end{subfigure}
  \begin{subfigure}[c]{0.3\linewidth}
    \centering
    \begin{tikzpicture}[scale=1]
      \gvert (a1) at (0,0) {};
      \gvert (a2) at (3,0) {};
      \gvert (b1) at (0,2/5) {};
      \gvert (b2) at (8/5,2/5) {};
      \gvert (b3) at (2,1/3) {};
      \gvert (b4) at (3,1/3) {};
      \dgedge (a1) -- (a2);
      \dgedge (b1) -- (b2);
      \dgedge (b3) -- (b4);
      
      \squig (a1) to[out=180,in=180] (b1);
      \squig (b4) to[out=0,in=0] (a2);
      \squig (b2) to[out=0,in=180] (b3);
    \end{tikzpicture}
    \caption{}
    \label{fig:central_region:three_escapes:slim}
  \end{subfigure}
  \begin{subfigure}[c]{0.3\linewidth}
    \centering
    \begin{tikzpicture}[scale=1]
      \gvert (a1) at (0,0) {};
      \gvert (a2) at (3,0) {};
      \gvert (b1) at (0,1/3) {};
      \gvert (b2) at (13/10,4/3) {};
      \gvert (b3) at (17/10,4/3) {};
      \gvert (b4) at (3,1/3) {};
      \dgedgedec{\node[label={above:$a$}] {};} (a1) -- (a2);
      \dgedgedec{\node[label={left:$x$}] {};} (b1) -- (b2);
      \dgedgedec{\node[label={right:$y$}] {};} (b3) -- (b4);
      
      \squig (a1) to[out=180,in=210] (b1);
      \squig (b4) to[out=-30,in=0] (a2);
      \squig (b2) to[out=30,in=150] (b3);
    \end{tikzpicture}
    \caption{}
    \label{fig:central_region:three_escapes:fat}
  \end{subfigure}
  \\
  \begin{subfigure}[c]{0.3\linewidth}
    \centering
    \begin{tikzpicture}[scale=1]
      \gvert (a1) at (0,0) {};
      \gvert (a2) at (3,0) {};
      \gvert (a3) at (133/40,1/4) {};
      \gvert (a4) at (43/10,1) {};
      \gvert (b1) at (0,1/3) {};
      \gvert (b2) at (13/5,7/3) {};
      \gvert (b3) at (3,7/3) {};
      \gvert (b4) at (43/10,4/3) {};
      \dgedgedec{\node[label={above:$a$}] {};} (a1) -- (a2);
      \dgedgedec{\node[label={left:$x$}] {};} (a3) -- (a4);
      \dgedgedec{\node[label={left:$x$}] {};} (b1) -- (b2);
      \dgedgedec{\node[label={right:$y$}] {};} (b3) -- (b4);
      
      \squig (a1) to[out=150,in=210] (b1);
      \squig (b4) to[out=-30,in=30] (a4);
      \squig (b2) to[out=30,in=150] (b3);
      \squig (a3) to[out=210,in=0] (a2);
    \end{tikzpicture}
    \caption{}
    \label{fig:central_region:four_escapes:a_embed}
  \end{subfigure}
  \begin{subfigure}[c]{0.31\linewidth}
    \centering
    \begin{tikzpicture}[scale=1]
      \gvert (a1) at (0,0) {};
      \gvert (a2) at (4,0) {};
      \gvert (a3) at (4,-1/5) {};
      \gvert (a4) at (4-13/30,-1/5-1/3) {};
      \gvert (b1) at (0,1/3) {};
      \gvert (b2) at (104/120,1) {};
      \gvert (b3) at (152/120,1) {};
      \gvert (b4) at (1956/600,-1/5-1/3) {};
      \dgedgedec{\node[label={below:$a$}] {};} (a1) -- (a2);
      \dgedgedec{\node[label={below:$x$}] {};} (a4) -- (a3);
      \dgedgedec{\node[label={left:$x$}] {};} (b1) -- (b2);
      \draw[thickeraser] (b3) -- (b4);
      \dgedgedec{\node[label={above:$y$}] {};} (b3) -- (b4);
      
      \squig (a1) to[out=150,in=210] (b1);
      \squig (b4) to[out=-60,in=239] (a4);
      \squig (b2) to[out=30,in=150] (b3);
      \squig (a3) to[out=0,in=30] (a2);
    \end{tikzpicture}
    \caption{}
    \label{fig:central_region:four_escapes:a_cross}
  \end{subfigure}
  \begin{subfigure}[c]{0.3\linewidth}
    \centering
    \begin{tikzpicture}[scale=1]
      \gvert (a1) at (0,0) {};
      \gvert (a2) at (13/10,-1) {};
      \gvert (a3) at (17/10,-1) {};
      \gvert (a4) at (73/20,1/2) {};
      \gvert (b1) at (0,1/3) {};
      \gvert (b2) at (39/20,11/6) {};
      \gvert (b3) at (47/20,11/6) {};
      \gvert (b4) at (73/20,5/6) {};
      \dgedgedec{\node[label={left:$y$}] {};} (a1) -- (a2);
      \dgedgedec{\node[label={right:$x$}] {};} (a3) -- (a4);
      \dgedgedec{\node[label={left:$x$}] {};} (b1) -- (b2);
      \dgedgedec{\node[label={right:$y$}] {};} (b3) -- (b4);
      
      \squig (a1) to[out=150,in=210] (b1);
      \squig (b4) to[out=-30,in=30] (a4);
      \squig (b2) to[out=30,in=150] (b3);
      \squig (a3) to[out=220,in=-30] (a2);
    \end{tikzpicture}
    \caption{}
    \label{fig:central_region:four_escapes:rectangle}
  \end{subfigure}
  \caption{The possible configurations, up to symmetry, of the long
    escapes of a central region of the minimal area disk diagram of a
    biLipschitz cycle.  The directed segments represent the oriented
    corridors of the long escapes with labels indicating the flavor of
    the escape.  (In the configurations without labels, all long
    escapes must have the same flavor but each of the three flavors is
    possible.)  The squiggly paths represent the complement of the
    long escapes in the cycle.}
  \label{fig:central_region_cases}
\end{figure}

\begin{defn}
  Let $\gamma \from S \to X$ be a cycle such that $\gamma\cap
  H\neq\emptyset$ and such that all of the escapes of $\gamma$ from
  $H$ have length at most $|\gamma|/2$.
  Let $\{\gamma|_{P_1},\,\gamma|_{P_2},\dots,\,\gamma|_{P_k}\}$ be a
  set of  distinct escapes from $H$.
The set is \emph{manageable} if one of the following conditions referencing
  Figure~\ref{fig:central_region_cases} holds.
  \begin{itemize}
  \item $k = 2$ and
    \begin{enumerate}
    \item[\eqref{fig:central_region:two_escapes}] $\gamma|_{P_1}$ and
      $\gamma|_{P_2}$ are of the same flavor with opposite sign.
    \end{enumerate}
   \item $k = 3$ and
    \begin{enumerate}
    \item[\eqref{fig:central_region:three_escapes:slim}] the
      $\gamma|_{P_i}$ are all of the same flavor but are not all of
      the same sign, or
    \item[\eqref{fig:central_region:three_escapes:fat}] there is
      exactly one escape of each flavor among the $\gamma|_{P_i}$ and
      the $a$-escape has sign opposite to that of the $x$-escape and
      the $y$-escape.
    \end{enumerate}
  \item $k = 4$ and three consecutive escapes $\gamma|_{P_{i_1}}$,
    $\gamma|_{P_{i_2}}$ and $\gamma|_{P_{i_3}}$, alternate between $x$
    and $y$ flavors and the fourth escape $\gamma|_{P_{i_4}}$ has sign
    opposite to $\gamma|_{P_{i_2}}$ and either
    \begin{enumerate}
    \item[\eqref{fig:central_region:four_escapes:a_embed}]
      $\gamma|_{P_{i_4}}$ is an $a$-escape and the escapes
      $\gamma|_{P_{i_1}}$ and $\gamma|_{P_{i_3}}$ have opposite sign,
      or
    \item[\eqref{fig:central_region:four_escapes:a_cross}]
      $\gamma|_{P_{i_4}}$ is an $a$-escape and the escapes
      $\gamma|_{P_{i_1}}$ and $\gamma|_{P_{i_3}}$ have the same sign,
      or
    \item[\eqref{fig:central_region:four_escapes:rectangle}]
      $\gamma|_{P_{i_4}}$ is of the same flavor as $\gamma|_{P_{i_2}}$
      and the escapes $\gamma|_{P_{i_1}}$ and $\gamma|_{P_{i_3}}$ have
      opposite sign.
    \end{enumerate}
  \end{itemize}
\end{defn}

\begin{prop}\label{prop:manage_or_simplify}
  For any  $M \geq 24C+1$ and $R > M-1$ there are
  $J > 0$ and a $K \in (1,2)$ such that the following holds.

  Let $\gamma \from S \to X$ be a $K$-biLipschitz cycle of length
  $|\gamma| \ge J$ such that $\gamma\cap H\neq \emptyset$ and every
  escape of $\gamma$ from $H$ has length at most $|\gamma|/2$.
  Let $\gamma|_{P_1}$,
  $\gamma|_{P_2}$, \ldots, $\gamma|_{P_k}$ be distinct escapes from $H$
  such that $k \le 5$ and
  $\sum_{i=1}^k\bigl|\gamma|_{P_i}\bigr| \ge \frac{R-1}{R} \cdot
  |\gamma|$.  Then:
  \begin{enumerate}
  \item[(I)] $k > 1$
  \item[(II)] If $k \in \{2,3,4,5\}$ then either the $\gamma|_{P_i}$
    are manageable or
    $\bigl|\gamma|_{P_i}\bigr| \le \frac{M-1}{R} \cdot |\gamma|$ for
    some $i$.
  \end{enumerate}
\end{prop}
\begin{proof}
  Given $M$ and $R$ as in the statement, let $J$ and $K$ be large
  enough with respect to $R$ so that Theorem~\ref{longescapesinlength}
  is satisfied, and so that $J\geq R$.
  
  We have
  $\sum_{i=1}^k\bigl|\gamma|_{P_i}\bigr| \ge \frac{R-1}{R} \cdot
  |\gamma| > \frac{1}{2} \cdot J > 0$ so $k > 0$.  If $k=1$ then by hypothesis
  $\frac{|\gamma|}{2} \ge \bigl|\gamma|_{P_1}\bigr| \ge \frac{R-1}{R}
  \cdot |\gamma| > \frac{1}{2} \cdot |\gamma|$, a contradiction.  This
  establishes (I).

  Let $\{Q_i\}_{i=1}^k$ be the set of closures of components of
  $S \setminus \bigcup_{i=1}^k P_i$ and let $\beta \from Q \to X$ be
  the path obtained by concatenating translates of the
  $\gamma|_{Q_i}$.
  Note that if the endpoints of the concatenation of translates of the $\gamma|_{P_i}$ differ by $g \in G_L$ then the
  endpoints of $\beta$ differ by $g^{-1}$.

  For the cases $k\in \{2,3,4\}$, it will suffice to prove that if the
  $\gamma|_{P_i}$ are not manageable then $|\beta| \ge \frac{1}{2}
  |r|^{1/\alpha} - 5$ where $g^r$ is the difference between the endpoints of
  some $\gamma|_{P_i}$ with $g \in \{a, x, y\}$.  Indeed, then we have
  \begin{align*}
    \frac{|\gamma|}{R}
    &\ge |\gamma| - \sum_{i=1}^k\bigl|\gamma|_{P_i}\bigr| \\
    &=|\beta| \\
    &\ge \frac{1}{2} |r|^{1/\alpha} - 5 \\
    &\ge \frac{1}{2C} \cdot |g^r| - 5 \\
    &\ge \frac{1}{2KC} \cdot \bigl|\gamma|_{P_i}\bigr| - 5
  \end{align*}
  so that, as long as $J \ge R$, we have
  \begin{align*}
    |\gamma|_{P_i}\bigr|
    &\le 2KC \cdot \biggl(\frac{|\gamma|}{R} + 5 \biggr) \\
    &< 4C \cdot \biggl(\frac{1}{R} + \frac{5}{J} \biggr)
      \cdot |\gamma| \\
    &\le \frac{24C}{R} \cdot |\gamma|
  \end{align*}
  and so we need only ensure that $M \ge 24C + 1$ to
  obtain $|\gamma|_{P_i}\bigr| \le \frac{M-1}{R} \cdot |\gamma|$.

  Consider the case $k=2$.  If the $\gamma|_{P_i}$ are unmanageable
  then either they differ in flavor or they have the same sign.
  Suppose that $\gamma|_{P_1}$ and $\gamma|_{P_2}$ have different
  flavors.  Without loss of generality, $\gamma|_{P_1}$ is an
  $x$-escape and $\gamma|_{P_2}$ is either a $y$-escape or an
  $a$-escape.  Let $x^m$ be the difference between the endpoints of
  $\gamma|_{P_1}$ and let $y^n$ and $a^{\ell}$ be the same for
  $\gamma|_{P_2}$ in each of the two cases.  In the first case,
  Corollary~\ref{cor:planebound} tells us that
  \[ |\beta| \ge |x^my^n| \ge |m|^{1/\alpha} - 5 \] so we are done.
  In the second case, choose $p$ and $q$ so that $\ell=pL+q$ with
  $0 \leq |q| < L$ and either $q=0$ or $q$ and $\ell$ have opposite
  sign.  Corollary~\ref{cor:planebound} tells us that
  \[|\beta| \ge |a^\ell x^m| \ge |p|^{1/\alpha} - 5 \] but our choice
  of $p$ and $q$ ensures that
  $|p| = \frac{|\ell-q|}{L} = \frac{|\ell|+|q|}{L} \ge
  \frac{|\ell|}{L}$ and so we have
  $|\beta| \ge \frac{|\ell|^{1/\alpha}}{L^{1/\alpha}} - 5=\frac{1}{2}|\ell|^{1/\alpha} - 5$.  It
  remains to consider the case where $\gamma|_{P_1}$ and
  $\gamma|_{P_2}$ have the same flavor $g \in \{x,y,a\}$ and the same
  sign.
  Suppose the endpoints of $\gamma|_{P_1}$ differ by $g^{m_1}$
  and those of $\gamma|_{P_2}$ differ by $g^{m_2}$.  Then
  $|\beta| \ge |g^{m_1+m_2}| \ge |m_1 + m_2|^{1/\alpha} \ge
  |m_1|^{1/\alpha}$.  This completes the proof for $k=2$.

  Consider the case $k=3$.  We split into subcases based on the number
  of flavors that appear among the the $\gamma|_{P_i}$: (i) exactly
  one flavor, (ii) exactly two flavors or (iii) exactly three flavors.
  If the $\gamma|_{P_i}$ are unmanageable then in case~(i) they all
  have the same sign, in case~(ii) up to symmetry the flavors that
  appear are $(x,a,a)$, $(x,x,a)$ or $(x,x,y)$ and in case~(iii) up to
  symmetry the $x$-escape and the $a$-escape have the same sign.
  Case~(i) is treated similarly to the case $k=2$ with all escapes
  having the same flavor and sign.  In case~(iii), let the endpoints
  of the escapes differ by $a^{\ell}$, $x^m$ and $y^n$.  If $p$ and
  $q$ so that $\ell=pL+q$ with $0 \leq |q| < L$ then $\ell$ and $p$
  have the same sign so that, by Corollary~\ref{cor:planebound}, we
  have
  \[|\beta| \ge |a^\ell x^my^n| \ge |m+p|^{1/\alpha} - 5 =
    \bigl(|m|+|p|\bigr)^{1/\alpha} - 5 \ge |m|^{1/\alpha} - 5 \] so we
  are done.  For case~(ii), the $(x,x,a)$ and $(x,x,y)$ subcases are
  treated by again applying Corollary~\ref{cor:planebound} to obtain
  \[|\beta| \ge |m_1 + m_2 + p|^{1/\alpha} + |p|^{1/\alpha} - 5 \ge
    |p|^{1/\alpha} - 5 \ge\frac{|\ell|^{1/\alpha}}{L^{1/\alpha}} -
    5=\frac{|\ell|^{1/\alpha}}{2} -
    5 \] (where the last inequality is obtained as in the $(x,a)$ case
  of $k=2$ above) and
  \[|\beta| \ge |m_1 + m_2|^{1/\alpha} + |n|^{1/\alpha} - 5 \ge
    |n|^{1/\alpha} - 5 \] respectively.  It remains to consider the
  $(x,a,a)$ subcase for which Corollary~\ref{cor:planebound} gives us
  the following inequality.
  \[|\beta| \ge |m + p_1 + p_2|^{1/\alpha} + |p_1 + p_2|^{1/\alpha} -
    5 \] If $|m| \le |p_1+p_2|$ we are done so assume that
  $|p_1+p_2| \le |m|$.  Then, by the triangle inquality we have
  $|m+p_1+p_2| \ge |m|-|p_1+p_2| \ge 0$ and so
  $|m+p_1+p_2|^{1/\alpha} \ge \bigl(|m|-|p_1+p_2|\bigr)^{1/\alpha} \ge
  |m|^{1/\alpha} - |p_1+p_2|^{1/\alpha}$ by
  Fact~\ref{fact:reverse_holder}.  Thus we have
  \[ |\beta| \ge |m|^{1/\alpha} - |p_1+p_2|^{1/\alpha} + |p_1 +
    p_2|^{1/\alpha} - 5 = |m|^{1/\alpha} - 5 \] as required.

  Consider the case $k=4$.  Assume that the $\gamma|_{P_i}$ are
  unmanageable.  Renumber the segments so that $P_1$, $P_2$, $P_3$ and
  $P_4$ appear consecutively along $S$.  Then either $P_1 \cup P_2$ is
  contained in a segment of length at most $\frac{|S|}{2}$ of $S$ or
  $P_3 \cup P_4$ is contained in such a segment and similarly for
  $P_2 \cup P_3$ and $P_1 \cup P_4$.  We again renumber the segments
  so that the first possibility holds in each case.
  Apply
  Theorem~\ref{longescapesinlength} to the minimal geodesic
  segments respectively containing $P_1 \cup P_2$ and $P_2 \cup P_3$
  to conclude that either
  $\bigl|\gamma|_{P_i}\bigr| \le \frac{M-1}{R} \cdot |\gamma|$ for
  some $i$ or that $\gamma|_{P_1}$, $\gamma|_{P_2}$ and
  $\gamma|_{P_3}$ alternate between $x$ and $y$ flavors.  Then, up to
  symmetry, we can assume that $\gamma|_{P_1}$ and $\gamma|_{P_3}$ are
  $x$-escapes and $\gamma|_{P_2}$ is a $y$-escape.  If $\gamma|_{P_4}$
  were an $x$-escape then, applying Corollary~\ref{cor:planebound} we
  would obtain:
  \[ |\beta| \ge |m_1+m_3+m_4|^{1/\alpha} + |n_2|^{1/\alpha} - 5 \ge
    |n_2|^{1/\alpha} - 5 \] as needed.  It remains to consider the
  cases where $\gamma|_{P_4}$ is a (i) $y$-escape or an (ii)
  $a$-escape.  In case~(i), since the $\gamma|_{P_i}$ are
  unmanageable, either the escapes $\gamma|_{P_1}$ and $\gamma|_{P_3}$
  have the same sign or the escapes $\gamma|_{P_2}$ and
  $\gamma|_{P_4}$ have the same sign.  Up to symmetry we can assume
  the former so that Corollary~\ref{cor:planebound} gives us
  \[|\beta| \ge |m_1+m_3|^{1/\alpha} + |n_2+n_4|^{1/\alpha} - 5 \ge
    |m_1|^{1/\alpha} - 5 \] as needed.  In case~(ii), since the
  $\gamma|_{P_i}$ are unmanageable, the escapes $\gamma|_{P_2}$ and
  $\gamma|_{P_4}$ must have the same sign but then
  Corollary~\ref{cor:planebound} gives us
  \[|\beta| \ge |m_1+m_3+p_4|^{1/\alpha} + |n_2+p_4|^{1/\alpha} - 5
    \ge |n_2|^{1/\alpha} - 5 \] as needed.

  Consider the case $k=5$.  The $\gamma|_{P_i}$ are unmanageable so we
  need to show that
  $\bigl|\gamma|_{P_i}\bigr| \le \frac{M-1}{R} \cdot |\gamma|$ for
  some $i$.  Suppose this does not hold. 
  Apply
  Theorem~\ref{longescapesinlength} to conclude that no three
  consecutive escapes among the $\gamma|_{P_i}$ are contained in a
  subsegment of $\gamma$ of length at most $\frac{|\gamma|}{2}$ of
  $\gamma$.  Then any two consecutive segments among the
  $\gamma|_{P_i}$ are contained in a subsegment of $\gamma$ of length
  at most $\frac{|\gamma|}{2}$.  Then
  Theorem~\ref{longescapesinlength} implies that, for any two
  consecutive escapes, one of them is an $x$-escape and the other one
  is a $y$-escape.  Since $k$ is odd this is a contradiction.
\end{proof}

\begin{corollary}\label{cor:manageable_subset}
  For all $M\geq 24C+1$ and $R > M-1$ there exist
  $J > 0$ and $1<K<2$ such that the following holds.

  Let $\gamma \from S \to X$ be a $K$-biLipschitz cycle of length
  $|\gamma| \ge J$ such that $\gamma\cap H\neq\emptyset$ and every
  escape of $\gamma$ from $H$ has length at most $|\gamma|/2$.
  Let $\gamma|_{P_1}$,
  $\gamma|_{P_2}$, \ldots, $\gamma|_{P_k}$ be distinct escapes from $H$
  such that $k \le 5$ and
  $\sum_{i=1}^k\bigl|\gamma|_{P_i}\bigr| \ge \frac{R-1}{R} \cdot
  |\gamma|$.  Then a subset
  $T \subseteq \bigl\{\gamma|_{P_i}\bigr\}_{i=1}^{k}$ of
  $|T| \in \{2,3,4\}$ longest escapes is manageable and satisfies
  $\sum_{\gamma|_P \in T}\bigl|\gamma|_P\bigr| \ge \frac{R-M^3}{R}
  \cdot |\gamma|$ and
  $\min_{\gamma|_P \in T}\bigl|\gamma|_P\bigr| \ge (M-1) \max
  \Bigl\{\frac{|\gamma|}{R}, |\gamma| - \sum_{\gamma|_P \in
    T}\bigl|\gamma|_P\bigr| \Bigr\}$.
\end{corollary}
\begin{proof}
  Let $R_0>M^4(M-1)$, and let $R_i:=R_0M^{-i}$ for $i\leq 4$, so that
  for all $i$, $M$ and $R_i$
  satisfy the hypotheses of Proposition~\ref{prop:manage_or_simplify}.
  Let $K \in (1,2)$ be close enough to $1$ and
  let $J > 0$ be large enough that 
  Proposition~\ref{prop:manage_or_simplify} holds with respect to 
  $M$ and $R_i$, for all $0\leq i\leq 4$.

  Let $\gamma \from S \to X$ be a $K$-biLipschitz cycle of length
  $|\gamma| \ge J$ such that $\gamma\cap H\neq\emptyset$ and every
  escape of $\gamma$ from $H$ has length at most $|\gamma|/2$.
  Let $\gamma|_{P_1}$,
  $\gamma|_{P_2}$, \ldots, $\gamma|_{P_k}$ be distinct escapes from $H$
  such that $k \le 5$ and
  $\sum_{i=1}^k\bigl|\gamma|_{P_i}\bigr| \ge \frac{R_0-1}{R_0} \cdot
  |\gamma|$.

  If every escape of $T_0 = \bigl\{\gamma|_{P_i}\bigr\}_{i=1}^{k}$ has
  length greater than $\frac{M-1}{R_0}\cdot|\gamma|$ then applying
  Proposition~\ref{prop:manage_or_simplify} with $R = R_0$ the escapes
  of $T_0$ must be manageable.  Moreover, in this case we would have
  \[ \min_{\gamma|_P \in T_0}\bigl|\gamma|_P\bigr| > (M-1) \cdot
    \biggl(|\gamma| - \frac{R_0-1}{R_0} \cdot |\gamma|\biggr) \ge
    (M-1) \cdot \biggl(|\gamma| - \sum_{\gamma|_P \in
      T_0}\bigl|\gamma|_P\bigr|\biggr) \] so we would have the
  required conclusion with $R = R_0$ and $T = T_0$.  Hence we may
  assume that some escape of $T_0$ has length at most
  $\frac{M-1}{R_0}\cdot|\gamma|$.

  Let $T_1$ be obtained from $T_0$ by removing a shortest escape.  We
  have
  \[ \sum_{\gamma|_P \in T_1}\bigl|\gamma|_P\bigr| \ge
    \biggl(\frac{R_0 - 1}{R_0} - \frac{M - 1}{R_0}\biggr) \cdot
    |\gamma| = \frac{R_0 - M}{R_0} \cdot |\gamma| = \frac{R_1 -
      1}{R_1} \cdot |\gamma| \] so, applying
  Proposition~\ref{prop:manage_or_simplify} with $R = R_1$, we can
  conclude that $2 \le |T_1| = |T_0| - 1 \le 4$.  Moreover, if every
  escape of $T_1$ had length greater than
  $\frac{M-1}{R_1}\cdot|\gamma|$ then $T_1$ would be manageable and,
  as argued above, we would have
  $\min_{\gamma|_P \in T_1}\bigl|\gamma|_P\bigr| \ge (M-1) \cdot
  \Bigl(|\gamma| - \sum_{\gamma|_P \in
    T_1}\bigl|\gamma|_P\bigr|\Bigr)$.  Thus we would have the required
  conclusion with $R = R_0$ and $T = T_1$.  Hence we may assume that
  some escape of $T_1$ has length at most
  $\frac{M-1}{R_1}\cdot|\gamma|$.

  Continuing in this way, we obtain $T_2$ from $T_1$ by removing a
  shortest escape and if $T = T_2$ is not as needed then we again
  remove a shortest escape to obtain $T_3$.  By the same arguments, we
  have
  $\sum_{\gamma|_P \in T_3}\bigl|\gamma|_P\bigr| \ge \frac{R_3 -
    1}{R_3} \cdot |\gamma| = \frac{R_0 - M^3}{R_0} \cdot |\gamma|$ and
  $|T_3| = 2$.  If the shorter escape of $T_3$ had length at most
  $\frac{M-1}{R_3}\cdot|\gamma|$ then the remaining escape would have
  length at least $\frac{R_4 - 1}{R_4} \cdot |\gamma|$ which, by
  Proposition~\ref{prop:manage_or_simplify} with $R=R_4$, is a
  contradiction.  Thus the escapes of $T_3$ must all have length
  greater than
  $\frac{M-1}{R_3}\cdot|\gamma| \ge (M-1) \cdot \Bigl(|\gamma| -
  \sum_{\gamma|_P \in T_3}\bigl|\gamma|_P\bigr|\Bigr)$.  Applying
  Proposition~\ref{prop:manage_or_simplify} with $R=R_3$ the escapes
  of $T_3$ must also be manageable, so we are done with $R = R_0$ and
  $T = T_3$.
\end{proof}

\begin{corollary}\label{cor:very_manageable}
  For all $R>0$ there exist $J>0$ and $1<K<2$ such that if $\gamma
  \from S \to X$  is a 
  $K$-biLipschitz cycle of length
  $|\gamma| \ge J$ with $\gamma\cap H\neq\emptyset$ and all escapes of
  $\gamma$ from $H$ have length at most $|\gamma|/2$
then there is a manageable set $T$ of longest
  escapes of $\gamma$ from $H$ that satisfies
  $\sum_{\gamma|_P \in T}\bigl|\gamma|_P\bigr| \ge \frac{R-1}{R} \cdot
  |\gamma|$ and
  $\min_{\gamma|_P \in T}\bigl|\gamma|_P\bigr| \ge R \cdot \max
  \Bigl\{\frac{|\gamma|}{J}, |\gamma| - \sum_{\gamma|_P \in
    T}\bigl|\gamma|_P\bigr| \Bigr\}$.
\end{corollary}
\begin{proof}
  Let $M \ge R+1$ be large enough that we can apply
  Corollary~\ref{cor:manageable_subset} and let $R'$ be large enough
  that $\frac{R'-M^3}{R'} \ge \frac{R - 1}{R}$ and large enough that
  we can apply Corollary~\ref{cor:manageable_subset} with $M$ and $R'$.
  Assume that $J \ge R'$ is large enough and $K > 1$ is close enough to
  $1$ that the second paragraph of Corollary~\ref{cor:manageable_subset} holds.

  Let $\gamma \from S \to X$ be a $K$-biLipschitz cycle of length $|\gamma| \ge J$.
  Let $v$ be a vertex of $S$ that maps to $H$.
  Let $\bar v$ be the antipode of $v$ in $S$.
  Since the relators of $G_L$ all have even length so then does $S$ and so $\bar v$ is a vertex of $S$.
  Let $Q'_1$ and $Q'_2$ be the  two geodesic segments of $S$ between $v$ and $\bar v$.
  If $\bar v \in P$ for some escape $\gamma|_P$ of $A$ then, for each
  $i$, let $Q_i = Q'_i \setminus P^{\circ}$ where $P^{\circ}$ is the
  interior of $P$; otherwise, for each $i$, let $Q_i = Q'_i$.
  If $J$ is large enough and $K$ is close enough to $1$, Theorem~\ref{longescapesinlength} implies that at most four
  escapes occupy an $\frac{R'-1}{R'}$ proportion of $\gamma|_{Q_1}$
  and $\gamma|_{Q_2}$.
  Thus a set $T'$ of five longest escapes of $A$
  have length at least $\frac{R'-1}{R'} \cdot |\gamma|$.
  Applying Corollary~\ref{cor:manageable_subset} to $T'$ we obtain a
  manageable set $T$ of longest escapes of $A$ such that:
  \[ \sum_{\gamma|_P \in T}\bigl|\gamma|_P\bigr| \ge \frac{R'-M^3}{R'}
    \cdot |\gamma| \ge \frac{R-1}{R} \cdot |\gamma| \] and
  \[ \min_{\gamma|_P \in T}\bigl|\gamma|_P\bigr| \ge (M-1) \cdot
    \Bigl(|\gamma| - \sum_{\gamma|_P \in T}\bigl|\gamma|_P\bigr|
    \Bigr) \ge R \cdot \Bigl(|\gamma| - \sum_{\gamma|_P \in
      T}\bigl|\gamma|_P\bigr| \Bigr) \] and
  \[ \min_{\gamma|_P \in T}\bigl|\gamma|_P\bigr| \ge (M-1) \cdot
    \frac{|\gamma|}{R'} \ge R \cdot \frac{|\gamma|}{J} \qedhere\]
\end{proof}

 \subsection{Non-central `regions'}\label{sec:noncentralregions}
 Recall compound escapes from Definition~\ref{def:compoundescape}.
Fix $R$ and let $J$ and $K$ be the corresponding constants from
Theorem~\ref{longescapesinlength}, and suppose that $\gamma\from
S\to X$ is a $K$--biLipschitz cycle.
Suppose that $\gamma\cap gH\neq\emptyset$ and suppose that $\gamma$ has
a designated compound escape $\gamma|_P$ from $gH$ that has length greater than
$|\gamma|/2$.
Call $\gamma|_P$ the ``central escape''.
We wish to describe the configuration of non-central escapes of
$\gamma$ from $gH$.
This is easier than in the previous subsection because $\bigl|\gamma\bigr| -
\bigl|\gamma|_{P}\bigr|<|\gamma|/2$, so all of the non-central escapes
are contained in a common $K$--biLipschitz path.
This fact eliminates most of the cases considered
in  Figure~\ref{fig:central_region_cases}; only the
analogues of cases \eqref{fig:central_region:two_escapes} and
\eqref{fig:central_region:three_escapes:fat} can occur. 
Specifically, according to Theorem~\ref{longescapesinlength}, there are
only three possibilities:
\begin{enumerate}
\item $\bigl|\gamma\bigr| -
\bigl|\gamma|_{P}\bigr|<J$. In this case say \emph{the non-central
  part of $\gamma$ at $gH$
  is very small}.
  \item The non-central part is not very small, the central escape is
    a compound $a$--escape,
    and there are a non-central $x$--escape and non-central $y$--escape whose
    lengths sum to at least $\frac{R-1}{R}\left(\bigl|\gamma\bigr| -
\bigl|\gamma|_{P}\bigr|\right)$. In this case say that the non-central
part of $\gamma$ is
\emph{branching at $gH$}.

\item  The non-central part is not very small and there is exactly one escape of length at least $\frac{R-1}{R}\left(\bigl|\gamma\bigr| -
\bigl|\gamma|_{P}\bigr|\right)$, and it is of the same flavor as the
central escape. 
\end{enumerate}

Instances of the third case can be chained together to form
super-regions that we now describe:

\begin{definition}\label{def:renfilade}
  Let $R>2$.
  Let $\gamma$ be a $g$-escape for some $g\in\{a,x,y\}$.
  The  \emph{$R$--enfilade decomposition} of $\gamma$ is the
  factoring of $\gamma$ as:
  \[ \gamma_0:=\gamma = \epsilon_0+\alpha_1+\epsilon_1 +\cdots+\alpha_n+ \epsilon_n+\gamma'_n+\epsilon_n^{-1}+
    \beta_n +\cdots+\epsilon_1^{-1}+ \beta_1 +\epsilon_0^{-1} \] such
  that the following conditions hold (see Figure~\ref{fig:enfilade}).
  \begin{itemize}
  \item Each $\alpha_i$ or $\beta_i$ is a path whose endpoints differ
    by an element of $H$.
  \item For each $i$, we have
    $\epsilon_i \in \{s^{\pm 1}, t^{\pm 1}\}$ and $\gamma_i'$ is
    defined to be the path such that $\gamma_i=\epsilon_i+\gamma_i'+\epsilon_i^{-1}$.
    \item For each $i \in \{0, 1, \ldots, n-1\}$, we have
    $|\gamma_{i+1}| \ge \frac{R-1}{R} \cdot |\gamma'_i|$.
  \item For each $i \in \{0,1, 2, \ldots, n\}$
    \[ \gamma_i = \epsilon_i+\alpha_{i+1} +\cdots+ \epsilon_n+\gamma'_n+\epsilon_n^{-1}+ \cdots +\beta_{i+1}+
      \epsilon_i^{-1} \] is a $g_i$-escape, for some
    $g_i\in\{a,x,y\}$, whose endpoints differ by 
    $g_i^{m_i}$ for some $m_i \in \Z$.
    \item $n$ is maximal such that $\gamma$ admits a decomposition
      satisfying these conditions. 
    \end{itemize}
    Call the subsegment $\gamma_n'$ the \emph{end} of the enfilade.
\end{definition}
Note that since $R>2$ there is at most one candidate for
$\gamma_{i+1}$ for each $i$, so the decomposition is unique. 
\begin{figure}
  \centering
  \labellist
  \tiny
  \pinlabel $g_0^{m_0}$ [r] at 15 310
  \pinlabel $\epsilon_0$ [u] at 62 560
  \pinlabel $g_1^{m_0}$ [l] at 110 310
  \pinlabel $\alpha_1$ [u] at 225 650
  \pinlabel $\beta_1$ [d] at 225 -20
  \pinlabel $g_1^{m_1}$ [r] at 340 310
  \pinlabel $\epsilon_1$ [u] at 385 560
  \pinlabel $g_2^{m_1}$ [l] at 430 310
  \pinlabel $\alpha_2$ [u] at 545 650
  \pinlabel $\beta_2$ [d] at 545 -20
  \pinlabel $g_2^{m_2}$ [r] at 660 310
  \pinlabel $\epsilon_2$ [u] at 710 560
  \pinlabel $g_3^{m_2}$ [l] at 760 310
  \pinlabel $g_n^{m_n}$ [r] at 1110 310
  \pinlabel $g_{n+1}^{m_n}$ [l] at 1200 310
  \pinlabel $\gamma'_n$ [l] at 1480 310
  \pinlabel $\epsilon_n$ [u] at 1150 550
  \pinlabel $\alpha_n$ [u] at 1050 575
  \pinlabel $\beta_n$ [d] at 1050 40
  \pinlabel $\epsilon_n$ [d] at 1150 70
  \pinlabel $\epsilon_0$ [d] at 60 70
  \pinlabel $\epsilon_1$ [d] at 385 70
    \pinlabel $\epsilon_2$ [d] at 710 70
  \endlabellist
  \includegraphics[width=.7\textwidth]{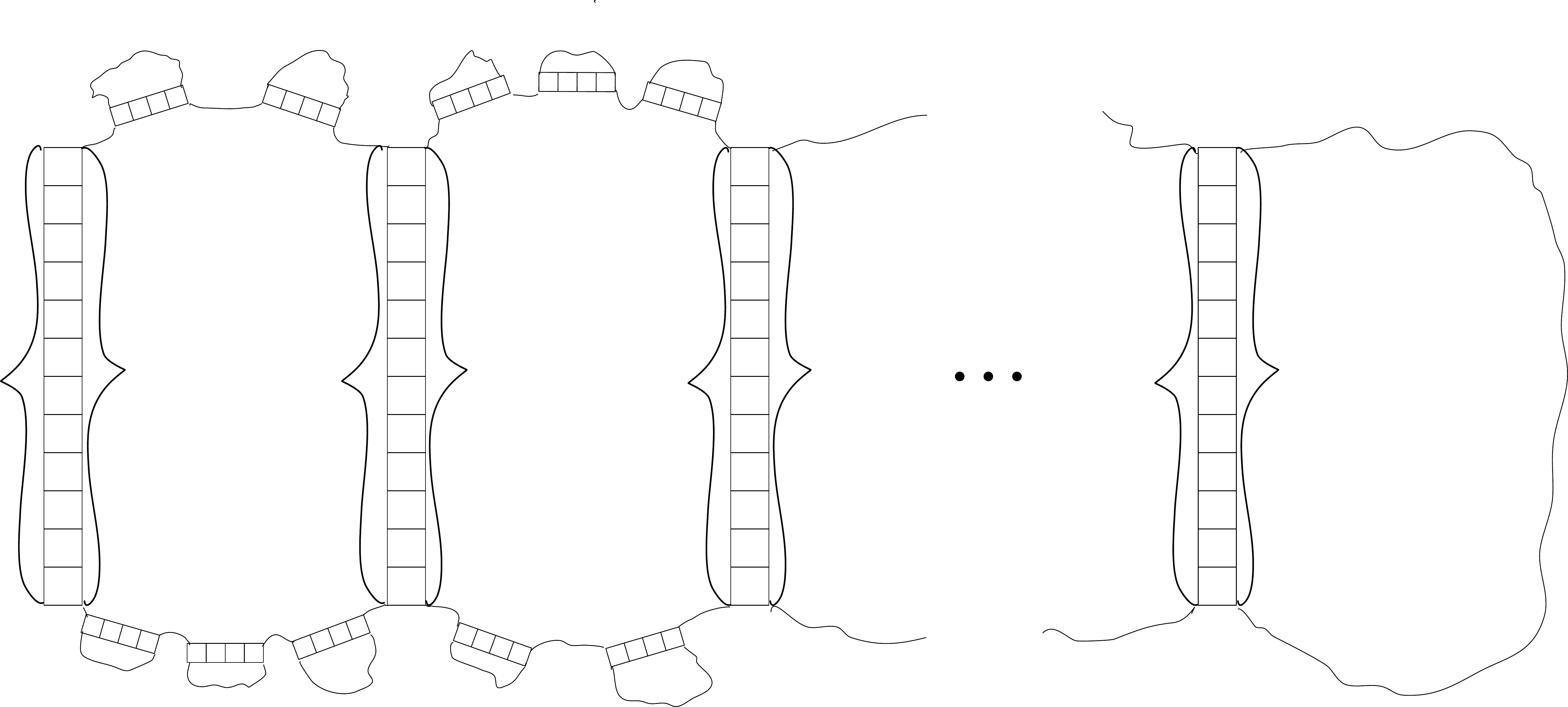}
  \caption{Enfilade of length $n$.}\label{fig:enfilade}
\end{figure}
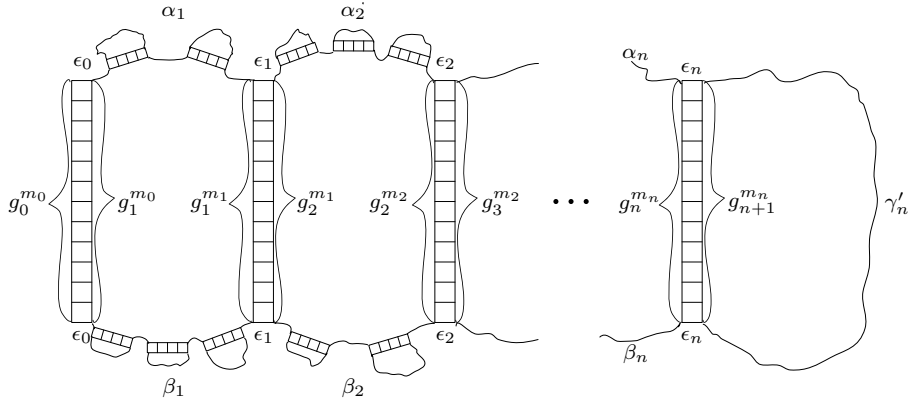

\begin{lemma}\label{enfiladesamesign}
  If $R>2$ and $1<K<2$ and $\gamma$ is a $K$--biLipschitz  $g$--escape for some
  $g\in\{a,x,y\}$ then in the $R$--enfilade decomposition of $\gamma$
  all of the exponents $m_i$ have the same sign. 
\end{lemma}
\begin{proof}
  If $m_{i-1}$ and $m_{i}$ have different signs for some $i$ then
  $|\beta_i|\geq |g_i^{m_{i-1}}|-|\alpha_i|+|g_i^{m_i}|$, but
  $\gamma'_{i-1}=\alpha_i+\gamma_{i}+\beta_i$, so:
  \[|\gamma_{i-1}'|-|\gamma_i|=|\alpha_i|+|\beta_i|\geq
    |g_i^{m_{i-1}}|+|g_i^{m_i}|\geq
    \frac{1}{K}(|\gamma_{i-1}'|+|\gamma_i|)\implies
    |\gamma_i|\leq\frac{K-1}{K+1}|\gamma_{i-1}'|\]
But   $|\gamma_i|>\frac{R-1}{R}|\gamma_{i-1}'|$, so $\frac{1}{2}<\frac{R-1}{R}<\frac{K-1}{K+1}<\frac{1}{3}$,
 and we have a contradiction.
\end{proof}
\begin{lemma}\label{enfilade_group_action}
  Assume that $\gamma$ begins at 1 and that we have an enfilade
  decomposition of $\gamma$ with notation as above.
  Let $h_i\in H$ be the difference between the endpoints of $\alpha_i$.
 For any $k\in\mathbb{Z}$, we have that $g_0^k.(\epsilon_0+\sum_{i=1}^n\alpha_i+\epsilon_i)$ is a path from
 $g_0^k$ to $(\epsilon_0\prod_{i=1}^nh_i\epsilon_i)g_{n+1}^k$.
\end{lemma}
\begin{proof}
  This follows by induction on $n$ using that facts that 
$g_i =
  \epsilon_{i-1}^{-1}g_{i-1}\epsilon_{i-1} $ and $g_i$ and
  $h_i$ commute.
\end{proof}

\begin{remark}
  Let $R > 2$, let $2>K > 1$, let $\gamma$ be a $K$-biLipschitz  $g$--escape for some $g \in \{a,x,y\}$ and let
  \[ \gamma = \epsilon_0+\alpha_1+\cdots+\epsilon_n+\gamma'_n+\epsilon_n^{-1}
   +\cdots+\beta_1 +\epsilon_0^{-1} \] be the 
 $R$-enfilade decomposition of $\gamma$.
 By
  Theorem~\ref{longescapesinlength}, there exists a $J > 1$ such
  that if $K$ is close enough to $1$ then either $|\gamma'_n| \le J$ or
  the endpoints of $\gamma'_n$ differ by a power of $a$ and
  $\gamma'_n$ has an $x$--escape and a $y$--escape that collectively
  account for an at least $\frac{R-1}{R}$ fraction of its length.
\end{remark}

\begin{prop}[Enfilades are short]\label{enfilade_width}
  Let $R > 3$.  There exists a $2>K > 1$ and $J > 2$ such that if $\gamma$ is a $K$-biLipschitz
  $g$-escape for some $g \in \{a,x,y\}$ and $|\gamma| \ge J$ and
    \[ \gamma = \epsilon_0+\alpha_1+\cdots+\epsilon_n+\gamma'_n+\epsilon_n^{-1}
      +\cdots+\beta_1 +\epsilon_0^{-1} \] is the
  $R$-enfilade decomposition of $\gamma$ then $|\gamma'_n| \ge \frac{R-3}{R}\cdot |\gamma|$.
\end{prop}
\begin{proof}
  Since $R>3$, $\frac{R-1}{R}\cdot\frac{R-2}{R}>\frac{R-3}{R}$.
  Let $R_0 > 1$ be large enough to satisfy
  $\frac{R_0}{R_0 + 1} > \frac{2}{R}$ and
  $\frac{R-1}{R} \cdot \frac{R_0}{R_0 + 1} > \frac{R-2}{R}$.  Let
  $J>2$ , $2>K > 1$ be such that the conclusion of
  Corollary~\ref{cor:one_long_path} holds with $(R,J,K) = (R_0,J,K)$
  and so that all of the following hold:
  \begin{itemize}
   \item$\frac{2}{R} + \frac{4}{J} < \frac{R_0}{R_0 + 1}$ 
  \item$\frac{R-1}{R} \cdot \bigl(\frac{R_0}{R_0 + 1} - \frac{2}{J}\bigr) >
  \frac{R-2}{R}$.
  \item$\frac{R-1}{R}\cdot\frac{R_0}{R_0 +
       1}-\frac{4}{J}
   \geq \frac{R-2}{R}$
 \item $\frac{R-1}{R}\cdot\frac{J-2}{J}\geq \frac{R-2}{R}$
   \item $\frac{R-1}{R}\cdot\frac{R-2}{R}\cdot\frac{J-2}{J}>\frac{R-3}{R}$
  \end{itemize}
 
  Let $\gamma$ be a $K$-biLipschitz  $g$-escape for some $g
  \in \{a,x,y\}$, let $|\gamma| \ge J$ and consider the $R$-enfilade
  decomposition of $\gamma$ with the same notation as in
  Definition~\ref{def:renfilade}.

  First assume that $g_0=a$, and notice that then $g_i = a$ whenever $i$ is even.
  We have $\alpha_i+g_i^{m_i}+\beta_i = g_i^{m_{i-1}}$, for each $i \in
  \{1, 2, \ldots n\}$.
  By commutativity of $H$, the endpoints of
  $\alpha_i+\beta_i$ differ by $g_i^{m_{i-1} - m_i}$.
Define $\delta_{2i}:=\epsilon_{2i-2}+\alpha_{2i-1}+\epsilon_{2i-1}+\alpha_{2i}+
\beta_{2i}+\epsilon_{2i-1}^{-1}+\beta_{2i-1}+\epsilon_{2i-2}^{-1}$.
By the previous observation and an argument similar to
Lemma~\ref{enfilade_group_action}, the endpoints of $\delta_{2i}$
differ by $a^{m_{2i-2} - m_{2i}}$, so the endpoints of  
  $\delta_{2i}+\gamma_{2i}$ differ by $a^{m_{2i-2} -
    m_{2i}}=a^{m_{2i-2}}$, which is the same as the difference between
  the endpoints of $\gamma_{2i-2}$.
  Moreover, we have:
  \begin{align*}
    |\delta_{2i}+\gamma_{2i}|
    &= |\epsilon_{2i-2}+\alpha_{2i-1}+\epsilon_{2i-1}+\alpha_{2i}+
      \beta_{2i}+\epsilon_{2i-1}^{-1}+\beta_{2i-1}+\epsilon_{2i-2}^{-1}| +
      |\gamma_{2i}| \\
    &= |\epsilon_{2i-2}+\alpha_{2i-1}+\epsilon_{2i-1}+\alpha_{2i}+\gamma_{2i}
      +\beta_{2i}+\epsilon_{2i-1}^{-1}+\beta_{2i-1}+\epsilon_{2i-2}^{-1}| \\
    &= |\gamma_{2i-2}|
  \end{align*}
  By induction, the path
  $\delta_2+\delta_4+\cdots+\delta_{2i}+\gamma_{2i}$ has the same
  length and the same endpoints as $\gamma$.
  Let $k \in \Z$ be defined by $n-1 \le 2k \le n$ and
  consider the path $\delta_2+\delta_4+\cdots+\delta_{2k}+\gamma_{2k}$.
  Since $\gamma$ is $K$-biLipschitz, we have
  $J \le |\gamma| = |\delta_2+\delta_4+\cdots+\delta_{2k}+\gamma_{2k}| =
  |\gamma| \le K|a^{m_0}|$.
  Also the endpoints of $\delta_{2i}$ differ by $a^{m_{2i-2} - m_{2i}}$, for each $i$, and
  the endpoints of $\gamma_{2k}$ differ by $a^{m_{2k}}$, so, by
  Corollary~\ref{cor:one_long_path}, either
  $|\delta_{2i}| > \frac{R_0}{R_0 + 1} \cdot |\gamma|$, for some
  $i \in \{1, 2, \ldots, k\}$, or
  $|\gamma_{2k}| > \frac{R_0}{R_0 + 1} \cdot |\gamma|$. But
  \begin{align*}
    |\delta_{2i}|
    &= |\epsilon_{2i-2}+\alpha_{2i-1}+\epsilon_{2i-1}+\alpha_{2i}+
      \beta_{2i}+\epsilon_{2i-1}^{-1}+\beta_{2i-1}+\epsilon_{2i-2}^{-1}| \\
    &= |\epsilon_{2i-2}+\alpha_{2i-1}+\beta_{2i-1}+\epsilon_{2i-2}^{-1}|
      + |\epsilon_{2i-1}+\alpha_{2i}+\beta_{2i}+\epsilon_{2i-1}^{-1}| \\
    &= \bigl(2 + |\gamma'_{2i-2}| - |\gamma_{2i-1}|\bigr) +
      \bigl(2 + |\gamma'_{2i-1}| - |\gamma_{2i}|\bigr) \\
    &\le |\gamma'_{2i-2}| - \frac{R-1}{R}\cdot|\gamma'_{2i-2}| +
      |\gamma'_{2i-1}| - \frac{R-1}{R}\cdot|\gamma'_{2i-1}| + 4 \\
    &= \frac{1}{R}\cdot|\gamma'_{2i-2}|
      + \frac{1}{R}\cdot|\gamma'_{2i-1}| + 4 \\
    &\le \frac{2}{R}\cdot|\gamma| + 4 \\
    &= \biggl(\frac{2}{R} + \frac{4}{|\gamma|}\biggr) \cdot|\gamma| \\
    &\le \biggl(\frac{2}{R} + \frac{4}{J}\biggr) \cdot|\gamma| \\
    &< \frac{R_0}{R_0 + 1} \cdot|\gamma|
  \end{align*}
  where the first inequality follows from the definition of
  $R$-enfilade decomposition.
  Thus:
  \[ \frac{R_0}{R_0 + 1}\cdot|\gamma| < |\gamma_{2k}| = |\gamma'_{2k}|
    + 2 \le \frac{R}{R-1}\cdot|\gamma_n| + 2 \]
  We conclude:
   \begin{align*}
    |\gamma'_n|&=|\gamma_n|-2\\
    &> \frac{R-1}{R}\cdot\biggl(\frac{R_0}{R_0 + 1}\cdot|\gamma| -
      2\biggr)-2 \\
     &= \frac{R-1}{R}\cdot\frac{R_0}{R_0 + 1}\cdot|\gamma| -
       2\cdot\frac{2R-1}{R} \\
     &> \frac{R-1}{R}\cdot\frac{R_0}{R_0 + 1}\cdot|\gamma| -
       4\\
    &=\left( \frac{R-1}{R}\cdot\frac{R_0}{R_0 +
      1}-\frac{4}{|\gamma|}\right)|\gamma| \\
     &\geq\left( \frac{R-1}{R}\cdot\frac{R_0}{R_0 +
       1}-\frac{4}{J}\right)|\gamma| \\
     &\geq \frac{R-2}{R}\cdot|\gamma|
   \end{align*}

   If $g_0\neq a$ then $g_0\in\{x,y\}$.
   If $n=0$ we are done, so suppose not.
   Since $g_0\neq a$, $g_1=a$ and $g_i=a$ for all odd $i$. 
   Apply the previous argument to $\gamma_1$ to get:
  
   \begin{align*}
     |\gamma_n'|&\geq  \frac{R-2}{R}\cdot |\gamma_1|\\
                &\geq  \frac{R-2}{R}\cdot\frac{R-1}{R}\cdot|\gamma_0'|\\
                &= \frac{R-2}{R}\cdot\frac{R-1}{R}\cdot(|\gamma|-2)\\
                &= \frac{R-2}{R}\cdot\frac{R-1}{R}\cdot\frac{J-2}{J}\cdot|\gamma|\\
     &\geq \frac{R-3}{R}\cdot|\gamma|\qedhere 
   \end{align*}
\end{proof}

\section{Constructing fillings}\label{sec:constructfillings}
We will show that for all $\lambda\in (0,1)$ we have that for all sufficiently
large $R$ and $\Lambda$ there exist $1<K<2$ and $A,J,M>0$ such that every $K$--biLipschitz loop
$\gamma$ can be filled by a diagram of mesh at most
$\lambda|\gamma|+M$ of area at most $A$.
We do so inductively.
An overview of the inductive argument is given in Section~\ref{sec:overview}.
In Section~\ref{approximatepolygons} we establish basic filling
results for the approximate polygons that will occur as vertex
regions. 
In Section~\ref{sec:base_case} we handle the base case: identifying
and filling a central region.
In Section~\ref{sec:inductive_steps} we cover three possible types of
induction steps: small enough parts, enfilades, and branching
regions.
In many of these steps we will need to make assumptions on the relationships
between the various constants.
These are boxed for emphasis.
In Theorem~\ref{maintheorem} we put all of the pieces of the proof
together, including observing that all of the boxed relations can be
satisfied simultaneously, so that we can actually, quantitatively,
perform the induction described in Section~\ref{sec:overview}.

\subsection{Overview: Constructing a filling inductively}\label{sec:overview}
In Lemma~\ref{lem:subdividecentral} we identify a coset $gH$ from
which $\gamma$ has a manageable collection of at most 4 long escapes. 
Call these \emph{the $0$-th level escapes}.
We resect those escapes and replace each with its trace in $gH$.
We then subdivide each of the traces into segments of bounded exponent
and replace each segment by a geodesic with the same endpoints.
We find a filling diagram for the resulting loop $\gamma'$ of
controlled mesh and area.
Call the filling diagram for $\gamma'$ the \emph{$0$-th shell}.

It remains to take each of the $0$-th level escapes, complete it to a
loop using the same concatenation of geodesic subsegments as used to
build $\gamma'$, and construct a filling diagram of that loop.

For the induction, 
suppose we have an $n$-th level escape $\gamma|_P$ of $\gamma$ and an $n$-th
shell that includes a filling diagram with a boundary segment that is a subdivision of the trace of $\gamma|_P$. 
If $\gamma|_P$ is ``small enough'' we can fill it with a single
2--cell.
This is quantified explicitly in \eqref{small_enough} below.
We will see in Section~\ref{sec:inductive_steps} that very small
parts are small enough.
If $\gamma|_P$ is not small enough then we look at its enfilade
decomposition.
Let $\gamma|_{P'}$ be the end of the enfilade, and let $\delta$ be the
toral segment between the endpoints of $\gamma|_{P'}$. 
In Section~\ref{sec:inductive_steps} we describe how to fill the
enfilade up to $\delta$, which also yields a subdivision of $\delta$.
Then according to Section~\ref{sec:noncentralregions} there are two possibilities: either $\gamma|_{P'}$ is short
enough to cap off with a single $2$--cell, or it is branching.
If it is branching then $\gamma|_{P'}$ has two long escapes and their
traces make an approximate triangle with $\delta$.
We fill that triangle. The fillings of the enfilade part and the
short/branching part together make up a component of the $(n+1)$-st
shell.
In the case of branching, the two long escapes are added to the list
of $(n+1)$-st level escapes. 

We finish the argument by showing that after finitely many steps, independent of $\gamma$, every remaining
branch is small enough that it can be capped off with a single
2--cell, so that no escapes get added to a deeper level of the induction.
The diagram filling $\gamma$ is the union of the shells, identified
along common segments. 
The induction hypothesis is:
\begin{equation}
  \label{induction_hypothesis}
  \parbox{.85\textwidth}{For an $n$--th level escape that is not
    small enough to cap off, the $n$--th shell provides a subdivision
    of its trace into at most
    $2(\Lambda+1)$--many subsegments of exponent at most
    $E|\gamma|^\alpha\cdot (3/L)^{n}$, where $E$ is the constant of
    Lemma~\ref{lem:subdividecentral}, which is independent of $\gamma$.}
\end{equation}

Lemma~\ref{lem:subdividecentral} establishes
\eqref{induction_hypothesis} for $n=0$.
The inductive step is proved in Section~\ref{sec:inductive_steps}.
Let us suppose that
\eqref{induction_hypothesis} is true for all $n$ and see how to finish
the argument.

\begin{equation}
  \label{eq:11}
  \fbox{\text{Assume $R>12C(\Lambda+1)/\lambda$.}}
\end{equation}
Suppose $\gamma$ is a $K$--biLipschitz loop and $gH$ is a coset such that:
\begin{equation}
  \label{small_enough}
  \parbox{.85\textwidth}{$\gamma$ has a central escape, a compound escape from $gH$ of length at least
$|\gamma|/2$, whose trace is subdivided into at most
$2(\Lambda+1)$--many subsegments of exponent bounded by
$L\left(\frac{|\gamma|}{R}\right)^\alpha$.}
\end{equation}

Choose geodesic segments $\delta_1,\dots,\delta_m$, $m\leq 2(\Lambda+1)$,  with the same endpoints as the subdivision
segments. 
Let $\gamma|_P$ be the complement of the central escape.
It is $K$--biLipschitz, with the same endpoints as
$\delta_1+\cdots+\delta_m$, so $|\gamma|_P|\leq K\sum_{i=1}^m|\delta_i|$.
Since the exponents of the $\delta_i$ are bounded, we apply
Theorem~\ref{powerdistortion} to get
$|\delta_i|<C\left(L\left(\frac{|\gamma|}{R}\right)^\alpha\right)^{1/\alpha}=2C|\gamma|/R$.
Thus:
\begin{align*}
  |\gamma|_P|+\sum_{i=1}^m|\delta_i|&\leq
                                      (K+1)\sum_{i=1}^m|\delta_i|\\
                                    &<3\cdot 2(\Lambda+1)\cdot
                                      2C|\gamma|/R\\
  &=12C(\Lambda+1)\cdot|\gamma|\cdot\frac{1}{R}\\
  &<12C(\Lambda+1)\cdot |\gamma|\cdot
    \frac{\lambda}{12C(\Lambda+1)}\\
  &=\lambda|\gamma|
\end{align*}

Thus, if \eqref{small_enough} is satisfied then $\gamma|_P+\bar\delta_m+\cdots+\bar\delta_1$ is a loop of length
less than $\lambda|\gamma|$, so it is small enough to cap off 
with a single $2$--cell satisfying the desired mesh bound.

The induction hypothesis \eqref{induction_hypothesis} says that an
$n$--th level escape has trace subdivided into segments of exponent at
most $E|\gamma|^\alpha\cdot (3/L)^{n}$, so it is 
guaranteed to be small
enough to cap off with a single 2--cell once we have:
\begin{equation}
  \label{eq:9}
  E|\gamma|^\alpha\cdot (3/L)^{n}\leq
  L\left(\frac{|\gamma|}{R}\right)^\alpha\iff n\geq \log_{L/3}(ER^\alpha/L)
\end{equation}
Since $E$ and $R$ are independent of $\gamma$, \eqref{eq:9} gives a
uniform upper
bound on the number of shells we need to fill before we can cap off
all remaining branches with a single 2--cell apiece.

Suppose that we have area bounds $\mathcal{C}$, $\mathcal{E}$, and $\mathcal{B}$ 
for our fillings of the central region, enfilades, and branching
regions, respectively, by diagrams of controlled mesh.
The worst-case scenario is that the central region has four long
escapes, each belonging to an enfilade terminating in a
branching region, and for each of these branching regions the two long
escapes belong to enfilades terminating in a branching
region, etc, carried out to the $n$--th shell, where $n$ is the least
integer satisfying \eqref{eq:9}, at which point all remaining regions
are small enough and can be capped off with one additional 2--cell.
The area of the resulting diagram is bounded above by:
\begin{equation}
  \label{eq:13}
  \mathcal{C}+4(\mathcal{E}+\mathcal{B})(2^n-1)+2^{n+2}
\end{equation}
Thus, if we control the mesh and area of the filling diagrams used in
the induction steps then we can assemble a uniformly bounded number of
them to get a filling diagram for $\gamma$ of controlled mesh and
uniformly bounded area.

\subsection{Approximate polygons and fillings}\label{approximatepolygons}
The construction of fillings for the different pieces of the argument can
be unified in the framework of filling approximate polygons.
An approximate polygon, for us, will mean a cycle, in $H$, of segments of
$a$--lines, $x$--lines, and $y$--lines such that the terminal endpoint
of one is close to the initial endpoint of the next.
We construct fillings of certain approximate polygons.
These will be the building blocks for fillings of biLipschitz loops
constructed in the next two sections. 
  
  \subsubsection*{Approximate polygons are close to true polygons}
We will quantify approximate polygons as \emph{$D$--approximate} to
mean that the distance between consecutive segments is as most $D$.
A \emph{true} polygon is a 0--approximate polygon.

\begin{defn} \label{defn:approximatetriangle} Let a
  \emph{$D$--approximate triangle} mean a collection of elements
  $g_0,g_1,g_2$ of $\langle a,x,y\rangle$ and numbers
  $m_0,m_1,m_2\in\mathbb{Z}$ such that the distances
  $d(g_0x^{m_0},g_{1})$, $d(g_1y^{m_1},g_{2})$, and
  $d(g_2a^{m_2},g_{0})$ are at most $D$.
\end{defn}

 \begin{lemma}[Approximate triangles]\label{approximatetriangle}
    For every such $D$--approximate triangle there exists a
    true triangle described by $g_0',g_1',g_2'\in\langle
    a,x,y\rangle$ and $m_0',m_1',m_2'\in\mathbb{Z}$ such that the
    distance between every corresponding pair of points is at most
    $2D+L$. 
  \end{lemma}
  \begin{proof}
    By Lemma~\ref{notallxylinesintersect}, there exists
    $p\in\mathbb{Z}$ with $|p|\leq L/2$ such that $g_0\langle
    x\rangle\cap g_1a^p\langle y\rangle\neq\emptyset$.
    Choose such a $p$, and let $g_1':=g_0\langle
    x\rangle\cap g_1a^p\langle y\rangle$.
    By Lemma~\ref{axintersection}, we can set $g_0':=g_0\langle
    x\rangle\cap g_2\langle a\rangle$ and $g_2':=g_2\langle
    a\rangle\cap g_1a^p\langle y\rangle=g_2\langle
    a\rangle\cap g'_1a^p\langle y\rangle$.
    Define $m_i'$ so that $g_0'x^{m_0'}=g_1'$, $g_1'y^{m_1'}=g_2'$,
    and $g_2'a^{m_2'}=g_0'$.
    Note that $m_0'=m_1'$ and $m_2'=-Lm_0'$.
    See Figure~\ref{fig:approxtriangle}.

    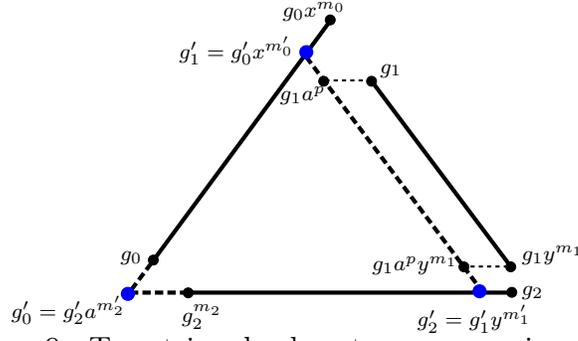
\begin{figure}[h]
      \centering
      \labellist
      \tiny
      \pinlabel $g_0$ [br] at 13 15
      \pinlabel $g_0x^{m_0}$ [b] at 93 131
      \pinlabel $g_1$ [bl] at 122 105
      \pinlabel $g_1y^{m_1}$ [bl] at 190 15
      \pinlabel $g_2$ [tl] at 190 10
      \pinlabel $g_2^{m_2}$ [t] at 39 2
      \pinlabel $g'_0=g'_2a^{m'_2}$ [tr] at 5 5
      \pinlabel $g_1'=g'_0x^{m'_0}$ [br] at 86 111
      \pinlabel $g_1a^p$ [tr] at 101 106
      \pinlabel $g_2'=g_1'y^{m_1'}$ [t] at 170 2
      \pinlabel $g_1a^py^{m_1}$ [r] at 166 17
      \endlabellist
      \includegraphics[height=1.5in]{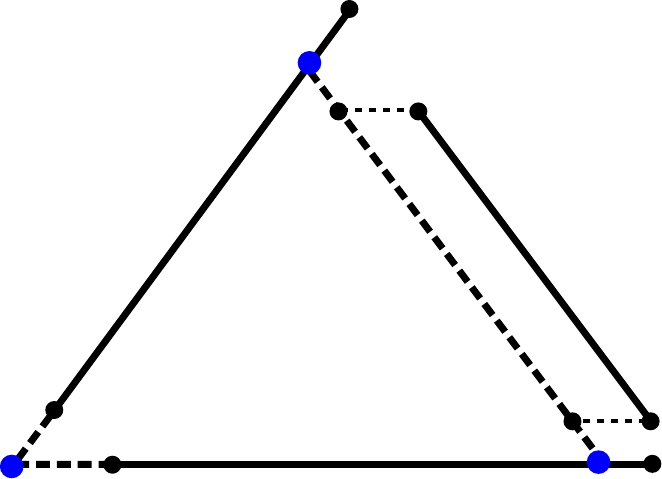}
      \caption{True triangle close to an approximate triangle}
      \label{fig:approxtriangle}
    \end{figure}

    It remains to bound the distance between corresponding points.
    We verify what we claim to be the worst case, and leave the
    remaining verifications to the reader.
    \begin{align*}
      d(g_2,g_2')&\leq d(g_2,g_1y^{m_1})+d(g_1y^{m_1},g_2')\\
                 &\leq D+d(g_1y^{m_1},g_2')\\
                 &= D+d(g_1y^{m_1},g_1'y^{m_1'})\\
      &\leq D+d(g_1y^{m_1},g_1a^py^{m_1})+d(g_1a^py^{m_1}, g_1'y^{m_1'})\\
      &=D+d(g_1y^{m_1},g_1a^py^{m_1})+d(g_1a^py^{m_1},g_2\langle
        a\rangle)\\
      &\leq D+2d(g_1y^{m_1},g_1a^py^{m_1})+d(g_1y^{m_1},g_2\langle
        a\rangle)\\
      &\leq D+L+d(g_1y^{m_1},g_2\langle
        a\rangle)\\
      &\leq D+L+d(g_1y^{m_1},g_2)\leq 2D+L\qedhere
    \end{align*}
  \end{proof}

\begin{lemma}[Approximate diamonds]\label{approximatediamond}
    Let a \emph{$D$--approximate diamond} mean a collection of elements
    $g_1$, $g_2$, $h_1$, $h_2\in\langle a,x,y\rangle$ and numbers
    $m_1$, $m_2$, $n_1$, $n_2\in\mathbb{Z}$ such that $d(g_1x^{m_1},h_1)$,
    $d(h_1y^{n_1},g_2)$, $d(g_2x^{m_2},h_2)$, and $d(h_2y^{n_2},g_1)$
    are all at most $D$.
    
    For every such $D$--approximate diamond there exits a
    true diamond described by $g'_1$, $g'_2$, $h'_1$, $h'_2\in\langle a,x,y\rangle$ and 
    $m'_1$, $m'_2$, $n'_1$, $n'_2\in\mathbb{Z}$ such that the distance between
    every corresponding pair of points is at most $D+3L/2$.
  \end{lemma}
  \begin{proof}
    By Lemma~\ref{notallxylinesintersect}, there exist numbers $p_1$,
    $p_2$, and $p_3$ with all $|p_i|\leq L/2$ such that $g_1\langle
    x\rangle\cap h_1a^{p_1}\langle y\rangle\neq \emptyset$,
    $g_1\langle x\rangle\cap h_2a^{p_2}\langle y\rangle\neq\emptyset$,
    and $g_2a^{p_3}\langle x\rangle\cap h_1a^{p_2}\langle
    y\rangle\neq\emptyset$.
    Choose such $p_i$.
    Define $h_1':=g_1\langle
    x\rangle\cap h_1a^{p_1}\langle y\rangle$.
    Define $g_2':=g_2a^{p_3}\langle x\rangle\cap h_1'\langle
    y\rangle$.
    Define $n_1'$ so that $h_1'y^{n_1'}=g_2'$.
    Define $h_2':=g_2'\langle x\rangle\cap h_2a^{p_2}\langle
    y\rangle$.
    Define $m_2'$ so that $h_2'=g_2'x^{m_2'}$.
    Define $g_1':=h_2'\langle y\rangle\cap g_1\langle
    x\rangle==h_2a^{p_2}\langle y\rangle\cap g_1\langle x\rangle$.
    Define $n_2'$ so that $g_1'=h_2'y^{n_2'}$.
    Define $m_1'$ so that $g_1'x^{m_1'}=h_2'$.
    Finally, note that $m_2'=-m_1'$ and $n_2'=-n_1'$.
    See Figure~\ref{fig:approxdiamond}.

    \begin{figure}[h]
      \centering
      \labellist
      \tiny
      \pinlabel $g_1$ [r] at 32 164
      \pinlabel $g_1x^{m_1}$ [r] at 130 285
      \pinlabel $h_1$ [l] at 137 252
      \pinlabel $h_1y^{n_1}$ [l] at 205 161
      \pinlabel $g_2$ [r] at 175 133
      \pinlabel $g_2x^{m_2}$ [r] at 80 5
      \pinlabel $h_2$ [r] at 70 45
      \pinlabel $h_2y^{n_2}$ [r] at 5 135
      \pinlabel $g_1'=h_2'y^{n_2'}$ [r] at 20 150
      \pinlabel $h_2a^{p_2}y^{n_2}$ [l] at 28 135
      \pinlabel $h_2a^{p_2}$ [bl] at 85 50
      \pinlabel $g_2a^{p_3}x^{m_2}$ [l] at 102 1
      \pinlabel $h_2'=g_2'x^{m_2'}$ [l] at 115 20
      \pinlabel $g_2'=h_1'y^{n_1'}$ [l] at 200 140
      \pinlabel $g_2a^{p_3}$ [tl] at 190 130
      \pinlabel $h_1a^{p_1}y^{n_1}$ [r] at 180 158
      \pinlabel $h_1a^{p_1}$ [tr] at 124 245
      \pinlabel $g_1'x^{m_1'}=h_1'$ [r] at 104 265
      \endlabellist
      \includegraphics[width=2in]{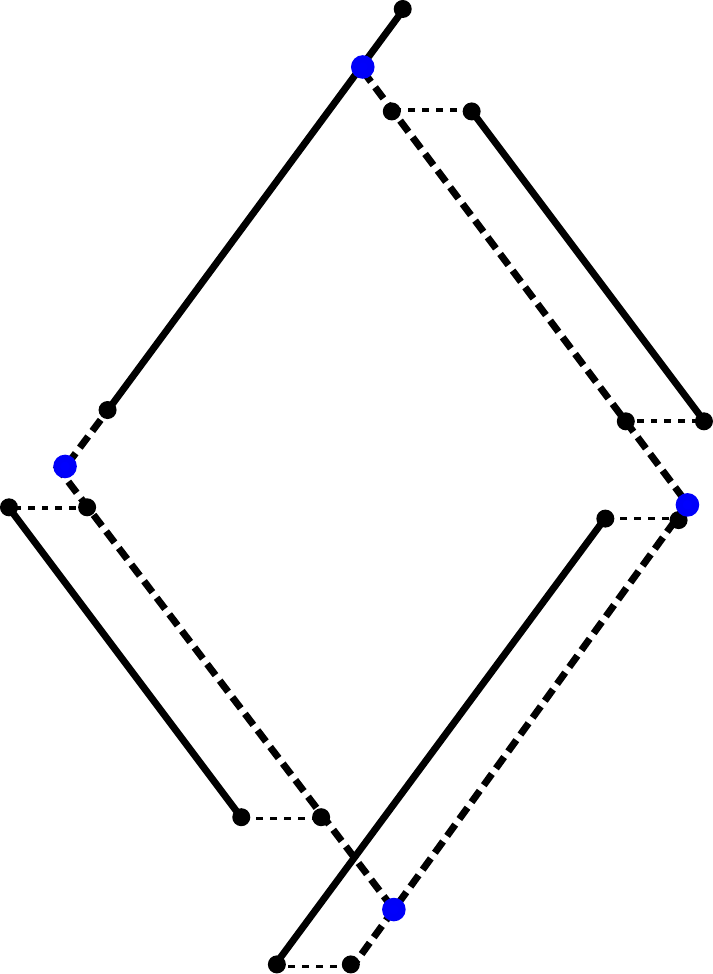}
      \caption{True diamond close to an approximate diamond}
      \label{fig:approxdiamond}
    \end{figure}
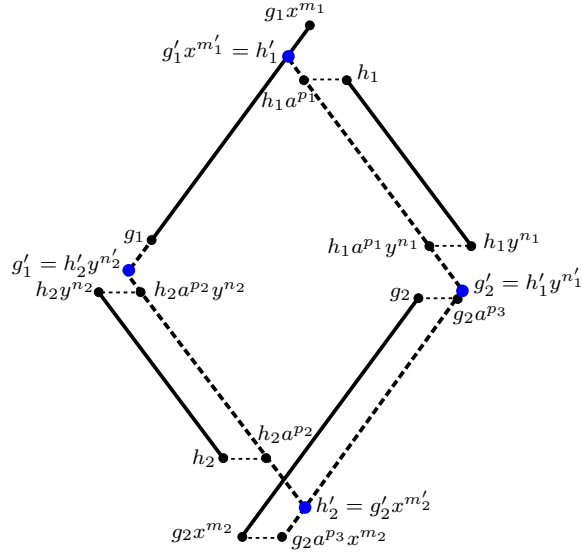

    It remains to bound the distance between corresponding points.
    We will check what we claim is the worst case, and leave the other,
    similar, verifications to the reader.
    \begin{align*}
      d(h_2,h_2')&\leq d(h_2,h_2a^{p_2})+d(h_2a^{p_2},h_2')\\
                 &= d(h_2,h_2a^{p_2})+d(h_2a^{p_2},g_2'x^{m_2'})\\
                 &\leq  d(h_2,h_2a^{p_2})+d(h_2a^{p_2},g_2'\langle x\rangle)\\
      &\leq  d(h_2,h_2a^{p_2})+d(h_2a^{p_2},g_2\langle
        x\rangle)+d_{\text{Haus}}(g_2\langle x\rangle, g'_2\langle
        x\rangle)\\
      &\leq L+d(h_2a^{p_2},g_2\langle
        x\rangle)\\
&\leq L +d(h_2a^{p_2},g_2x^{m_2})\\
&\leq L+d(h_2a^{p_2},h_2)+d(h_2,g_2x^{m_2})\leq L+L/2+D\qedhere
    \end{align*}
  \end{proof}

  \subsubsection*{Filling approxiate polygons}
  Let a \emph{$D$--approximate bigon} mean a collection of elements
  $g_0,g_1\in\langle a,x,y\rangle$, $g\in\{a,x,y\}$, and numbers
  $m_0,m_1\in\mathbb{Z}$ such that $d(g_ig^{m_i},g_{i+1})\leq D$ for
  all $0\leq i\leq 1$ with indices mod 2.

  For an approximate polygon, a collection of ``corner paths'' refers to a collection of
  paths connecting successive endpoints of sides.
 
  \begin{lemma}[Filling bigons]\label{fillbigon}
    There exist constants $M_0$, $M_1$, and $M_2$ such that for all
    positive $\Lambda$, $E$, and $D$, given an approximate bigon, a
    collection of corner paths of length at most $D$, and a
    subdivision of one side of the bigon into $\Lambda$--many
    subsegments, each of exponent at most $E$, the following hold.
    There exists a subdivision of the opposite side of the bigon into
    $\Lambda$--many subsegments of exponent bounded by $E+LD^\alpha$.
    Furthermore, if we construct a loop $\gamma$ out of the bigon with
    corner paths by replacing each of the subsegments of the sides by
    some geodesic, then there is a diagram for $\gamma$ with mesh at
    most $M_0+M_1D+M_2E^{1/\alpha}$ and area at most $\Lambda$.
    In fact, it suffices to take $M_0:=0$, $M_1:=2(2C+1)$, and
    $M_2:=2C$.
\end{lemma}
\begin{proof}
  Assume that we have been given a subdivision of the $g_0$--side of
  the bigon.  We refer to the $g_0$--side of the bigon as the
  ``bottom'' and the $g_1$--side as the ``top''.  Up to isometry we
  may assume $g_0=1$ and $m_0\geq 0$. Set $h:=g_1$. 
  
  Let $s_0:=0$ and let $s_i\geq 0$ be the sum of the first $i$
  exponents of the bottom, so that the bottom consists of
  $g$--segments $[g^{s_{i-1}},g^{s_i}]$ for $0<i\leq \Lambda$.  Note
  that $s_\Lambda=m_0$.  Let $\delta_0$ be the given corner path from
  $1$ to $hg^{m_1}$, and let $\delta_\Lambda$ be the corner path from
  $g^{s_\Lambda}$ to $h$.  Define $p$ to be the lesser of $\Lambda-1$
  and the greatest index such that
    $s_p$ is contained in the closed interval defined by $0$ and $m_1$
    (i.e. $[m_1,0]$ or $[0,m_1]$ depending on the sign of $m_1$).

  There are two cases according to whether or not $m_0>-m_1$.
    See Figure~\ref{fig:fillbigon}.
  \begin{figure}[h]
    \begin{subfigure}[c]{0.45\linewidth}
        \labellist
        \tiny
      \pinlabel $1$ [t] at 10 5
      \pinlabel $g^{s_p}$ [t] at 425 10
      \pinlabel $g^{s_\Lambda}$ [t] at 540 10
      \pinlabel $h$ [b] at 660 140
      \pinlabel $hg^{m_1}$ [b] at 40  140
      \pinlabel $hg^{m_1+s_p}$ [b] at 500 140
      \pinlabel $g^{s_1}$ [t] at 200 10
      \pinlabel $hg^{m_1+s_1}$ [b] at 250 140
        \endlabellist
         \centering
   \includegraphics[height=.5in]{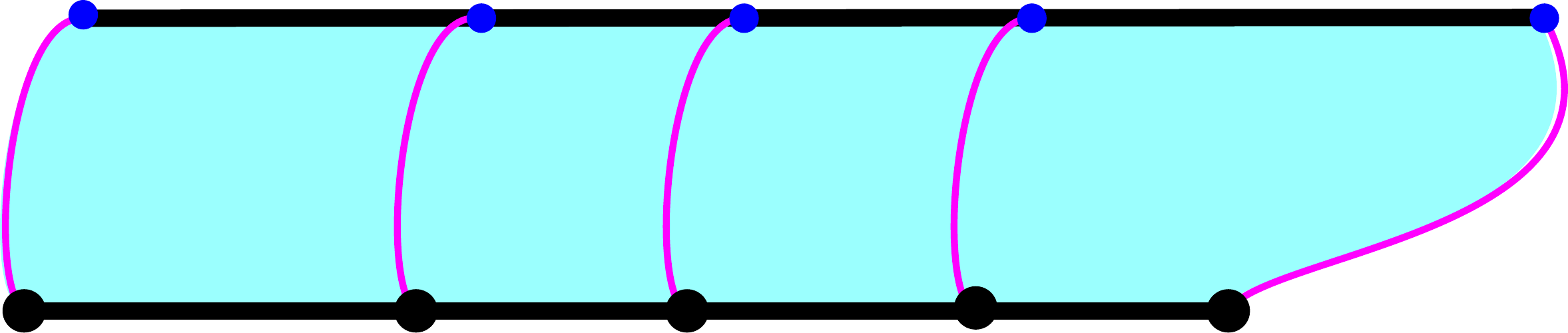}
   \caption{}
   \label{fillbigonA}
      \end{subfigure}
      \hfill
 \begin{subfigure}[c]{0.45\linewidth}
         \labellist
         \tiny
      \pinlabel $1$ [t] at 10 5
      \pinlabel $g^{s_p}$ [t] at 355 10
      \pinlabel $g^{s_\Lambda}$ [t] at 550 10
      \pinlabel $h$ [b] at 440 130
      \pinlabel $hg^{m_1}$ [b] at 40  140
      \pinlabel $hg^{m_1+s_p}$ [b] at 410 150
         \endlabellist
        \centering
    \includegraphics[height=.5in]{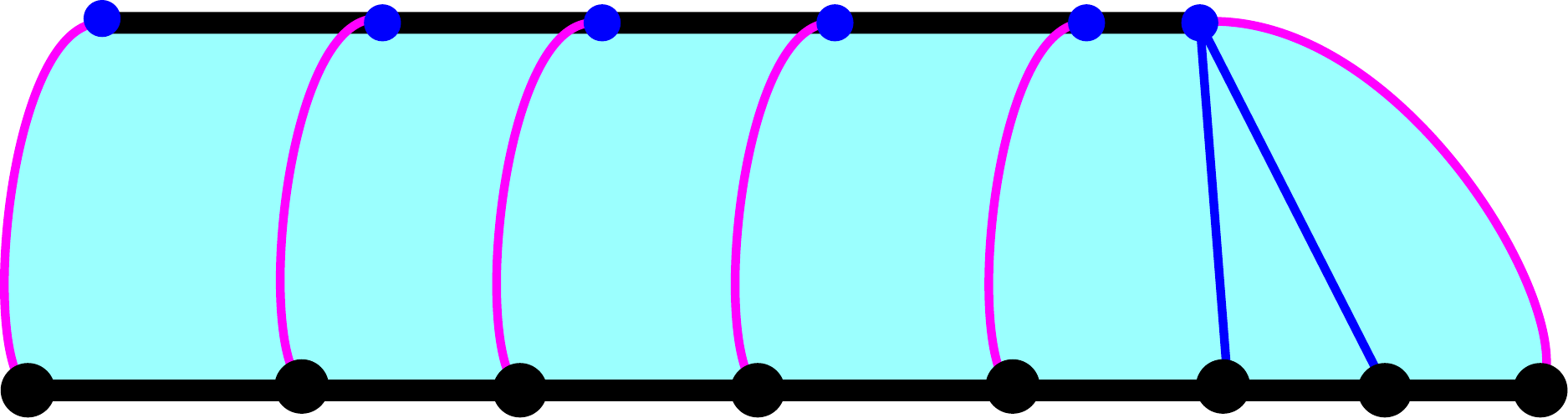}
    \caption{}
       \label{fillbigonB}
        \end{subfigure}
    \caption{Two cases for filling a bigon}
    \label{fig:fillbigon}
  \end{figure}
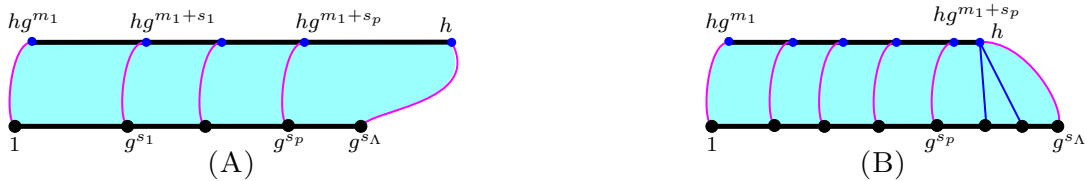

  First suppose, as in Figure~\eqref{fillbigonA}, that $m_0\leq -m_1$, so that $m_1 \leq 0$ and $p=\Lambda-1$.
    For each $0<i\leq p$, define $\delta_i=g^{s_i}.\delta_0$, so that $\delta_i$ is a path of
  length at most $D$ between $g^{s_i}$ and
  $g^{s_i}hg^{m_1}=hg^{m_1+s_i}$.
  
  For each $i<p$ define a 2--cell with boundary consisting of a
  geodesic from $g^{s_{i-1}}$ to $g^{s_{i}}$, $\delta_i$, a geodesic
  from $hg^{m_1+s_i}$ to $hg^{m+2+s_{i-1}}$, and $\bar\delta_{i-1}$.
  The $\delta$ sides have length at most $D$, while the other two
  sides have length bounded above by $CE^{1/\alpha}$, since
  $s_i-s_{i-1}\leq E$, so these 2--cells have perimeter of length at
  most $2D+2CE^{1/\alpha}<M_1D+M_2E^{1/\alpha}$.
  
  Now define a 2--cell whose boundary consists of a geodesic from
  $g^{s_{p}}$ to $g^{s_\Lambda}$, $\delta_\Lambda$, a geodesic from
  $h$ to $hg^{m_1+s_{p}}$, and $\bar\delta_{p}$.  The bottom has
  length less than $CE^{1/\alpha}$, since $p=\Lambda-1$, and the
  vertical sides each have length at most $D$ by hypothesis.  We
  have:
  \[|m_1+s_\Lambda|^{1/\alpha}<|g^{m_1+s_\Lambda}|=|(hg^{m_1+s_\Lambda})^{-1}h|\leq
    d(hg^{m_1+s_\Lambda},g^{s_\Lambda})+d(g^{s_\Lambda},h)\leq
    |\delta_0|+|\delta_\Lambda|\leq 2D\] So
  $|m_1+s_\Lambda|<(2D)^\alpha$, which implies
  $|m_1+s_{p}|<E+(2D)^\alpha=E+LD^\alpha$ and the length of a geodesic
  between $hg^{m_1+s_{p}}$ and $h$ is less than $CE^{1/\alpha}+2CD$,
  so the perimeter has length less than
  $2(C+1)D+2CE^{1/\alpha}<M_1D+M_2E^{1/\alpha}$.

  These 2--cells glue together along common $\delta_i$ to give the
  desired diagram.

  Now suppose $m_0>-m_1$, as in Figure~\eqref{fillbigonB}.
  In this case it is possible that $m_1 \ge 0$ with $p = 0$.
  For $i\leq p$ define $\delta_i$ as in the previous case.
  For each $i<p$, as in the previous case, define a 2--cell with boundary consisting of a geodesic from $g^{s_{i-1}}$ to $g^{s_{i}}$, $\delta_i$, a geodesic from $hg^{m_1+s_i}$ to $hg^{m+2+s_{i-1}}$, and $\bar\delta_{i-1}$.

  For $p<i<\Lambda$ choose a geodesic $\delta_i$ from $g^{s_i}$ to
  $h$.
  As in the previous case, $0<s_\Lambda+m_1\leq (2D)^\alpha$, so for
  all $p<i\leq \Lambda$ we also have $0\leq s_i+m_1\leq (2D)^\alpha$.
  This implies that for $i>p$ we have $|\delta_i|\leq (2C+1)D$.
  It also implies that, when $m_1 \ge 0$, we have $|g^{m_1}| < C|m_1|^{1/\alpha} \le 2D$.

  Define a 2--cell with boundary consisting of a geodesic from
  $g^{s_{p}}$ to $g^{s_{p+1}}$, $\delta_{p+1}$, a geodesic from $h$ to
  $hg^{m_1+s_{p}}$, and $\bar\delta_{p}$.
  The bottom of this 2--cell
    has exponent bounded by $E$, so has length less than
    $CE^{1/\alpha}$.  When $m_1 < 0$, the top of the 2--cell similarly
    has length bounded by $CE^{1/\alpha}$.  Otherwise, the top has
    length $|g^{m_1}| < 2D$.
    The sides have length bounded by $D$
  and $(2C+1)D$, so the perimeter has length less than
  $(2C+4)D+2CE^{1/\alpha} <
  2(2C+1)D+2CE^{1/\alpha}=M_1D+M_2E^{1/\alpha}$.

  Finally, for $p<i<\Lambda$ define a 2--cell\footnote{ We are
    constrained by the hypothesis that the bottom of $\gamma$ uses
    geodesics along the given subdivision of the bottom side, so if we
    just flip the construction from Figure~\eqref{fillbigonA} we
    cannot use the estimate for the length of a geodesic along the
    long side of the last 2--cell in the mesh calculation. We add the
    extra 2--cells and accept more area to get finer mesh.} with
  boundary consisting of a geodesic from $g^{s_i}$ to $g^{s_i+1}$,
  $\delta_{i+1}$, and $\bar\delta_i$.  It has length less than
  $CE^{1/\alpha}+2(2C+1)D<M_1D+M_2E^{1/\alpha}$.

  Again, glue the 2--cells together along common $\delta_i$ to get the
  desired diagram.
\end{proof}

\begin{lemma}[Filling triangles]\label{filltriangle}
  There exist constants $M_0$, $M_1$, and $M_2$ such that for all
  positive $\Lambda$, $E$, and $D$, given an approximate triangle,
  corner paths of length at most $D$, and a subdivision of the
  $a$--side of the triangle into $\Lambda$-many subsegments, each of
  exponent at most $E$, the following hold.  There exist subdivisions
  of the $x$--side and $y$--side into $\Lambda$--many subsegments of
  exponent bounded by $1+E/L+D^\alpha$.  Furthermore, if we construct
  a loop $\gamma$ by replacing each of the subsegments by some
  geodesic then there is a diagram for $\gamma$ with mesh at most
  $M_0+M_1D+M_2E^{1/\alpha}$ and area at most
  $(\Lambda^2+9\Lambda+6)/2$.  In fact, it suffices to take $M_0:=4C$,
  $M_1:=6C+2$ and $M_2:=2C$.
\end{lemma}
\begin{proof}
  By Lemma~\ref{approximatetriangle} there is a true triangle
 with
  corresponding points at distance at most $2D+L$.
  That is, for each of the three sides, the corresponding pairs of
  sides from the approximate and true triangle form a $(2D+L)$--approximate bigon. 
  We construct a filling by filling the true triangle, using
  Lemma~\ref{fillbigon} to fill the three approximate bigons, and then adding one more 2--cell at each
  corner to close the corner path.

  Apply Lemma~\ref{fillbigon} between the $a$--side approximate
  bigon. 
  This gives a decomposition of the $a$--side of the true triangle
  into at most $\Lambda$--many subsegments of exponent at most
  $E+(2D)^\alpha$, and a diagram of area $\Lambda$ and mesh at
  most $2(2C+1)D+2CE^{1/\alpha}<M_1D+M_2E^{1/\alpha}$.

  It would be preferable, because we want to apply
  Lemma~\ref{projectionatox}, if the subsegments of the $a$--side all
  had exponents that were multiplies of $L$.  This can be achieved by
  moving the endpoints slightly along the side, each moving distance
  at most $L/2$.  For the purposes of the diagram, we add an extra
  strip of 2--cells along the previous strip, the bottom of the new
  strip is identified with the top of the old strip and the top of the
  new strip is the $a$--side of the triangle with its new subdivision.
  A vertical division into 2--cells can be obtained by connecting the
  original subdivision points with the new ones by geodesics.  The
  exponents of the new subdivision segments increase by at most $L$
  from the previous bound.  The new strip has area $\Lambda$ and mesh
  bounded by:
\begin{align*}
  C(E+(2D)^\alpha)^{1/\alpha} +& C(L+E+(2D)^\alpha)^{1/\alpha}+L/2+L/2\\
                               &\leq 2CE^{1/\alpha}+4CD+CL^{1/\alpha}+L\\
                               &=2CE^{1/\alpha}+4CD+2C+L\\
  &<M_0+M_1D+M_2E^{1/\alpha}
\end{align*}
   
Having made this adjustment, by Lemma~\ref{projectionatox}, every
subdivision point $p$ of the $a$--side of the triangle has a unique
closest point $q$ on the $x$--side.
Furthermore, from the proof of  Lemma~\ref{projectionatox}, every geodesic from $p$ to $q$ consists of a single
geodesic $y$--escape, so they all have the same trace: it is the
$y$--segment between $p$ and $q$.
A similar argument holds for $p$ and the closest point on the
$y$--side of the triangle. 

Subdivide the triangle by taking the trace of a geodesic from each
subdivision point of the $a$--side to its closest point on the
$x$--side and the $y$--side.
This divides the triangle into $(\Lambda^2-\Lambda)/2$--many
small diamonds and $\Lambda$--many small triangles.
See Figure~\ref{fig:filltriangle}. 
\begin{figure}[h]
  \centering
  \includegraphics[width=1.5in]{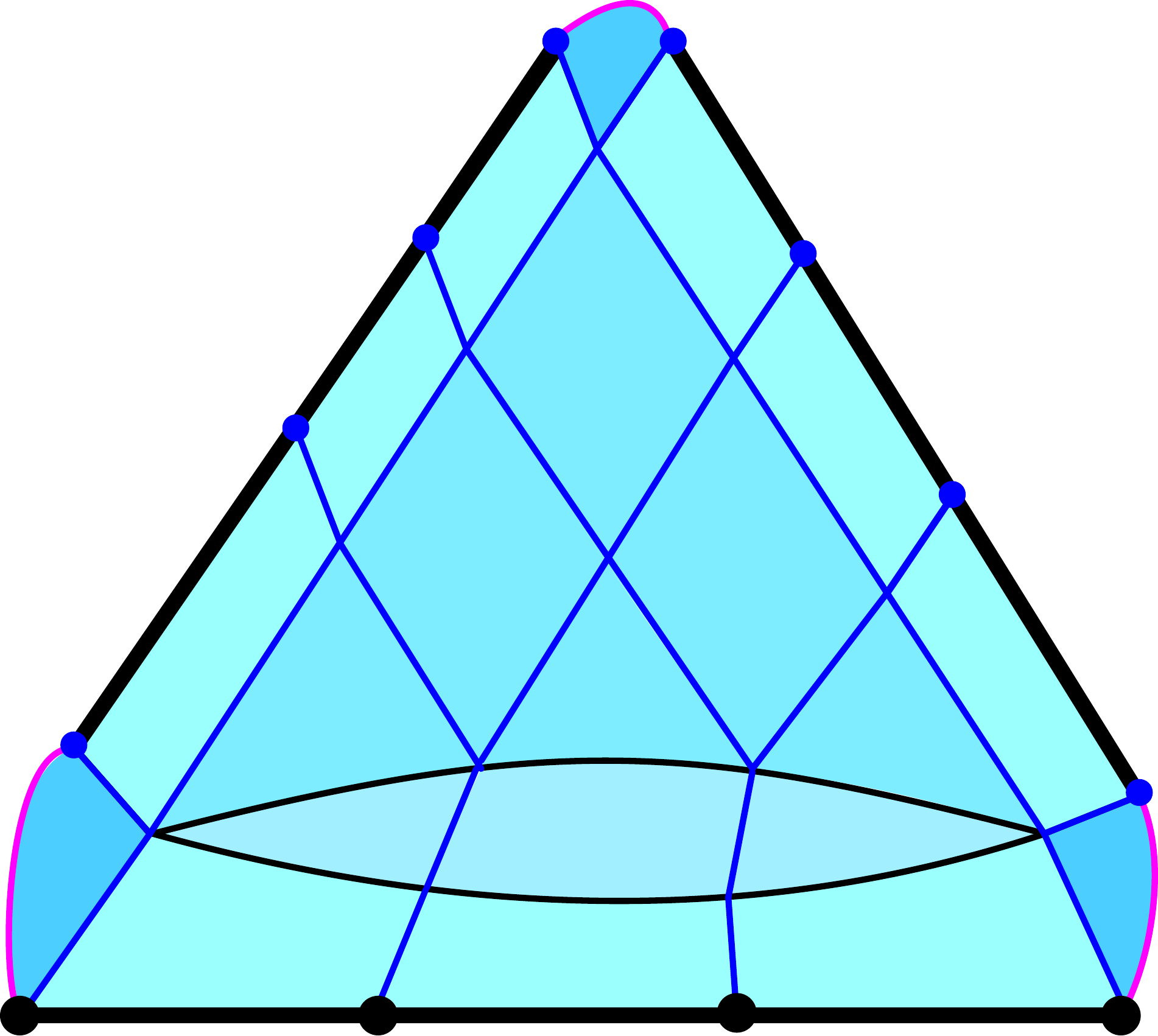}
  \caption{Filling an approximate triangle}
  \label{fig:filltriangle}
\end{figure}

The small triangles have as $a$--sides one of the subsegments of the
$a$--side of the triangle, so we know they have exponent at most
$L+E+(2D)^\alpha$.
From Lemma~\ref{projectionatox}, the resulting subdivisions of the
$x$ and $y$--sides have exponent that decreases by a factor of $L$, so
all $x$ and $y$--sides of all of the small triangles and small
diamonds have exponents bounded by $(L+E+(2D)^\alpha)/L=1+E/L +(2^\alpha/L)D^\alpha=1+E/L +D^\alpha$.

From this subdivision of the triangle define a diagram with one
2--cell for each of the small diamonds and small triangles, where the
1--skeleton of the diagram is obtained by replacing each side of each
of the small polygons by a geodesic with the same endpoints.
The mesh of this diagram is at most:
\begin{align*}
  4\cdot C((L+E+(2D)^\alpha)/L)^{1/\alpha}&= 4C\cdot
                                      L^{-1/\alpha}(L^{1/\alpha}+E^{1/\alpha}+2D)\\
                                    &=4C\cdot(1/2)(2+E^{1/\alpha}+2D)\\
  &<M_0+M_1D+M_2E^{1/\alpha}
\end{align*}

The next step is to use Lemma~\ref{fillbigon} to fill the approximate
bigons formed by the $x$--side of the approximate triangle and the
$x$--side of the true triangle, and similarly for the $y$--sides.
This results in subdivisions of the $x$--side and $y$--side of the approximate
triangle with $\Lambda$--many pieces of exponent at most
$(L+E+(2D)^\alpha)/L$ and two diagrams with area $\Lambda$ and
mesh at most:
\begin{align*}
  M_{1,\ref{fillbigon}}D&+M_{2,\ref{fillbigon}}\left((L+E+(2D)^\alpha)/L\right)^{1/\alpha}\\
                        &\leq
                          M_{1,\ref{fillbigon}}D+M_{2,\ref{fillbigon}}\left(1+E^{1/\alpha}/2+D\right)\\
                        &=M_{2,\ref{fillbigon}}+(M_{1,\ref{fillbigon}}+M_{2,\ref{fillbigon}})D+(M_{2,\ref{fillbigon}}/2)E^{1/\alpha}\\
                        &=2C+(6C+2)D+CE^{1/\alpha}\\
  &<M_0+M_1D+M_2E^{1/\alpha}
\end{align*}

To complete the diagram, at each corner add one 2--cell whose boundary
is the given path between the sides of the approximate triangle,
followed by a geodesic from one endpoint of one side of the
approximate triangle to the corner of the true triangle, followed by a
geodesic to the endpoint of the other side of the approximate
triangle.
The boundary of this 2--cell has length at most $D+2(2D+L)<M_0+M_1D$.

We have verified that the perimeter of every 2--cell of the diagram has length at
most $M_0+M_1D+M_2E^{1/\alpha}$.
The total area of our diagram is:
\[\Lambda+\Lambda+(\Lambda^2-\Lambda)/2+\Lambda + 2\Lambda+3=(\Lambda^2+9\Lambda+6)/2\qedhere\]
\end{proof}

\begin{lemma}[Filling diamonds]\label{filldiamond}
  There exist constants $M_0$, $M_1$, and $M_2$ such that for all
  positive $\Lambda$, $E$, and $D$, given an approximate diamond,
  corner paths of length at most $D$, and subdivisions of one
  $x$--side and one $y$--side into at most $\Lambda$-many subsegments
  of exponent at most $E$, the following hold.  There exist
  subdivisions of the other two sides into at most $\Lambda$--many
  subsegments of exponent at most $E+2LD^\alpha$.  Furthermore, if we
  construct a loop $\gamma$ by replacing each of the subsegments by
  some geodesic then there is a diagram for $\gamma$ with mesh at most
  $M_0+M_1D+M_2E^{1/\alpha}$ and area at most $\Lambda^2+4\Lambda+4$.
  In fact, it suffices to take $M_0:=3L$, $M_1:=8C+2$, and $M_2:=4C$.
\end{lemma}
\begin{proof}
  The strategy is the same as in the triangle case.
  By Lemma~\ref{approximatediamond} there is a true diamond whose
  corners are within distance $D+3L/2$ of the ends of the
  corresponding sides of the approximate diamond.
  Use Lemma~\ref{fillbigon} with the given subdivided sides of the
  approximate diamond to get subdivisions of two sides of the true
  diamond.
  Each of these consists of at most $\Lambda$--many segments of
  exponent at most $E+(2D)^\alpha$.
  In terms of diagrams we get two strips each with area at most
  $\Lambda$ and mesh at most $2(2C+1)D+2CE^{1/\alpha}<M_1D+M_2E^{1/\alpha}$.

By Lemma~\ref{intersectingxylinesprojectnicely} for each of the
subdivision points $p$ on the given $x$--side of the true diamond, there
is a unique closest point $q$ on the opposite $x$--side.
Furthermore, by Proposition~\ref{3piecegeodesic}, every geodesic
between them is a single geodesic $y$--escape, so all have the same
trace: it is the $y$--segment between $p$ and $q$.
A similar statement holds for the $y$--sides. 

Subdivide the diamond by taking the trace of a geodesic between each
subdivision point and its closest point on the opposite side of the diamond.
This divides the diamond into at most $\Lambda^2$--many small diamonds whose
sides have the same exponent bound as we started with, $E+(2D)^\alpha$.
See Figure~\ref{fig:filldiamond}.
    \begin{figure}[h]
      \centering
  \includegraphics[height=1.5in]{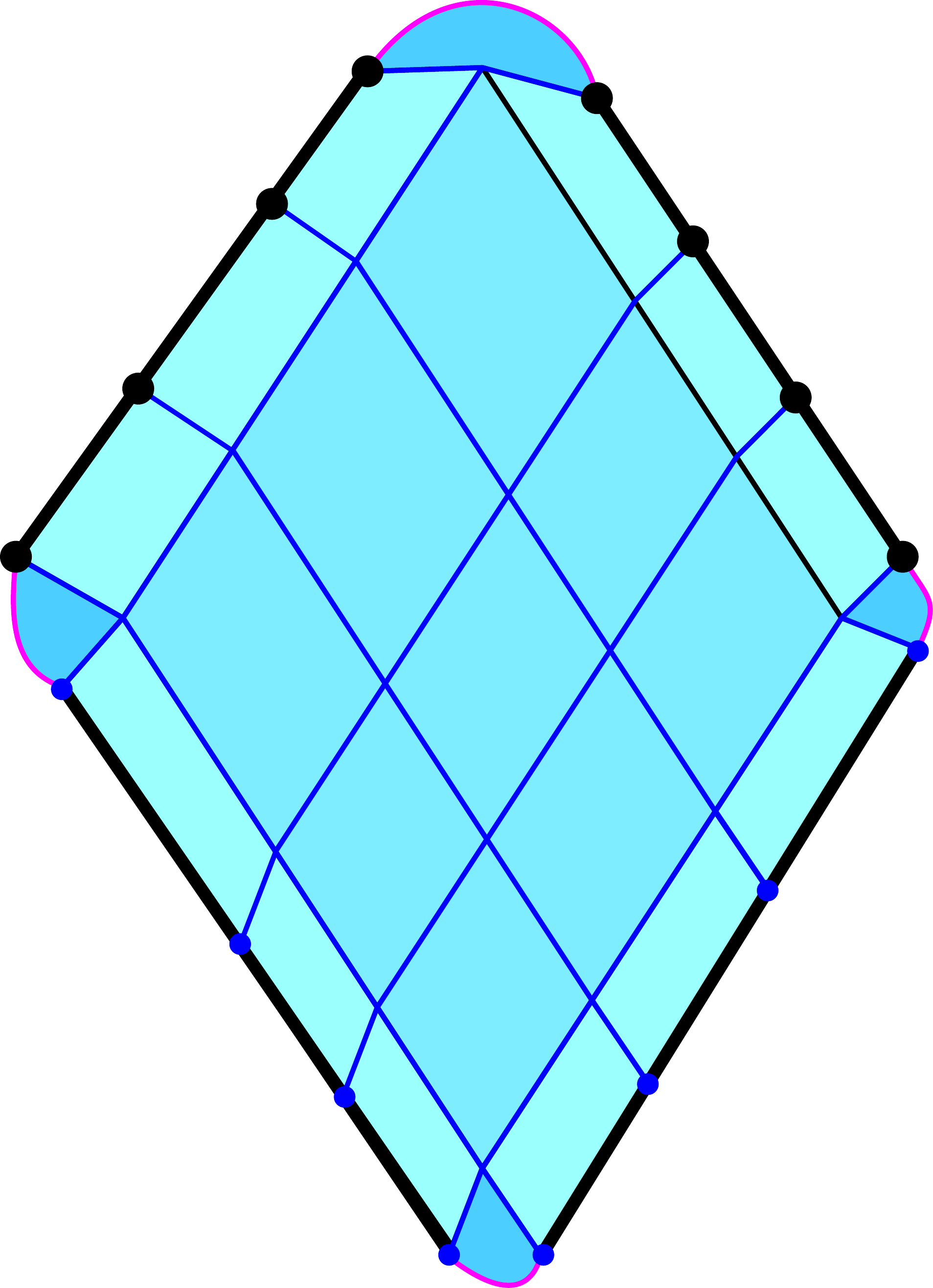}
  \caption{Filling an approximate diamond}
  \label{fig:filldiamond}
    \end{figure}

From the subdivision of the diamond define a diagram with one 2--cell
for each of the small diamonds, where the 1--skeleton is obtained by
replacing each side of each small diamond by a geodesic.
The mesh of this diagram is at most:
\[4\cdot C\left(E+(2D)^\alpha\right)^{1//\alpha}\leq
  4CE^{1/\alpha}+8D<M_1D+M_2E^{1/\alpha}\]

The next step is to use Lemma~\ref{fillbigon} to relate the
remaining $x$ and $y$--side of the true diamond with the corresponding
sides of the approximate diamond.
Each adds a subdiagram of area at most $\Lambda$ and mesh at most:
\[2(2C+1)D+2C(E+(2D)^\alpha)^{1/\alpha}\leq 2CE^{1/\alpha}+(8C+2)D<M_1D+M_2E^{1/\alpha}\]
The resulting subdivision of the sides of the approximate diamond have
exponent at most:
\[ E+(2D)^\alpha+(2D)^\alpha=E+2LD^\alpha\]

To complete the diagram, add a corner 2--cell at each of the four
corners are in the triangle case. The boundaries of these 2--cells have
length at most $D+2(D+3L/2)=3D+3L<M_0+M_1D$.

We have verified that the perimeter of every 2--cell of our diagram
has length bounded by $M_0+M_1D+M_2E^{1/\alpha}$.
Its area is $2\Lambda+\Lambda^2+2\Lambda+4$. 
\end{proof}

\subsection{Base case: identifying and filling a central region}\label{sec:base_case}
 \begin{lemma}\label{lem:subdividecentral}
   For all $\lambda\in (0,1)$ and all sufficiently large $\Lambda$
   depending on $\lambda$ and all sufficiently large $R$ depending on
   $\lambda$ and $\Lambda$ there exist $K>1$ depending on $R$, $A$
   depending on $\Lambda$, and
   $E:=\left(\frac{1}{L\Lambda}+\frac{2L}{R^\alpha}\right)$ and $J$
   depending on $R$ and $\Lambda$ such that for any $K$--biLipschitz
   loop $\gamma$ of length $|\gamma|> J$ we can construct a loop
   $\gamma'$ from $\gamma$ as follows:
   \begin{enumerate}
   \item    Identify a coset $gH$ such that every escape of $\gamma$ from $gH$
     has length at most $|\gamma|/2$.
     \item Select a manageable collection of long escapes of $\gamma$
       from $gH$.
       \item Subdivide the trace of each long escape
   into at most $2(\Lambda+1)$--many segments of exponent at most
   $E|\gamma|^\alpha$.
   \item Define $\gamma'$ to be the loop obtained by replacing each long
   escape of $\gamma$  by a concatenation of geodesics
   connecting the endpoint of these segments.
   \end{enumerate}   
   
  The loop $\gamma'$ can be filled by a diagram of area at most $A$
  and mesh at most $\lambda|\gamma|+4C$.
 \end{lemma}
 
 \begin{proof}
   Define:
\begin{align*}
  M_0&:=\max\{M_{0,\ref{fillbigon}},
       M_{0,\ref{filltriangle}}, M_{0,\ref{filldiamond}}\}=4C\\
  M_1&:=\max\{M_{1,\ref{fillbigon}},
       M_{1,\ref{filltriangle}}, M_{1,\ref{filldiamond}}\}\\
  M_2&:=\max\{M_{2,\ref{fillbigon}},
M_{2,\ref{filltriangle}}, M_{2,\ref{filldiamond}}\}/2
\end{align*}
Assume that $\Lambda>(2M_2/\lambda)^\alpha$ and that
$R>\max\{2M_1/\lambda,
\bigl(\frac{L^2\Lambda}{L-1}\bigr)^{1/\alpha}\}$.  Let $J$ and $K$ be
as in Corollary~\ref{cor:very_manageable} with respect to $R$, and
further increase $J$ if necessary so that
$J\geq
\left(\frac{L-1}{L^2\Lambda}-\frac{1}{R^\alpha}\right)^{-1/\alpha}$.
 
Assume $|\gamma|>J$.
As described in Step 2 of Section~\ref{sec:generalcase}, any HNN
diagram for $\gamma$ has a central region.
The boundary of the central region maps to a loop in some coset $gH$
of the vertex group. Up to isometry, we may assume $g=1$.
By definition, the components of $\gamma$ complementary to the central
region have length at most $|\gamma|/2$.
These components are compound escapes of $\gamma$ from $H$, and they
contain the simple escapes from $H$ as subpaths, so all of the escapes of
$\gamma$ from $H$ have length at most $|\gamma|/2$.
Corollary~\ref{cor:very_manageable} gives us a
manageable collection of  2, 3, or 4 $R$--long
     escapes from $H$.
     Call the subsegments of $\gamma$ between these long escapes the
     ``junk''.
     The total length of the junk
     is strictly less than $|\gamma|/R$.
     
    Together the traces of the long escapes form an approximate 2,
    3, or 4--gon, and the junk segments give corner paths of
    length less than $|\gamma|/R$.
    Furthermore, since each escape from the central region has length
    at most $|\gamma|/2$, its trace has exponent at most
    $(|\gamma|/2)^\alpha$.
    We will apply Lemmas~\ref{fillbigon},
\ref{filltriangle}, and \ref{filldiamond}, in all cases with 
$D_{8.*}=|\gamma|/R$, $E_{8.*}=(|\gamma|/2)^\alpha/\Lambda$, and $\Lambda_{8.*}=\Lambda+1$.
Then all
three lemmas produce a diagram with:
\begin{equation}
  \label{eq:14}
  \parbox{.85\textwidth}{$\text{mesh}\leq M_0+M_1
  |\gamma|/R+M_2 |\gamma|/\Lambda^{1/\alpha}$

  $\text{area}\leq (\Lambda+1)^2+4(\Lambda+1)+4$}
\end{equation}

To construct subdivisions and fillings there are six cases to consider, according to the configuration of
 traces of the escapes of the manageable collection as in Figure~\ref{fig:central_region_cases}.
   
     {\bf Case~\eqref{fig:central_region:two_escapes}:}
Approximate bigon:
Choose a subdivision of one side into at
most $(\Lambda+1)$--many segments of exponent at most $(|\gamma|/2)^\alpha/\Lambda$.
Apply Lemma~\ref{fillbigon}.
This yields a subdivision of the opposite side into at most
$(\Lambda+1)$--many segments of exponent at most
$|\gamma|^\alpha\left(\frac{1}{L\Lambda}+\frac{L}{R^\alpha}\right)$ and a filling
diagram  with mesh and area bounded by \eqref{eq:14}.

{\bf Case~\eqref{fig:central_region:three_escapes:fat}:}
Approximate triangle:
Choose a subdivision of the $a$--side into at most $(\Lambda+1)$--many
subsegments of exponent at most $(|\gamma|/2)^\alpha/\Lambda$.
Apply Lemma~\ref{filltriangle}.
This yields subdivisions of the other two sides into at most
$(\Lambda+1)$--many segments of exponent at most $1+|\gamma|^\alpha\left(\frac{1}{L^2\Lambda}+ \frac{1}{R^\alpha}\right)$
and a filling diagram  with mesh and area bounded by \eqref{eq:14}.

{\bf Case~\eqref{fig:central_region:four_escapes:rectangle}:}
Approximate diamond:
Choose subdivisions of one $x$--side and one $y$--side into at most
$\Lambda+1$ many segments of exponent at most $(|\gamma|/2)^\alpha/\Lambda$.
Apply Lemma~\ref{filldiamond}. 
This yields subdivisions of the other two sides into
$(\Lambda+1)$--many segments of exponent at most 
$|\gamma|^\alpha\left(\frac{1}{L\Lambda}+\frac{2L}{R^\alpha}\right)$ and a filling diagram  with mesh and area bounded by \eqref{eq:14}.

{\bf Case~\eqref {fig:central_region:three_escapes:slim}:}
The traces of the manageable escapes can be described as if they were an approximate triangle by
elements $g_0$, $g_1$, $g_2$ and numbers $m_0$, $m_1$, $m_2$, with the
exception that instead of the sides alternating $x$, $y$, $a$, they
are all of the same flavor $g\in\{a,x,y\}$.
Let us assume the $g_2$--side is the long side.
Let $\delta_i$ be the junk segment from $g_{i-1}g^{m_{i-1}}$ to $g_i$.
Then $g^{m_0}\delta_0$ is a path from $g_2g^{m_2+m_0}$ to
$g_0g^{m_0}$.

We treat this configuration as a pair of approximate bigons.
One of these has side $\delta_0$, top the $g_0$--side, side
$g^{m_0}\bar\delta_0$, and bottom the part of the $g_2$--side from
$g_2g^{m_2+m_0}$ to $g_2g^{m_2}$.
The two sides have length $|\delta_0|<|\gamma|/R$.
The second bigon has side $g^{m_0}\delta_0+\delta_1$, top the
$g_1$--side, side $\bar\delta_2$, and bottom the part of the
$g_2$--side from $g_2$ to $g_2g^{m_2+m_0}$.
One side has length $|\delta_2|<|\gamma|/R$.
The other has length $|\delta_0|+|\delta_1|$, which, since our junk bound
is actually on the total length of the junk and not the individual
segments, is still less than $|\gamma|/R$.

Choose a subdivision of the $g_0$ and $g_1$ sides into at most
$(\Lambda+1)$--many subsegments of exponent at most
$(|\gamma|/2)^\alpha/\Lambda$.
Apply Lemma~\ref{fillbigon} to each of the two approximate bigons.
The constants involved are exactly the same as in
Case~\eqref{fig:central_region:two_escapes}, except that the bounds
for the area and
number of induced subsegments of the  $g_2$--side are doubled. 

{\bf Case~\eqref {fig:central_region:four_escapes:a_embed}:}
Use the same trick as in
Case~\eqref{fig:central_region:three_escapes:slim}: use the group
action to translate the $y$--side and its attached junk segments to
the opposite end of the shorter $x$--side, and view this configuration
as a combination of an approximate diamond and an approximate
triangle, and fill them separately.

{\bf Case~\eqref{fig:central_region:four_escapes:a_cross}:}
In this case think of the configuration as two approximate triangles
meeting at the point where the $a$--side and $y$--side intersect. 
Fill them each separately and combine the fillings.

\medskip

In all cases we produce diagrams with mesh bounded by $M_0+M_1
  |\gamma|/R+M_2 |\gamma|/\Lambda^{1/\alpha}$, so since $M_0=4C$, 
$R>2M_1/\lambda$, and $\Lambda>(2M_2/\lambda)^\alpha$, we have mesh bounded by $\lambda|\gamma|+4C$.

In all cases we fill at most two approximate polygons, each with a
diagram of area at most $ (\Lambda+1)^2+4(\Lambda+1)+4$, so it
suffices to take $A$ to be double this bound.

Scanning the six cases for exponent bounds, we have that the
hypotheses $R>\bigl(\frac{L^2\Lambda}{L-1}\bigr)^{1/\alpha}$ and
$|\gamma| > J\geq
\left(\frac{L-1}{L^2\Lambda}-\frac{1}{R^\alpha}\right)^{-1/\alpha}$
imply that the largest of the bounds for the exponents of the
subdivision segments is at most
$|\gamma|^\alpha\left(\frac{1}{L\Lambda}+\frac{2L}{R^\alpha}\right)$,
so it suffices to take
$E:=\left(\frac{1}{L\Lambda}+\frac{2L}{R^\alpha}\right)$.
\end{proof}

\subsection{Inductive steps}\label{sec:inductive_steps}
\subsubsection*{Inductive step:  capping off small enough
  parts}\label{capoffverysmallregion}
Recall  \eqref{small_enough}.
There $\gamma$ had a central escape, from some
  coset $gH$, whose trace in $gH$ had been subdivided into
at most $2(\Lambda+1)$--many subsegments of exponent at most
$L\left(\frac{|\gamma|}{R}\right)^\alpha$, which were then replaced by geodesics
$\delta_1$,\dots,$\delta_m$.
Then if $\gamma|_P$ was the part of $\gamma$ complementary to the
central escape, we argued that the length of
$\gamma|_P+\bar\delta_m+\cdots+\bar\delta_1$ was small enough to cap
off with a single 2--cell satisfying the desired mesh bound.

Here we argue additionally that very small parts can be capped off.
Recall from Section~\ref{sec:noncentralregions} that the non-central
part $\gamma|_P$ of $\gamma$ at some coset of $H$ is called ``very small'' when its
length is less than $J$, where $J$ is some pre-specified constant,
independent of $\gamma$.
The setup is similar to the case described following
\eqref{small_enough}: the central escape of $\gamma$ from the coset has
trace subdivided into at most $m = 2(\Lambda+1)$--many
subsegments, and there are given geodesics $\delta_i$ with the same
endpoints as the subsegments.

  \begin{equation}
    \label{eq:15}
    \fbox{\text{Assume $M\geq(C(2(\Lambda+1))^{1-1/\alpha}+1)J$.}}
  \end{equation}

Since $\gamma|_P$ has the same endpoints as the central escape and
length less than $J$, by Theorem~\ref{powerdistortion} the
exponent of the central escape is less than $J^\alpha$.
Suppose that the exponent of $\delta_i$ is $e_i$.
Then $|\sum_i e_i|<J^\alpha$.
By Theorem~\ref{powerdistortion}, $|\delta_i|<C|e_i|^{1/\alpha}$.
  By reverse H\"older, $\sum_{i=1}^m|e_i|^{1/\alpha}\leq
  m^{1-1/\alpha}|\sum_{i=1}^me_i|^{1/\alpha}$.
  Therefore, the length of
  $\gamma|_P+\bar\delta_m+\cdots+\bar\delta_1$ is less than:
  \begin{equation}
    \label{very_small_region_mesh}
    (Cm^{1-1/\alpha}+1)J\leq(C(2(\Lambda+1))^{1-1/\alpha}+1)J\leq M
  \end{equation}

  Thus, we can cap off a very small part $\gamma|_P$ of $\gamma$ with a single 2--cell with
  boundary $\gamma|_P+\bar\delta_m+\cdots+\bar\delta_1$, whose length
  is at most $M$.

\subsubsection*{Inductive step: enfilades}
\begin{equation}
  \label{eq:inductionenfilades}
  \fbox{\text{Define $M_1:=2C+1+C\cdot2^{2+1/\alpha}$ and $M_2:=C$.}}
\end{equation}
\begin{equation}
  \label{eq:16}
  \fbox{\text{Assume $R\geq 2M_1/\lambda$ and $\Lambda\geq(2M_2/\lambda)^\alpha$.}}
\end{equation}

Suppose we have an escape $\gamma|_P$ in the $n$--th level, so that by
\eqref{induction_hypothesis} the central
side has been subdivided into at most $2(\Lambda+1)$--many
geodesic segments of exponent at most $|\gamma|^\alpha E(3/L)^{n}$,
where $E$ is the constant of Lemma~\ref{lem:subdividecentral}, which
depends on $R$ and $\Lambda$, but not on $\gamma$.
Suppose this exponent  bound exceeds that of \eqref{small_enough}, so that the escape is
not small enough to cap off. This means:
\begin{equation}
  \label{eq:enfiladenotsmall}
  |\gamma|^\alpha E(3/L)^{n}>L(|\gamma|/R)^\alpha\implies \frac{1}{R^\alpha}<\frac{3^nE}{L^{n+1}}
\end{equation}

\begin{equation}
  \label{eq:17}
  \fbox{\text{Assume $J$ and $K$ satisfy Proposition~\ref{enfilade_width} with
$R_{\ref{enfilade_width}}=3R$.}}
\end{equation}
Apply Proposition~\ref{enfilade_width} with
$R_{\ref{enfilade_width}}=3R$.  Then the distance between the
endpoints of the escape and the corresponding endpoints of the end of
the enfilade is at most
$3|\gamma|_P|/R_{\ref{enfilade_width}}=|\gamma|_P|/R\leq |\gamma|/2R$.
We claim that we can fill the enfilade with the same argument as
Lemma~\ref{fillbigon}, treating the enfilade as a
$(|\gamma|/(2R))$--approximate bigon.  Specifically, the point in that
proof where we used that the two sides of the bigon were in the same
coset of $\langle a,x,y\rangle$ was to say that we could translate
$\delta_0$ along the bottom side and obtain ``parallel'' copies, in
the sense that $g^a.\delta_0$ and $g^b.\delta_0$ are both paths from
the bottom side to the top side and their endpoint at the bottom and
top both differ by $g$ raised to the power $b-a$.  By
Lemma~\ref{enfilade_group_action}, exactly the same thing is true for
enfilades, except that we possibly have different $g$'s from the set
$\{a,x,y\}$ on the two sides.  As our length versus exponent estimates
are uniform for $\{a,x,y\}$, this does not change the rest of the
argument.

Apply that argument with $\Lambda_{\ref{fillbigon}}:=2(\Lambda+1)$,
$E_{\ref{fillbigon}}:=|\gamma|^\alpha E(3/L)^{n} $, and
$D_{\ref{fillbigon}}:=|\gamma|/(2R)$, with $M_{i,\ref{fillbigon}}$
denoting the constant $M_i$ in the conclusion of Lemma~\ref{fillbigon}.
This yields a subdivision of the outgoing side into at most
$2(\Lambda+1)$--many subsegments of exponent at most
$E_{\ref{fillbigon}}+LD_{\ref{fillbigon}}^\alpha$ and a filling
diagram of area  $\Lambda_{\ref{fillbigon}}$ and mesh at most:
\begin{align*}
  M_{0,\ref{fillbigon}}&+M_{1,\ref{fillbigon}}D_{\ref{fillbigon}}+M_{2,\ref{fillbigon}}E_{\ref{fillbigon}}^{1/\alpha}\\
  &=0+2(2C+1)D_{\ref{fillbigon}}+2CE_{\ref{fillbigon}}^{1/\alpha}\\
  &=(2C+1)|\gamma|/R+2C|\gamma|(E(3/L)^n)^{1/\alpha}\\
 &=(2C+1)|\gamma|/R+2CE^{1/\alpha}\left(3/L\right)^{n/\alpha}\cdot|\gamma|\\
                            &=(2C+1)|\gamma|/R+2C\left(\frac{1}{L\Lambda}+\frac{2L}{R^\alpha}\right)^{1/\alpha}\cdot\left(3/L\right)^{n/\alpha}\cdot|\gamma|&\text{definition of $E$}\\
  &<(2C+1)|\gamma|/R+2C\left(\frac{1}{L\Lambda}+\frac{2L}{R^\alpha}\right)^{1/\alpha}\cdot|\gamma|&\text{because
  $L>3$}\\
  &\leq
    (2C+1)|\gamma|/R+2C\left(\frac{1}{2\Lambda^{1/\alpha}}+\frac{2^{1+1/\alpha}}{R}\right)\cdot|\gamma|&\text{reverse
  H\"older}\\
                            &=(2C+1+C\cdot 2^{2+1/\alpha})|\gamma|/R+C|\gamma|/\Lambda^{1/\alpha}\\
  &=M_1|\gamma|/R+M_2|\gamma|/\Lambda^{1/\alpha} &\text{definition of $M_i$}
\end{align*}
Since $R\geq 2M_1/\lambda$ and $\Lambda\geq (2M_2/\lambda)^\alpha$, the mesh is at most $\lambda|\gamma|$.

Consider the ratio of outgoing to incoming exponent bounds:
\begin{align*}
  \frac{E_{\ref{fillbigon}}+LD_{\ref{fillbigon}}^\alpha}{E_{\ref{fillbigon}}}&=1+\frac{L(|\gamma|/(2R))^\alpha}{|\gamma|^\alpha E(3/L)^n}\\
  &=1+\frac{L^{n+1}}{2^\alpha3^nER^\alpha}\\
  &=1+\frac{L^{n}}{3^nER^\alpha}\\
           &<1+\frac{L^{n}}{3^nE}\cdot\frac{3^nE}{L^{n+1}}\qquad\text{by
             \eqref{eq:enfiladenotsmall}}\\
  &\leq 1+\frac{1}{L}= \frac{L+1}{L}
\end{align*}
Recall that for the induction we are aiming for an exponent ratio of
$3/L$, and here we have only shown for enfilades that
it is not too much larger than 1. The branching parts compensate.

\subsubsection*{Inductive step: branching
  parts}
\begin{equation}
  \label{eq:inductionbranchingregions}
  \fbox{\text{Let $M_0:=4C$, $M_1:=C\cdot
2^{3+1/\alpha}+6C+2$, and let $M_2:=2C$.}}
\end{equation}

Assume \eqref{eq:16} and \eqref{eq:17} are satisfied, and furthermore:
\begin{equation}
  \label{eq:18}
  \fbox{\text{Assume
$R\geq 2M_1/\lambda$ and $\Lambda\geq(2M_2/\lambda)^\alpha$ and $J\geq
2R$ and $M\geq M_0$.}}
\end{equation}
Let $E$ be the constant of Lemma~\ref{lem:subdividecentral}, which
depends on $R$ and $\Lambda$.

Suppose now that we have reached a branching part of $\gamma$ at the end of an
enfilade as in the previous subsection.
Then by the induction hypothesis \eqref{induction_hypothesis} and the
result of the previous subsection, the trace of the central escape is subdivided into at most
$2(\Lambda+1)$--many subsegments of exponent at most
$\frac{L+1}{L}|\gamma|^\alpha E(3/L)^n$.
Again, we assume that the branching part is not small enough to
cap off yet, so that \eqref{small_enough} fails and we have:
\begin{equation}
  \label{eq:branchingregionnotsmall}
  \frac{L+1}{L}|\gamma|^\alpha E(3/L)^{n}>L(|\gamma|/R)^\alpha\implies \frac{1}{R^\alpha}<\frac{3^nE(L+1)}{L^{n+2}}
\end{equation}

The traces of the central and two long non-central escapes form an
approximate triangle, so we apply Lemma~\ref{filltriangle} with
$\Lambda_{\ref{filltriangle}}:=2(\Lambda+1)$,
$E_{\ref{filltriangle}}:=\frac{L+1}{L}|\gamma|^\alpha E(3/L)^{n} $, and
$D_{\ref{filltriangle}}:=|\gamma|/R$.
This yields subdivisions of the other two sides into at most $2(\Lambda+1)$--many subsegments of
exponent at most
$1+E_{\ref{filltriangle}}/L+D_{\ref{filltriangle}}^\alpha$ and a
filling diagram of area quadratic in $\Lambda$ and mesh at most:
\begin{align*}
  M_{2,\ref{filltriangle}}&E_{\ref{filltriangle}}^{1/\alpha}+M_{1,\ref{filltriangle}}D_{\ref{filltriangle}}+M_{0,\ref{filltriangle}}\\
  &=2CE_{\ref{filltriangle}}^{1/\alpha}+(6C+2)D_{\ref{filltriangle}}+4C\\
  &=2C\left(\frac{L+1}{L}|\gamma|^\alpha
                                                                       E(3/L)^n\right)^{1/\alpha}+(6C+2)|\gamma|/R+4C\\
  &<4C|\gamma|
                                                                       E^{1/\alpha}+(6C+2)|\gamma|/R+4C&\text{because
                                                                                                                $L\geq
                                                                                                         6>3$}\\
  &=4C|\gamma|
                                                                       \cdot\left(\frac{1}{L\Lambda}+\frac{2L}{R^\alpha}\right)^{1/\alpha}+(6C+2)|\gamma|/R+4C&\text{definition
                                                                                                         of
                                                                                                                                                                $E$}\\
  &\leq 4C|\gamma|
                                                                       \cdot\left(\frac{1}{2\Lambda^{1/\alpha}}+\frac{2^{1+1/\alpha}}{R}\right)+(6C+2)|\gamma|/R+4C&\text{reverse
                                                                                                                                                                     H\"older}\\
  &=(C\cdot
    2^{3+1/\alpha}+6C+2)|\gamma|/R+2C|\gamma|/\Lambda^{1/\alpha}+4C\\
  &=M_0+M_1|\gamma|/R+M_2|\gamma|/\Lambda^{1/\alpha}&\text{definition of $M_i$}
\end{align*}
Since $M\geq M_0$, $R\geq 2M_1/\lambda$, and
$\Lambda\geq(2M_2/\lambda)^\alpha$, the mesh is less than $\lambda|\gamma|+M$.

Since  $|\gamma|>J\geq 2R$, the ratio of outgoing to incoming
  exponents is at most:
\begin{align*}
  \frac{1+E_{\ref{filltriangle}}/L+D_{\ref{filltriangle}}^\alpha}{E_{\ref{filltriangle}}}&=\frac{1}{L}+\frac{L}{L+1}\cdot\frac{L^n}{3^nE|\gamma|^\alpha}+\frac{L}{L+1}\cdot\frac{L^n}{3^nER^\alpha}\\
  &=
    \frac{1}{L}+\frac{L}{L+1}\cdot 
    \frac{L^n}{3^nE}\left(\frac{R^\alpha}{|\gamma|^\alpha}+1\right)\frac{1}{R^\alpha}\\
    &\leq
    \frac{1}{L}+\frac{L}{L+1}\cdot 
    \frac{L^n}{3^nE}\cdot\frac{5}{4}\cdot\frac{1}{R^\alpha}\qquad\text{
  Since $\alpha>2$ and $|\gamma|\geq 2R$.}\\
  &\leq
    \frac{1}{L}+\frac{L}{L+1}\cdot 
    \frac{L^n}{3^nE}\cdot\frac{5}{4}\cdot \frac{3^nE(L+1)}{L^{n+2}}\qquad\text{by
    \eqref{eq:branchingregionnotsmall}}\\
  &=\frac{9}{4L}
\end{align*}
This means that if we take the ratio of the exponent bound when
entering the enfilade to the bound when exiting through one of the
long escapes of the branching part at the end of the enfilade, then
we get $\frac{L+1}{L}\cdot\frac{9}{4L}$.
Since we have $L>3$, it follows that 
$\frac{L+1}{L}\cdot\frac{9}{4L}<\frac{3}{L}$, so this confirms \eqref{induction_hypothesis}.

\section{Proof of the main theorem and closing
  remarks.}\label{sec:closing}
\begin{thm}\label{maintheorem}
  For all even $L\geq6$, every asymptotic cone of $G_L$ is simply connected.
\end{thm}
\begin{proof}
  For $i\in\{1,2\}$ let
  $M_i:=\max\{M_{i,{\ref{lem:subdividecentral}}},
  M_{i,{\ref{eq:inductionenfilades}}},
  M_{i,{\ref{eq:inductionbranchingregions}}}\}$.
  These depend only on $C$.
  
  Pick any $\lambda\in (0,1)$.  Pick any
  $\Lambda>(2M_2/\lambda)^\alpha$.  Pick any
  $R>\max\bigl\{3,2M_1\lambda, 12C(\Lambda+1)/\lambda,
  \bigl(\frac{L^2\Lambda}{L-1}\bigr)^{1/\alpha}\bigr\}$.  Pick $J$
  large enough and $1<K<2$ close enough to 1 to satisfy
  Corollary~\ref{cor:very_manageable} and
  Lemma~\ref{lem:subdividecentral} with respect to $R$ and
  Proposition~\ref{enfilade_width} with respect to $3R$.  Assume
  further that $J\geq 2R$.  Define
  $M=M_0:=\max\{4C,(C(2(\Lambda+1))^{1-1/\alpha}+1)J\}$.

These choices simultaneously satisfy Lemma~\ref{lem:subdividecentral}
and the assumptions \eqref{eq:15}, \eqref{eq:inductionenfilades}, \eqref{eq:16}, \eqref{eq:17},
\eqref{eq:inductionbranchingregions}, and \eqref{eq:18}.
  
Let $\mathcal{C}$ be the area bound given by
Lemma~\ref{lem:subdividecentral}, which is a quadratic function of $\Lambda$.
Let $\mathcal{E}=\Lambda_{\ref{fillbigon}}=2(\Lambda+1)$ be the area
bound for enfilades.
Let $\mathcal{B}$ be the area bound for branching parts, which is quadratic in $\Lambda$.
Let $A$ be the area bound of \eqref{eq:13} with respect to
$\mathcal{C}$, $\mathcal{E}$, and $\mathcal{B}$.

If $\gamma$ is a $K$--biLipschitz loop of length at most $J$ we
fill it by a single 2--cell with mesh $J<M$.

If $\gamma$ is a $K$--biLipschitz loop of length greater than $J$ then
we apply the inductive argument of Section~\ref{sec:constructfillings}
to get a filling diagram of $\gamma$ with area at most $A$ and mesh at
most the maximum of the meshes for the central part, very small
parts, enfilades, and branching parts.
We have chosen $M_0$, $M_1$, $M_2$, $R$, and $\Lambda$ so that this is
at most $\lambda|\gamma|+M$.

We conclude that for this choice of $K>1$, every $K$--biLipschitz loop $\gamma$ can be filled by a diagram of area
at most $A$ and mesh at most $\lambda|\gamma|+M$, with $A$, $\lambda$,
and $M$ independent of $\gamma$.
Now apply Corollary~\ref{cor:reduction} and
Theorem~\ref{rileycriteria} to conclude that every asymptotic cone of
$G_L$ is simply connected.
\end{proof}
We close with a list of open questions:
\begin{enumerate}
\item \label{q_gl_ss} Is $G_L$ a strongly shortcut group?
\item \label{q_acss_ss} If every asymptotic cone of a group is simply
  connected, is the group strongly shortcut?
\item \label{q_ss_gs_inv} Is the strong shortcut property for Cayley
  graphs of groups invariant under change of generating set?
\item Is the strong shortcut property for groups invariant under
  quasi-isometry?
\item Can a group of polynomial growth have an asymptotic cone
  containing an isometrically embedded circle?
\item Can a group with quadratic Dehn function have an asymptotic cone
  containing an isometrically embedded circle?
\item Are all asymptotic cones of all of the Brady-Bridson snowflake
  groups $G_{p,q}$ of \eqref{bradybridsonsnowflake} simply connected,
  or is it important that we have taken $q=1$?
\end{enumerate}
Note that, by Theorem~\ref{maintheorem} and
Corollary~\ref{cor:notstrongshort},
\begin{itemize}
\item if (\ref{q_gl_ss}) has a negative answer then (\ref{q_acss_ss})
  has a negative answer and
\item if (\ref{q_gl_ss}) has a positive answer then
  (\ref{q_ss_gs_inv}) has a negative answer.
\end{itemize}
Thus (\ref{q_acss_ss}) and (\ref{q_ss_gs_inv}) cannot both have
positive answers.

\bibliographystyle{hypersshort}
\bibliography{refs}
\end{document}